\documentclass[12pt]{article}
\usepackage{amsmath,amssymb,amsthm,amscd,latexsym,}
 \font\goth=eusm10
\usepackage{graphicx}

\makeatletter
\renewcommand{\mod}[1]{\allowbreak \if@display \mkern 8mu \else
\mkern 5mu\fi {\operator@font mod}\,\,#1}
\makeatother

\newcommand{\bc}{\mathbb C}
\newcommand{\bn}{\mathbb N}
 \newcommand{\bq}{\mathbb Q}
\newcommand{\br}{\mathbb R}
\newcommand{\bz}{\mathbb Z}

\DeclareMathOperator{\rk}{rk}
\DeclareMathOperator{\ch}{ch}

\newtheorem{theorem}{Theorem}
\newtheorem{proposition}{Proposition}
\newtheorem{definition}{Definition}
\newtheorem{corollary}{Corollary}
\newtheorem{remark}{Remark}
\newtheorem{lemma}{Lemma}
\newtheorem{example}{Example}

\numberwithin{proposition}{section}
\numberwithin{definition}{section}
\numberwithin{corollary}{section}
\numberwithin{remark}{section}
\numberwithin{lemma}{section}
\numberwithin{equation}{section}
\numberwithin{theorem}{section}
\numberwithin{conjecture}{section}
\numberwithin{example}{section}

\newcommand\Hh{\mathcal H}
\newcommand\La{\mathcal L}

\newcommand\M{\mathcal M}
\newcommand\gggg{\mathfrak g}

\newcommand{\CC}{\mathbb C}
\newcommand{\ZZ}{\mathbb Z}
\newcommand{\cD}{\mathcal D}
\newcommand{\QQ}{\mathbb Q}
\newcommand{\supp}{\mathop{\null\mathrm {Supp}}\nolimits}
\newcommand{\Mapsto}{\mapstochar\longrightarrow}
\newcommand{\latt}[1]{{\langle{#1}\rangle}}
\newcommand{\PP}{\mathbb P}
\newcommand{\Orth}{\mathop{\null\mathrm {O}}\nolimits}
\newcommand{\id}{\mathop{\mathrm {id}}\nolimits}
\newcommand{\SO}{\mathop{\mathrm {SO}}\nolimits}

\newcommand{\emb}{\hookrightarrow}
\newcommand{\divv}{\mathop{\null\mathrm {div}}\nolimits}
\newcommand{\pr}{\mathop{\mathrm {pr}}\nolimits}
\newcommand{\Sum}{\sum\limits}
\newcommand{\cZ}{\mathcal Z}
\newcommand{\RR}{\mathbb R}
\newcommand{\cH}{\mathcal H}
\newcommand{\Sp}{\mathop{\mathrm {Sp}}\nolimits}
\newcommand{\gz}{\mathfrak z}
\newcommand{\HH}{\mathbb H}
\newcommand{\NN}{\mathbb N}
\newcommand{\SL}{\mathop{\mathrm {SL}}\nolimits}
\newcommand{\cB}{\mathcal B}
\newcommand{\Lift}{\mathop{\mathrm {Lift}}\nolimits}
\renewcommand{\Tilde}{\widetilde}

\begin{document}
\title{Lorentzian Kac-Moody algebras with\\ Weyl groups of 2-reflections}
\date{29 March 2016}
\author{Valery  Gritsenko\footnote{The first author was partially supported by a subsidy granted to the HSE by the  Government of the Russian Federation for the implementation of the Global Competitiveness Program and by Institut Univeristaire
de France (IUF).}\ \ \
and  Viacheslav V. Nikulin}
\maketitle

\begin{abstract}
We describe a new large  class of  Lorentzian Kac--Moody algebras.
For all ranks, we classify 2-reflective hyperbolic lattices $S$ with the
group of 2-reflections
of finite volume and with a lattice Weyl vector.  They define the corresponding
hyperbolic Kac--Moody algebras of restricted arithmetic type
which are graded by $S$. For most of them, we construct Lorentzian Kac--Moody algebras
which give their automorphic corrections: they are graded by the $S$, have the same
simple real roots, but their denominator identities are given by automorphic forms with
2-reflective divisors.
We give exact constructions of these automorphic forms as Borcherds products
and, in some cases, as additive Jacobi liftings.

\end{abstract}
\newpage

\section{Introduction} \label{introduction}
One of the  most known examples of Lorentzian Kac--Moody algebras
is the Fake Monster Lie algebra defined by R. Borcherds (see \cite{B3}--\cite{B4})
in his solution of  Moonshine Conjecture.
Lorentzian Kac--Moody (Lie, super) algebras are automorphic corrections
of hyperbolic Kac--Moody algebras.
In our papers \cite{Nik7}, \cite{Nik8}, \cite{Nik12}, \cite{GN1}--\cite{GN8},
we developed a general theory of Lorentzian Kac--Moody algebras
(see \cite{GN5} and \cite{GN6} for the most complete exposition)
based of the results by  Kac \cite{K1}--\cite{K3}
and Borcherds \cite{B1}--\cite{B4}.
In these our papers (especially see \cite{GN5} and \cite{GN6}),
we constructed and classified some of these algebras for the rank 3.

In this paper, we consruct and classify some of Lorentzian Kac--Moody algebras for
all ranks $\ge 3$. In our papers above and here, we mainly consider and classify
Lorentzian Kac--Moody algebras with Weyl groups $W$ of $2$-reflections. They
are groups generated by reflections in elements with square $2$ of
hyperbolic (that is of signature $(n,1)$) lattices
(that is integral symmetric bilinear forms) $S$ of $\rk S=n+1$.

For an (automorphic) Lorentzian Kac--Moody Lie algebra with a hyperbolic
{\it root} lattice $S$, the Weyl group $W$ must have the fundamental chamber $\M$
of finite ({\it elliptic case}) or almost finite ({\it parabolic case}) volume
in  the hyperbolic
(or Lobachevsky) space $\La(S)=V^+(S)/\br_{++}$ where $V^+(S)$ is a half of the
cone $V(S)\subset S\otimes \br$ of elements $x\in S\otimes \br$ with $x^2<0$.
For parabolic case,
there exists a point $c=\br_{++}r\in \M$, $r\in S$, $r\not=0$ and $r^2=0$, at
infinity of $\M$ such that $\M$ is finite at any
cone in $\La(S)$ with the vertex at $c$.

We denote by $P=P(\M)\subset S$ the set of {\it simple real roots}
or all elements of $S$ with
square $2$  which are perpendicular to faces of codimension one of $\M$
and directed outwards.
For a Lorentzian Kac--Moody algebra, $P=P(\M)$ must have
{\it the lattice Weyl vector} $\rho \in S\otimes \bq$ such that
$(\rho,\alpha)=-\alpha^2/2=-1$ for all $\alpha\in P=P(\M)$.
For elliptic case, $\rho^2=(\rho, \rho)<0$,
and $\rho^2=0$ for parabolic case where $\br_{++}\rho=c$.
For elliptic case, $W$ has finite index in $O(S)$,
then $S$ is called {\it elliptically $2$-reflective}.
For parabolic case, $O^+(S)/W$ is $\bz^m$, up to finite index,
for some $m>0$.
We want to construct Lorentzian Kac--Moody algebras with
the root lattice $S$, the set of simple real roots $P=P(\M)\subset S$
and the Weyl group $W$.

In this paper, we consider the basic case of this problem when
the Weyl group $W$ is the full group $W=W^{(2)}(S)$ generated by all reflections
in vectors with square $2$ of a hyperbolic even lattice $S$.

All elliptically $2$-reflective hyperbolic lattices $S$ when the
group  $W^{(2)}(S)$ has finite index in $O(S)$
were classified by the second author in \cite{Nik2} and \cite{Nik5}
for $\rk S\not=4$, and by E.B. Vinberg \cite{Vin5} for $\rk S=4$.
Their total  number is finite and $\rk S\le 19$.
The number of parabolically $2$ reflective hyperbolic lattices $S$
for $W=W^{(2)}(S)$ is also finite by \cite{Nik8}, but their full classification is
unknown. Many of them were found in $\cite{Nik2}$.
\smallskip

In Sect. \ref{sec:data1-3}, we give the list of elliptically $2$-reflective
even hyperbolic lattices $S$ from \cite{Nik2}, \cite{Nik5} and \cite{Vin4},
and in  Theorem \ref{th:el2refW}, we find those of them which have the lattice Weyl vector $\rho$
for $P=P(\M)$ of $W^{(2)}(S)$. {\bf There are 59 such lattices.}
$15$ of them are of rank $3$ and $44$ of rank $\ge 4$,
and the maximal rank is equal to $19$.
For all these lattices $S$, we calculate the set $P=P(\M)\subset S$ of
simple real roots and its Dynkin diagram which is equivalent to the
generalized Cartan matrix
\begin{equation}
A=\left((\alpha_1,\alpha_2)\right),\ \ \alpha_1,\alpha_2\in P=P(\M).
\label{genCartanA}
\end{equation}
This matrix defines the usual hyperbolic Kac--Moody algebra $\gggg(A)$,
see \cite{K1}. We calculate the lattice Weyl vector $\rho$ for $P=P(\M)$
for all these cases.

\smallskip

In Sect. \ref{pullback},
for an extended lattice $T=U(m)\oplus S$ of signature $(n+1,2)$
where $U$ is the even unimodular lattice of signature $(1,1)$,
$U(m)$ means that we multiply the pairing of the lattice $U$ by $m\in \bn$,
and $\oplus$ is the orthogonal sum of lattices, we consider the Hermitian
symmetric domain $\Omega(T)\cong S\otimes \br+iV^+(S)$.
For all $59$  lattices $S$ of Theorem 3, we conjecture existence for some $m$
of so called $2$-reflective holomorphic automorphic form $\Phi(z)\in M_k(\Gamma)$
on $\Omega(T)$ of weight $k>0$ with integral Fourier coefficients,
where  $\Gamma\subset O(T)$ is of finite index, whose divisor
is union of rational quadratic divisors with
multiplicity one orthogonal to the  elements with square $2$ of $T$.
{\it The Fourier coefficients of $\Phi(z)$ at a $0$-dimensional cusp define additional
sequence of simple imaginary roots ${P^\prime}^{im}\subset S$ with non-positive squares.}
The sequences  of the simple real roots $P$ and the imaginary simple roots
${P^\prime}^{im}$ define Lorentzian Kac--Moody--Borcherds Lie superalgebra
$\gggg(P(\M), \Phi)$  by exact generators and defining relations.
This superalgebra  is the {\it (automorphic)
Lorentzian Kac--Moody algebra} which we want to construct.
The Lorentzian Kac--Moody (Lie super) algebra $\gggg(P(\M), \Phi)$
is graded by $S$. {\it The dimensions
$\dim \gggg_\alpha(P(\M), \Phi)$,  $\alpha\in S$, of this grading
(equivalently, the multiplicities of all roots of the algebra) are defined
by the Borcherds product expansion of the automorphic form $\Phi(z)$ at a
zero dimensional cusp.}  See Sect. \ref{sec:data1-5} and
Sect. \ref{pullback} for
the exact definitions and details of the automorphic correction.

In this paper, we determine automorphic corrections  for $36$
of $59$ lattices of Theorem 3
but we consider here more than $70$ reflective modular forms.
We are planing to construct automorphic corrections for the rest $10$
of 2-reflective lattices of rank $4$ and $5$ from Theorem 3
in  a separate publication. Some of these functions
will be  modular with respect to congruence subgroups similar
to \cite{GN6}.
\smallskip

In Sect. \ref{sec:data1-5}, we give exact definitions of data (I) -- (V)
which define the Lorentzian Kac--Moody algebras for the case which we consider.
One can find more general definitions in our papers which we mentioned above.
\smallskip

In Sect. \ref{sec:data1-3}, we give classification of elliptically $2$-reflective
hyperbolic lattices $S$ with a lattice Weyl vector for $W^{(2)}(S)$. They
give all possible data (I) -- (III) for construction of the Lorentzian
Kac--Moody algebras which we consider.
\smallskip

In Sections  \ref{pullback}--\ref{seriesD8}, we find automorphic forms which finalize the construction
of the (automorphic) Lorentzian Kac--Moody algebra $\gggg(P(\M), \Phi)$
and  give automorphic corrections of the usual Kac--Moody algebra $\gggg(A)$
defined in \eqref{genCartanA}. We note that $\gggg(A)$  might have many
automorphic corrections! We give two such examples in Proposition \ref{Pr-2-forms}
and Theorem \ref{U(2)D4}.

In Sect. \ref{pullback}, we analyse the quasi pull-backs of the Borcherds modular form
$\Phi_{12}$ for $\Orth(II_{26,2},\det)$, construct $34$ strongly $2$-reflective
modular  forms which determine  the automorphic corrections
of $25$ hyperbolic lattices from Theorem 3 and of $9$ parabolically $2$-reflective hyperbolic lattices.
We note that the modular objects related to these Lorentzian Kac--Moody algebras
are very arithmetic. The $25$ modular  forms are cusp forms which
are new eigenfunctions of all Hecke operators (see  Corollary  \ref{Hecke1}).
{\it One can consider these cusp forms as  generalisations of the Ramanujan $\Delta$-function.}
All $34$ corresponding modular varieties of orthogonal type are, at least, uniruled
(see Corollary \ref{uniruled}).

In Sect. \ref{JBP},  we describe Borcherds products
of Jacobi type of the quasi pull-backs from Sect. \ref{pullback}.
 Our approach gives an explicite formula for the first
two Fourier-Jacobi coefficients of the reflective modular forms
and an interesting  relation between Lorentzian Kac--Moody algebras
and some affine Lie algebras in terms of the denominator functions.

In Sect. \ref{seriesD8}, we construct automorphic corrections of
fourteen hyperbolic root systems
of  Theorem \ref{th:el2refW} and four automorphic corrections of hyperbolic root systems
of parabolic type. Almost all modular forms of Sect. \ref{seriesD8} have simple Fourier expansions
because they are additive Jacobi liftings of Jacobi forms related to the dual lattices of some root lattices.
Some reflective modular forms from Sect. \ref{seriesD8}  determine automorphic corrections
of hyperbolic Kac--Moody algebras with  Weyl groups which
are overgroups or subgroups of the Weyl groups of type $W^{(2)}(S)$ considered in this paper.

We note that  the  denominator functions of the corresponding Lorentzian Kac--Moody algebras are automorphic discriminants of moduli spaces of some
$K3$ surfaces with a condition  on Picard lattices and they realise
the arithmetic mirror symmetry for such $K3$ surfaces
(see \cite{GN3}, \cite{GN7} and \cite{GN8}).

\section{Definition of Lorentzian Kac--Moody\\ algebras corresponding
to 2-reflective\\ hyperbolic lattices with a lattice\\ Weyl vector}
\label{sec:data1-5}

Here we want to give definition of Lorentzian Kac--Moody
algebras which we want to construct and consider in this paper.
They are given by data (I) --- (V) below.
We follow the general theory of
Lorentzian Kac--Moody algebras from our papers \cite{GN5}, \cite{GN6},
\cite{GN8} and \cite{Nik7}, \cite{Nik8}, \cite{Nik12} where we used ideas and results
by  Kac \cite{K1}--\cite{K3} and Borcherds \cite{B1}--\cite{B4}.
One can find more general definitions and possible data in these
our papers.

\medskip

(I) The datum (I) is given by a {\it hyperbolic lattice} $S$ of the
rank $\rk S\ge 3$.

We recall that a lattice (equivalently, a non-degenerate symmetric bilinear form
over $\bz$) $M$ means that $M$ is a free $\bz$-module $M$ of a finite
rank with symmetric $\bz$-bilinear non-degenerate pairing $(x,y)\in \bz$
for $x,y\in M$. A lattice $M$ is hyperbolic if the corresponding
symmetric bilinear form $M\otimes \br$ over $\br$ has signature
$(n,1)$ where $\rk M=n+1$.

\medskip

(II) This datum is given by the {\it Weyl group} which is the
2-reflection group $W=W^{(2)}(S)\subset O(S)$
of the hyperbolic lattice $S$ from (I). It is
generated by $2$-reflections $s_\alpha$ in all $2$-roots $\alpha \in S$
that is $\alpha^2=(\alpha,\alpha)=2$.

We recall that an element $\alpha$ of a lattice $M$ is called {\it root}
if $\alpha^2 >0$ and $\alpha^2\ | 2(\alpha,M)$ that is $\alpha^2|2(\alpha,x)$
for any $x\in M$. A root $\alpha\in M$ defines the reflection
\begin{equation}
s_\alpha: x\to x -(2(x,\alpha)/\alpha^2)\alpha,\ \ \forall x \in M
\label{refl}
\end{equation}
which belongs to the automorphism group $O(M)$ of the lattice $M$.
The reflection $s_\alpha$ is characterized by the properties:
$s_\alpha(\alpha)=-\alpha$ and $s_\alpha|(\alpha)^\perp_M$ is identity.
Any element $\alpha\in M$ with $\alpha^2=2$ gives a root of $M$.

\medskip

(III) This datum is given by the {\it set of simple real roots}
$P=P(\M)\subset S$ of all 2-roots which are perpendicular and
directed outwards to the fundamental chamber $\M\subset \La(S)$
of the Weyl group $W=W^{(2)}(S)$ acting in the hyperbolic space $\La(S)$
defined by $S$. The set $P=P(\M)$ must have the {\it lattice Weyl vector}
$\rho \in S\otimes \bq$ such that
\begin{equation}
(\rho,\alpha)=-1\ \ \forall \alpha\in P=P(\M).
\label{weylvec}
\end{equation}
The fundamental chamber $\M$ must have either a finite volume (then $S$ is
called {\it elliptically 2-reflective}) and
then $\rho^2<0$ and $P=P(\M)$ is finite ({\it elliptic case}), or
almost finite volume (then $S$ is called {\it parabolically 2-reflective})
and $\rho^2=0$, but $\rho\not=0$ ({\it parabolic
case}). Here almost finite volume means that $\M$ has finite volume
in any cone with the vertex $\br^{++}\rho$ at infinity of $\M$.

We recall that, for a hyperbolic
lattice $M$, we can consider the cone
$$
V(M)=\{x\in M\otimes \br\ |\ x^2<0\}
$$
of $M$, and its  half cone $V^+(M)$.
Any two elements $x,y\in V^+(M)$ satisfy $(x,y)<0$.
The half-cone $V^+(M)$ defines {\it the hyperbolic space of $M$,}
$$
\La^+ (M)=V^+(M)/\br_{++}=\{\br_{++}x\ |\ x\in V^+(M)\}
$$
of the curvature $(-1)$ with the hyperbolic distance
$$
\ch\rho(\br_{++}x, \br_{++}y)=\frac{-(x,y)}{\sqrt{x^2y^2}},\ \ x,y\in V^+(M).
$$
Here $\br_{++}$ is the set of all positive real numbers, and $\br_+$
is the set of all non-negative real numbers.
Any $\delta \in M\otimes \br$ with $\delta^2>0$ defines
{\it a half-space}
$$
\Hh^+_\delta=\{\br_{++}x\in \La^+(M)\ |\ (x, \delta)\le 0\}
$$
of $\La (M)$ bounded by the {\it hyperplane}
$$
\Hh_\delta=\{\br_{++}x\in \La^+(M)\ |\ (x, \delta)= 0\}\,.
$$
The $\delta$ is called orthogonal to the half-space
$\Hh^+_\delta$ and the hyperplane $\Hh_\delta$,
and it is defined uniquely if $\delta^2>0$ is
fixed.
For a root $\alpha\in M$, the reflection $s_\alpha$ gives the reflection
of $\La^+(M)$
with respect to the hyperplane $\Hh_\alpha$, that is $s_\alpha$ is
identity on $\Hh_\alpha$,
and $s_\delta(\Hh^+_\alpha)=\Hh^+_{-\alpha}$.
It is well-known that the group
$$
O^+(S)=\{\phi \in O(S)\ |\ \phi(V^+(S))=V^+(S) \}
$$
is discrete in $\La^+(S)$ and has a fundamental domain of finite
volume. The subgroup $W^{(2)}(S)$ is its subgroup generated by $2$-reflections.

The main invariant of the data (I) --- (III) is {\it the generalized Cartan
matrix}
\begin{equation}
A=\left(\frac{2(\alpha, \alpha^\prime)}{(\alpha,\alpha)}\right)=
\left((\alpha,\alpha^\prime)\right),\ \ \ \alpha,\alpha^\prime \in P=P(\M).
\label{genCartmatr}
\end{equation}
It is symmetric for the case we consider. It defines the corresponding
{\it hyperbolic Kac--Mody algebra $\gggg(A)$}, see \cite{K1}.
It has {\it restricted arithmetic type}
and is {\it graded by the lattice $S$.}
See \cite{Nik7} and \cite{Nik8} for details. The next data (IV) and (V) give
the {\it automorphic correction} $\gggg$  of this algebra.

\medskip

(IV) For this datum, we need an extended lattice
$T=U(m)\oplus S$ ({\it the symmetry lattice of the Lie algebra $\gggg$}) where
\begin{equation}
U=
\left(\begin{array}{rr}
 0 & -1  \\
-1 &  0
\end{array}\right),
\label{latU}
\end{equation}
$M(m)$ for a lattice $M$ and $m\in \bq$ means
that we multiply the pairing of $M$ by $m$,
the orthogonal sum of lattices is denoted by $\oplus$.
The lattice $T$ defines the Hermitian symmetric domain
of the type IV
\begin{equation}
\Omega (T)=\{\bc \omega\subset T\otimes \bc\ |\ (\omega, \omega)=0,\
(\omega,\overline{\omega})<0\}^{+}
\label{typeIVdom}
\end{equation}
where $+$ means a choice of one (from two) connected components.
The domain $\Omega(T)$ can be identified with
the complexified cone $\Omega(V^+(S))=S\otimes \br + iV^+(S)$ as follows:
for the basis $e_1,e_2$ of the lattice $U$ with the matrix \eqref{latU},
we identify $z\in \Omega(V^+(S))$ with $\bc\omega_z\in \Omega(T)$
where $\omega_z=(z,z)e_1/2+e_2/m\oplus z\in \Omega(T)^{\bullet}$
(the corresponding affine cone over $\Omega(T)$).
 The main datum in (IV) is a {\it holomorphic automorphic form}
$\Phi (z)$, $z\in \Omega(V^+(S))=\Omega(T)$ of some weight $k\in \bz/2$
on the Hermitian symmetric domain $\Omega(V^+(S))=\Omega (T)$
of the type IV with respect to a subgroup $G\subset O^+(T)$
of a finite index ({\it the symmetry group
of the Lie algebra $\gggg$}. Here $O^+(T)$ is the index two subgroup of
$O(T)$ which preserves $\Omega(T)$.

The automorphic form $\Phi(z)$ must have Fourier expension
which gives the denominator identity for the Lie algebra $\gggg$:
\begin{equation}
\Phi(z)=\sum_{w\in W}{\det(w)(exp(-2\pi i(w(\rho),z))}-
$$
$$
-\sum_{a\in S\cap \br_{++}\M}{m(a)exp(-2\pi i (w(\rho+a),z)))},
\label{denidlor1}
\end{equation}
where all coefficients $m(a)$ must be integral.
It also would be nice to calculate the
{\it infinite product expension (the Borcherds product) for the denominator identity of
the Lie algebra $\gggg$}
\begin{equation}
\Phi(z)=exp(-2\pi i (\rho,z))\prod_{\alpha\in \Delta_+}{(1-exp(-2\pi i (\alpha,z)))^{mult(\alpha)}}\ ,
\label{denidlor2}
\end{equation}
which gives {\it multiplicities $mult(\alpha)$}
of roots of the Lie algebra $\gggg$.
Here $\Delta_+\subset S$ (see below).

\medskip

(V) The automorphic form $\Phi(z)$ in $\Omega(V^+(S))=\Omega(T)$ must be
{\it 2-reflec\-tive.} It means that the divisor (of zeros) of
$\Phi(z)$ is union of rational quadratic divisors
which are orthogonal
to 2-roots of $T$. Hear, for $\beta\in T$ with $\beta^2>0$
{\it the rational quadratic divisor which
is orthogonal to $\beta$,} is equal to
$$
D_\beta=\{\bc \omega\in \Omega(T)\ |\ (\omega,\beta)=0\}.
$$
The property (V) is valid in a neighbourhood of the cusp of
$\Omega (T)$ where the infinite product \eqref{denidlor2} converges, but
we want to have it globally.

\medskip

For our case, the {\it Lorentzian Kac--Moody superalgebra
$\gggg$ corresponding to data (I) --- (V),}
which is a Kac--Moody--Borcherds superablebra or an
{\it automorphic correction} given by $\Phi(z)$ of
the Kac--Moody algebra $\gggg(A)$
given by the generalized Cartan matrix
\eqref{genCartmatr} above,
is defined by the sequence $P^\prime \subset S$ of
{\it simple roots.} It is divided to the set ${P^{\prime}}^{\,re}$ of
{\it simple real root} (all of them are even)
and the set ${P^{\prime}}^{\,im}_{\overline{0}}$
{\it of even simple imaginary roots} and the set
${P^{\prime}}^{\,im}_{\overline{1}}$
{\it of odd imaginary roots.}
Thus, $P^\prime={P^{\prime}}^{\,re}\cup
{P^{\prime}}^{\,im}_{\overline{0}}\cup {P^{\prime}}^{\,im}_{\overline{1}}$.

\medskip

For a primitive $a\in S\cap \br_{++}\M$ with $(a,a)=0$ one should find
$\tau(na)\in \bz$, $n\in \bn$, from the indentity with the formal variable $t$:
$$
1-\sum_{k\in \bn}{m(ka)t^k}=\prod_{n\in \bn}{(1-t^n)^{\tau(na)}}.
$$

\medskip

The set ${P^\prime}^{\,re}=P$ where $P$ is defined in
(III). The set ${P^\prime}^{\,re}$ is even:
${P^\prime}^{\,re}={{P^\prime}^{\,re}}_{\overline{0}}$,
${{P^\prime}^{\,re}}_{\overline{1}}=\emptyset$.
The set
\begin{equation}
{{P^\prime}^{\,im}}_{\overline{0}}=
\{m(a)a\ |\ a\in S\cap \br_{++}\M,\ (a,a)<0\ and\ m(a)>0\}\cup
$$
$$
\{\tau (a)a\ |\ a\in S\cap \br_{++}\M,\ (a,a)=0\ and\ \tau (a)>0\};
\label{rootim0}
\end{equation}

\medskip

\begin{equation}
{{P^\prime}^{\,im}}_{\overline{1}}=
\{-m(a)a\ |\ a\in S\cap \br_{++}\M,\ (a,a)<0\ and\ m(a)<0\}\cup
$$
$$
\{-\tau (a)a\ |\ a\in S\cap \br_{++}\M,\ (a,a)=0\ and\ \tau (a)<0\}
\label{rootim1}
\end{equation}
Here, $ka$ for $k\in \bn$ means that we repeat $a$ exactly $k$
times in the sequence.

The generalized Kac--Moody superalgebra $\gggg$ is the Lie superalgebra
with generators $h_r$, $e_r$, $f_r$ where $r\in P^\prime$. All
generators $h_r$ are even, generators $e_r$, $f_r$ are even
(respectively odd) if $r$ is even (respectively odd).

They have defining relations 1) --- 5) of $\gggg$
which are given below.

\medskip

1) The map $r \to h_r$ for $r\in P^\prime$ gives
an embedding $S\otimes \bc$ to $\gggg$ as Abelian subalgebra
(it is even).

\medskip

2) $[h_r,e_{r^\prime}]=(r,r^\prime)e_{r^\prime}$ and
$[h_r,f_{r^\prime}]=-(r,r^\prime)f_{r^\prime}$.

\medskip

3) $[e_r,f_{r^\prime}]=h_r$ if $r=r^\prime$, and it is $0$,
if $r\not=r^\prime$.

\medskip

4) $(ad\ e_r)^{1-2(r,r^\prime)/(r,r)}e_{r^\prime}=
(ad\ f_r)^{1-2(r,r^\prime)/(r,r)}f_{r^\prime}=0$, \newline
if $r\not=r^\prime$ и $(r,r)>0$
(equivalently, $r\in {P^{\prime}}^{\ re}$).

\medskip

5) If $(r,r^\prime)=0$, then $[e_r,e_{r^\prime}]=[f_r,f_{r^\prime}]=0.$

\medskip

The algebra $\gggg$ is graded by the lattice $S$ where the generators $h_r$, $e_r$ and
$f_r$ have weights $0$, $r\in S$ and $-r\in S$ respectively. We have
\begin{equation}
\gggg=\bigoplus_{\alpha\in S}{\gggg_\alpha}=\gggg_0\bigoplus
\left(\bigoplus_{\alpha\in \Delta_+}{\gggg_\alpha}\right)\bigoplus
\left(\bigoplus_{\alpha\in -\Delta_+}{\gggg_\alpha}\right),
\label{Liegrading}
\end{equation}
where $\gggg_0=S\otimes \bc$, and $\Delta$ is the set of roots (that is the set of $\alpha\in S$
with $\dim \gggg_\alpha\not=0$). The root $\alpha$ is positive $(\alpha\in \Delta_+)$ if $(\alpha,\M)\le 0$.
By definition, the multiplicity of $\alpha\in \Delta$ is equal to $mult(\alpha)=\dim \gggg_{\alpha,\overline{0}}-
\dim \gggg_{\alpha,\overline{1}}$.

For this definition, we use results by Borcherds, authors, U. Ray.

\medskip

In Section \ref{sec:data1-3}, we give the classification of possible data (I) --- (III)
of elliptic type.
In Sections \ref{pullback}---\ref{seriesD8}, we construct some data (IV) --- (V)
for these data (I) --- (III)
of elliptic type and for some data (I) --- (III) of parabolic type.


\newpage

\section{Classification of elliptically 2-reflective\\
hyperbolic lattices with lattice Weyl\\ vectors}
\label{sec:data1-3}

\subsection{Notations}
\label{subsec:2-ref,not}
We follow definitions from Sec. \ref{sec:data1-5} of lattices,
hyperbolic lattices, roots, $2$-roots, reflections in roots,
hyperbolic spaces of hyperbolic lattices.

For a lattice $M$, we denote by $x\cdot y=(x,y)$, $x,y\in M$
the symmetric pairing of $M$, and $x^2=x\cdot x$, $x\in M$.

For a hyperbolic lattice $S$, we denote by $V^+(S)$ the half-cone
of $S$ and by $\La^+(S)=V^+(S)/\br_{++}$ the hyperbolic space of
$S$. By $\Hh_\delta$ and $\Hh_\delta^+$ we denote the hyperplane
and the half-space of $\La^+(S)$ which are orthogonal  to
$\delta\in S\otimes \br$
where $\delta^2>0$.

\subsection{Classification of elliptically 2-reflective hyperbolic\\ lattices}
\label{subsec:clas.ell2-refl}

Let $S$ be a hyperbolic lattice of the signature $(n,1)$ where
$\rk S=n+1\ge 3$.

It is well-known that the group
$$
O^+(S)=\{\phi \in O(S)\ |\ \phi(V^+(S))=V^+(S) \}
$$
is discrete in $\La^+(S)$ and has a fundamental domain of finite
volume. The subgroups $W(S)$ and $W^{(2)}(S)$ are
its subgroups generated by all and $2$-reflections respectively.
We denote by $\M\subset \La^+(S)$ and $\M^{(2)}\subset \La^+(S)$ their
fundamental chambers respectively.

\begin{definition} A hyperbolic lattice $S$ of $\rk S\ge 3$ is called
{\it elliptically reflective} (respectively elliptically $2$-reflective)
if $[O(S):W(S)]<\infty$ (respectively, $[O(S):W^{(2)}(S)]<\infty$).
Equivalently, the fundamental chamber $\M\subset \La^+(S)$
(respectively the fundamental chamber $\M^{(2)}\subset \La^+(S)$)
has finite volume.
\label{defrefl}
\end{definition}

In \cite{Nik2} for $\rk S\ge 5$, \cite{Vin4} (see also \cite{Nik6})
for $\rk S=4$, \cite{Nik5} for $\rk S=3$, all elliptically
$2$-reflective hyperbolic lattices were classified. It is enough to
classify even $2$-reflective hyperbolic lattices. Indeed, an odd
lattice will be $2$-reflective if and only if its maximal even
sublattice is $2$-reflective.

The list of all even elliptically $2$-reflective hyperbolic lattices is given
below. We use the standard notations. By $A_n$, $n\ge 1$;
$D_n$, $n\ge 4$; $E_n$, $n=6,\,7,\,8$, we denote the positive definite
root lattices of the corresponding root systems ${\mathbb A}_n$,
${\mathbb D}_n$, ${\mathbb E}_n$ with roots having the square $2$.
Their standard bases consist of bases of the the root systems.  By $U$,
we denote the hyperbolic even unimodular lattices of the rank $2$. For the
standard basis $\{c_1,c_2\}$, it has the matrix
$$
U=
\left(\begin{array}{rr}
 0 & -1  \\
-1 &  0
\end{array}\right).  \ \ \ \
$$
By $M(n)$ we denote a lattice which is obtained from $M$ by
multiplication by $n\in \bq$ of the form of a lattice $M$.
By $\langle A \rangle$ we denote a lattice defined by the symmetric
integral matrix $A$ in some (we call it standard) basis. If a lattice $M$
has a standard basis $e_1,\dots, e_n$, then $M[\alpha_1,\dots, \alpha_n]$
denotes a lattice which is obtained by adding to $M$ the element
$\alpha_1e_1+\cdots + \alpha_n e_n$. Here $\alpha_1,\dots,\alpha_n$ are
from $\bq$.
By $\oplus$ we denote the orthogonal sum of lattices. If $M_1$ and $M_2$ have
the standard bases, then the standard basis of $M_1\oplus M_2$ consists of
the union of these bases.

We have the following list of all
$2$-reflective hyperbolic lattices of rank $\ge 3$, up to isomorphisms:
\vskip0.5cm

{\bf The list of all elliptically 2-reflective even hyperbolic lattices
of rank $\ge 3$.}

\medskip

\noindent
If $\rk S=3$, there are 26 lattices which are 2-reflective. They are
described in \cite{Nik5} (We must correct the list of these lattices
in \cite{Nik5}: In notations of \cite{Nik5}, the lattices
$S^\prime_{6,1,2}=[3a+c,b,2c]$ and $S_{6,1,1}=[6a,b,c]$ are isomorphic,
they have isomorphic fundamental polygons. See calculations for the
proof of Theorem \ref{th:el2refW} below.)

\medskip

\noindent
If $\rk S=4$, then $S=\langle -8 \rangle\oplus 3A_1$; $U\oplus 2A_1$;
$\langle -2 \rangle \oplus 3A_1$; $U(k)\oplus 2A_1$, $k=3,\,4$;
$U\oplus A_2$; $U(k)\oplus A_2$, $k=2,\,3,\,6$;
$\left\langle\begin{array}{rr}
 0 & -3  \\
 -3 & 2
\end{array}\right\rangle \oplus A_2$; $\langle -4 \rangle\oplus
\langle 4\rangle\oplus A_2$;
$\langle -4 \rangle\oplus A_3$;

\noindent
$\left\langle\begin{array}{rrrr}
 -2 & -1 & -1 & -1  \\
 -1  & 2 & 0 & 0 \\
 -1  & 0 & 2 & 0 \\
 -1  & 0 & 0 & 2
\end{array}\right\rangle$,\ \ \ \ \ \
$\left\langle\begin{array}{rrrr}
 -12 & -2 & 0 & 0   \\
  -2  & 2 & -1 & 0  \\
  0  & -1 & 2 & -1 \\
  0  & 0 & -1 & 2 \\
\end{array}\right\rangle
$.

\medskip

\noindent
If $\rk S=5$, then $S=U\oplus 3A_1$; $\langle -2 \rangle \oplus 4A_1$;
$U\oplus A_1\oplus A_2$;
$U\oplus A_3$; $U(4)\oplus 3A_1$;
$\langle 2^k\rangle \oplus D_4$, $k=2,\,3,\,4$;
$\langle 6 \rangle \oplus 2A_2$.

\medskip

\noindent
If $\rk S=6$, then $S=U\oplus D_4$, $U(2)\oplus D_4$, $U\oplus 4A_1$,
$\langle -2 \rangle \oplus 5A_1$,
$U\oplus 2A_1\oplus A_2$, $U\oplus 2A_2$, $U\oplus A_1\oplus A_3$,
$U\oplus A_4$, $U(4)\oplus D_4$, $U(3)\oplus 2A_2$.

\medskip

\noindent
If $\rk S=7$, then $S=U\oplus D_4\oplus A_1$, $U\oplus 5A_1$,
$\langle -2 \rangle \oplus 6A_1$, $U\oplus A_1\oplus 2A_2$,
$U\oplus 2A_1\oplus A_3$,
$U\oplus A_2\oplus A_3$, $U\oplus A_1\oplus A_4$,
$U\oplus A_5$, $U\oplus D_5$.

\medskip

\noindent
If $\rk S=8$, then $S=U\oplus D_6$, $U\oplus D_4\oplus 2A_1$,
$U\oplus 6A_1$, $\langle -2 \rangle \oplus 7A_1$,
$U\oplus 3A_2$, $U\oplus 2A_3$, $U\oplus A_2\oplus A_4$,
$U\oplus A_1\oplus A_5$, $U\oplus A_6$, $U\oplus A_2\oplus D_4$,
$U\oplus A_1\oplus D_5$, $U\oplus E_6$.

\medskip

\noindent
If $\rk S=9$, then $S=U\oplus E_7$, $U\oplus D_6\oplus A_1$,
$U\oplus D_4\oplus 3A_1$, $U\oplus 7A_1$, $\langle -2 \rangle \oplus 8A_1$,
$U\oplus A_7$, $U\oplus A_3\oplus D_4$, $U\oplus A_2\oplus D_5$,
$U\oplus D_7$, $U\oplus A_1\oplus E_6$.

\medskip

\noindent
If $\rk S=10$, then $S=U\oplus E_8$, $U\oplus D_8$,
$U\oplus E_7\oplus A_1$,
$U\oplus D_4\oplus D_4$, $U\oplus D_6\oplus 2A_1$,
$U(2)\oplus D_4\oplus D_4$,
$U\oplus 8A_1$, $U\oplus A_2\oplus E_6$.

\medskip

\noindent
If $\rk S=11$, then $S=U\oplus E_8\oplus A_1$, $U\oplus D_8\oplus A_1$,
$U\oplus D_4\oplus D_4\oplus A_1$,
$U\oplus D_4\oplus 5A_1$.

\medskip

\noindent
If $\rk S=12$, then $S=U\oplus E_8\oplus 2A_1$, $U\oplus D_8\oplus 2A_1$,
$U\oplus D_4\oplus D_4\oplus 2A_1$, $U\oplus A_2\oplus E_8$.

\medskip

\noindent
If $\rk S=13$, then $S=U\oplus E_8\oplus 3A_1$,
$U\oplus D_8\oplus 3A_1$, $U\oplus A_3\oplus E_8$.

\medskip

\noindent
If $\rk S=14$, then $S=U\oplus E_8\oplus D_4$,
$U\oplus D_8\oplus D_4$, $U\oplus E_8\oplus 4A_1$.

\medskip

\noindent
If $\rk S=15$, then $S=U\oplus E_8\oplus D_4\oplus A_1$.

\medskip

\noindent
If $\rk S=16$, then $S=U\oplus E_8\oplus D_6$.

\medskip

\noindent
If $\rk S=17$, then $S=U\oplus E_8\oplus E_7$.

\medskip

\noindent
If $\rk S=18$, then $S=U\oplus 2E_8$.

\medskip

\noindent
If $\rk S=19$, then $S=U\oplus 2E_8\oplus A_1$.

\medskip

\noindent
If $\rk S\ge 20$, there are no such lattices.

\medskip

Calculations of their fundamental chambers $\M^{(2)}$ and
the finite sets $P(\M^{(2)})$ of 2-roots which are perpendicular
to codimension one faces of $\M^{(2)}$ and directed outwards are also known.
See \cite{Nik2}, \cite{Vin4} and \cite{Nik5} for almost all cases.
See also below.

\begin{remark}
\label{rem:K3finautgroup}
{\rm By global Torelli Theorem for K3 surfaces \cite{PS},
ellptically $2$-reflective hyperbolic lattices $S$
from the list above, give all Picard lattices $S_X=S(-1)$ of
K3 surfaces $X$ over $\bc$ with finite automorphism group and
$\rk S_X\ge 3$.  They have signature $(1,n)$ where
$\rk S_X=\rk S(-1)=\rk S=n+1\ge 3$.
The set $P(\M^{(2)})\subset S(-1)=S_X$ gives classes
of all non-singular rational curves on $X$.
Their number is finite and they generate $S_X$ up to finite index.}
\end{remark}

\begin{remark}
\label{rem:Finreflat}
{\rm There are general finiteness results about reflective
hyperbolic lattices and arithmetic hyperbolic reflection groups.
See \cite{Nik3}, \cite{Nik4}, \cite{Nik6},
\cite{Nik9} and \cite{Vin3}, \cite{Vin4}.

Classification of maximal reflective
(elliptically, parabolically or hyperbolically)
hyperbolic lattices of rank 3 was obtained in \cite{Nik11}.
Classification of elliptically reflective hyperbolic lattices of rank 3
was obtained by D. Allcock in \cite{Al1}.}
\end{remark}

\subsection{Classification of elliptically 2-reflective even\\ hyperbolic
lattices $S$ with lattice Weyl vector\\ for $W^{(2)}(S)$}
\label{subsec:clas.el2-refWeyl}

The particular cases of elliptically 2-reflective hyperbolic lattices will be
important for us. They are characterized by the property that they
have the lattice Weyl vector.

Let $S$ be an elliptically 2-reflective hyperbolic lattice,
$\M^{(2)}(S)\subset \La^+(S)$ the fundamental chamber for $W^{(2)}(S)$,
and $P(\M^{(2)}(S))$ the set of perpendicular $2$-roots to $\M^{(2)}(S))$
directed outwards.
That is
$$
\M^{(2)}(S)=\{\br_{++}x\in \La^+(S)\ |\ x\cdot P(\M^{(2)}(S))\le 0\}
$$
and $P(\M^{(2)}(S))$ is minimal with this property.

\begin{definition} A 2-reflective hyperbolic lattice $S$ has a
lattice Weyl vector for
$W^{(2)}(S)$ (equivalently, for $P(\M^{(2)}(S))$) if there exists
$\rho\in S\otimes \bq$
such that
\begin{equation}
\rho\cdot \delta=-1\ \ \ \forall\  \delta\in P(\M^{(2)}(S)).
\label{latWfor2}
\end{equation}
The $\rho$ is called the lattice Weyl vector for $P(\M^{(2)}(S))$.

More generally, a reflective (elliptically or parabolically) hyperbolic lattice $S$ has a
lattice Weyl vector for a reflection subgroup
$W\subset W(S)$ and a set $P(\M)$ of roots of $S$ which are perpendicular to a
fundamental chamber $\M$ of $W$ and directed outwards if there exists
$\rho\in S\otimes \bq$
such that
\begin{equation}
\rho\cdot \delta=-\frac{\delta^2}{2}\ \ \ \forall\  \delta\in P(\M).
\label{latWfor}
\end{equation}
The $\rho$ is called the lattice Weyl vector for $P(\M)$.

\end{definition}

Since $\M^{(2)}(S)$ has finite volume, the set $P(\M^{(2)}(S))$
generates $S\otimes \bq$, and the $\rho$ is defined uniquely.
The $\br_{++}\rho$ belongs
to the interior of $\M^{(2)}(S)$, and $\rho^2<0$. Geometrically,
$\br_{++}\rho$ gives a center of a sphere which is inscribed
to the fundamental chamber $\M^{(2)}$.
Thus, this case is especially special and beautiful.

Using classification of elliptically 2-reflective hyperbolic lattices,
we obtain classification of
elliptically 2-reflective hyperbolic lattices with lattice Weyl vectors.

\begin{theorem} The following and only the following elliptically
2-reflective even hyperbolic lattices $S$ of $\rk S\ge 3$ have
a lattice Weyl vector $\rho$ for $W^{(2)}(S)$ (equivalently,
for $P(\M^{(2)}(S))$). We order them by the rank and the absolute
value of the determinant.
\label{th:el2refW}

\noindent
Rank 3:
$S_{3,2}=U\oplus A_1$,
$S_{3,8,a}=\langle -2\rangle \oplus 2A_1$,

\noindent
$S_{3,8,b}=(\langle -24 \rangle\oplus A_2)[1/3,-1/3,1/3]$,

\noindent
$S_{3,18}=U(3)\oplus A_1$,
$S_{3,32,a}=U(4)\oplus A_1$,
$S_{3,32,b}=\langle -8\rangle \oplus 2A_1$,
$S_{3,32,c}=U(8)[1/2,1/2]\oplus A_1$,
$S_{3,72}=\langle -24\rangle\oplus A_2$,
$S_{3,128,a}=U(8)\oplus A_1$,
$S_{3,128,b}=\langle -32\rangle \oplus 2A_1$,
$S_{3,288}=U(12)\oplus A_1$,

\noindent
anisotropic cases:
$S_{3,12}=\langle -4\rangle \oplus A_2$,
$S_{3,24}=\langle -6\rangle \oplus 2A_1$,
$S_{3,36}=\langle -12\rangle \oplus  A_2$,
$S_{3,108}=\langle -36\rangle \oplus A_2$.

\medskip

\noindent
Rank 4: $S_{4,3}=U\oplus A_2$, $S_{4,4}=U\oplus 2A_1$,
$S_{4,12}=U(2)\oplus A_2$, $S_{4,16,a}=\langle -2 \rangle\oplus 3A_1$,
$S_{4,16,b}=\langle -4\rangle\oplus A_3$, $S_{4,27,a}=U(3)\oplus A_2$,
$S_{4,27,b}=\left\langle\begin{array}{rr}
 0 & -3  \\
 -3 & 2
\end{array}\right\rangle \oplus A_2$,
$S_{4,64,a}=U(4)\oplus 2A_1$, $S_{4,64,b}=\langle -8 \rangle\oplus 3A_1$,
$S_{4,108}=U(6)\oplus A_2$,

\noindent
$S_{4,28}=
\left\langle\begin{array}{rrrr}
 -2 & -1 & -1 & -1  \\
 -1  & 2 & 0 & 0 \\
 -1  & 0 & 2 & 0 \\
 -1  & 0 & 0 & 2
\end{array}\right\rangle$ (anisotropic case).

\medskip

\noindent
Rank 5: $S_{5,4}=U\oplus A_3$, $S_{5,8}=U\oplus 3A_1$,
$S_{5,16}=\langle -4\rangle \oplus D_4$,
$S_{5,32,a}=\langle -2 \rangle\oplus 4A_1$,
$S_{5,32,b}=\langle -8\rangle\oplus D_4$,
$S_{5,64}=\langle -16 \rangle\oplus D_4$,
$S_{5,128}=U(4)\oplus 3A_1$.

\medskip

\noindent
Rank 6: $S_{6,4}=U\oplus D_4$, $S_{6,5}=U\oplus A_4$,
$S_{6,9}=U\oplus 2A_2$, $S_{6,16,a}=U(2)\oplus D_4$,
$S_{6,16,b}=U\oplus 4A_1$, $S_{6,64,a}=\langle -2 \rangle \oplus 5A_1$,
$S_{6,64,b}=U(4)\oplus D_4$, $S_{6,81}=U(3)\oplus 2A_2$.

\medskip

\noindent
Rank 7: $S_{7,4}=U\oplus D_5$, $S_{7,6}=U\oplus A_5$,
$S_{7,128}=\langle -2 \rangle \oplus 6A_1$.

\medskip

\noindent
Rank 8: $S_{8,3}=U\oplus E_6$, $S_{8,4}=U\oplus D_6$, $S_{8,7}=U\oplus A_6$,
$S_{8,16}=U\oplus 2A_3$, $S_{8,27}=U\oplus 3A_2$,
$S_{8,256}=\langle -2 \rangle\oplus 7A_1$.

\medskip

\noindent
Rank 9: $S_{9,2}=U\oplus E_7$, $S_{9,4}=U\oplus D_7$, $S_{9,8}=U\oplus A_7$,
$S_{9,512}=\langle -2\rangle\oplus 8A_1$.

\medskip

\noindent
Rank 10: $S_{10,1}=U\oplus E_8$, $S_{10,4}=U\oplus D_8$,
$S_{10,16}=U\oplus 2D_4$,
$S_{10,64}=U(2)\oplus 2D_4$.

\medskip

\noindent
Rank 18: $S_{18,1}=U\oplus 2E_8$.
\end{theorem}

We shall discuss the proof of Theorem \ref{th:el2refW} in the next section.

\medskip

\subsection{The fundamental chambers $\M^{(2)}$ and the lattice Weyl vectors
for lattices of Theorem \ref{th:el2refW}}
\label{subsec:grmatel2refW}

Below, for lattices of Theorem  \ref{th:el2refW}, we describe the sets
$P(\M^{(2)})$ and the Weyl vectors. This describes Gram graphs
$\Gamma(P(\M^{(2)}))$ too.

We recall that for $P(\M^{(2)})$
one connects $\delta_1, \delta_2\in P(\M^{(2)})$ by the edge if
$\delta_1\cdot \delta_2<0$. This edge a thin, thick, and broken of the
weight $-\delta_1\cdot\delta_2$ if $\delta_1\cdot \delta_2=-1$,
$\delta_1\cdot\delta_2=-2$, and $\delta_1\cdot\delta_2<-2$ respectively.

More generally, for the set $P(\M)$ of perpendicular roots to a
fundamental chamber $\M$ of a hyperbolic reflection group,
one adds weights $\delta^2$ to vertices corresponding to
$\delta\in P(\M)$ with $\delta^2\not=2$ (we draw them black
and don't put the weight if $\delta^2=4$). The edge for different
$\delta_1,\delta_2\in P(\M)$ is thin of the natural weight
$n\ge 3$ (equivalently, the $n-2$-multiple thin edge for small $n$),
thick, and broken of the weight
$-2\,\delta_1\cdot \delta_2/\sqrt{\delta_1\cdot \delta_2}$
if $2\,\delta_1\cdot \delta_2/\sqrt{\delta_1\cdot\delta_2}=-2\,\cos{(\pi/n)}$,
$2\,\delta_1\cdot \delta_2/\sqrt{\delta_1\cdot\delta_2}=-2$, and
$2\,\delta_1\cdot \delta_2/\sqrt{\delta_1\cdot\delta_2}<-2$ respectively.

\medskip

We recall that a lattice $M$ is 2-elementary if its discriminant group
$M^\ast/M$ is 2-elementary, that is $M^\ast/M\cong (\bz/2\bz)^a$.

\medskip

\noindent
{\it Cases $S=U\oplus K$ where $K=\oplus_i^n{K_i}$ is the orthogonal
sum of 2-roots lattices $A_n$, $D_n$, $E_n$.}

Then $P(\M^{(2)})$ consists of $e=-c_1+c_2$, bases of root lattices $K_i$,
$c_1-w_i$ where $w_i$ are the maximal roots of $K_i$ corresponding to
the standard bases of $K_i$, $i=1,\dots,n$. The corresponding graph
\begin{equation}
\Gamma=\Gamma(P(\M^{(2)}))=St(\Gamma(\widetilde{K_1}),\dots,
\Gamma(\widetilde{K_n}))
\label{star}
\end{equation}
is called {\it the Star of the corresponding
extended Dynkin diagrams.}
Here $e=-c_1+c_2$ is the center
of the Star. The graph $\Gamma-\{e\}$ consists of $n$ connected components
$\Gamma(\widetilde{K_i})$ with the bases which are
the bases of $K_i$ and $c_1-w_i$. They give the corresponding
extended Dynkin diagrams
$\widetilde{{\mathbb A}_n}$, $\widetilde{{\mathbb D}_n}$, and
$\widetilde{{\mathbb E}_n}$. Obviously, $e$ is connected (by the thin edge)
with
$c_1-w_i$, $i=1,\dots, n$, only. See \cite{Nik2} for details.

Using this description, for all these cases of Theorem \ref{th:el2refW},
one can calculate the Weyl vector $\rho$ directly using \eqref{latWfor2}, and
prove that it does exist. For example, in Figure \ref{ua2a2}, we draw
the graph for the lattice $U\oplus 2A_2$. The rational weights for its
vertices show the linear combination of elements of $P(\M^{(2)})$
which gives the lattice Weyl vector $\rho$.
If $n=1$ (this is valid for the most cases),
then $P(\M^{(2)})$ gives the basis of the lattice $S$, and then $\rho$
exists obviously.

\begin{figure}
\begin{center}
\includegraphics[width=8cm]{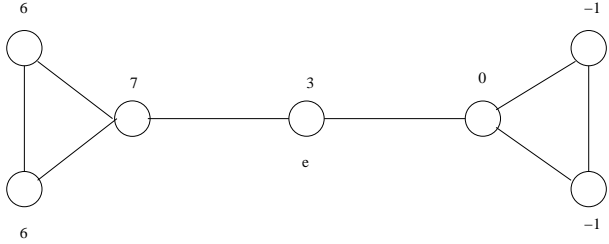}
\end{center}
\caption{The graph $\Gamma(P(\M^{(2)}))$ for $U\oplus A_2\oplus A_2$
is $St(\widetilde{\mathbb A_2},\widetilde{\mathbb A_2})$.}
\label{ua2a2}
\end{figure}

For all remaining similar cases of elliptically 2-reflective lattices
of Sect. \ref{subsec:clas.ell2-refl},
the star \eqref{star} gives a part of $\Gamma(P(\M^{(2)}))$
(for many cases these graphs coincide, for example if $S$
is not $2$-elementary). Calculation of the Weyl vector $\rho\in S\otimes \bq$
satisfying $\rho\cdot \delta=-1$ for all $\delta$ of the
star \eqref{star} show that it does not exist for all cases, except
$S=U\oplus nA_1$, $5\le n\le 8$. For these lattices, the full graph
$\Gamma(P(\M^{(2)}))$ is calculated in \cite[Table 1]{AN}, and
these calculations show that $\rho$ does not exist in these cases either.

\medskip

\noindent
{\it Cases $S=U(2)\oplus D_4$, $U(2)\oplus 2D_4$,
$\langle -2 \rangle\oplus nA_1$, $1\le n\le 8$.}

They give remaining $2$-elementary cases of Theorem \ref{th:el2refW}.
All these cases are classical. For example, one can find calculation
of the graphs $\Gamma(P(\M^{(2)}))$ in \cite[Table 1]{AN}.

We choose the standard basis $e_1,e_2,e_3,e_4$ for
$D_4$ such that $w=e_1+e_2+e_3+2e_4$ is the maximal root.

Let $S=U(2)\oplus D_4$ with the corresponding standard basis.
Then $P(\M^{(2)})$ consists of elements $e_1=(0,0,1,0,0,0)$,
$e_2=(0,0,0,1,0,0)$, \newline $e_3=(0,0,0,0,1,0)$, $e_4=(0,0,0,0,0,1)$,
$c_1-w=(1,0,-1,-1,-1,-2)$, $c_2-w=(0,1,-1,-1,-1,-2)$.
The Weyl vector $\rho=(3,3,-3,-3,-3,-5)$. See $\Gamma(P(\M^{(2)}))$
in Figure \ref{u2d4}.

\begin{figure}
\begin{center}
\includegraphics[width=6cm]{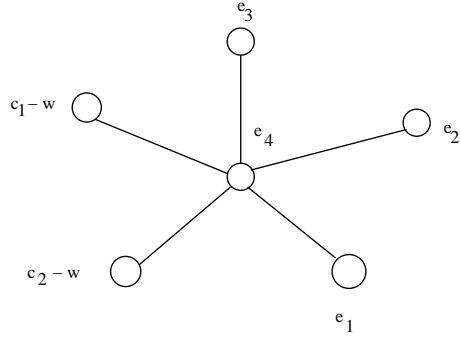}
\end{center}
\caption{The graph $\Gamma(P(\M^{(2)}))$ for $U(2)\oplus D_4$.}
\label{u2d4}
\end{figure}

Let $S=U(2)\oplus 2D_4$. The same lattice can be written in the
form $S=U\oplus \left(8A_1[1/2,1/2,1/2,1/2,1/2,1/2,1/2,1/2]\right)$.
We use the standard basis $c_1,c_2,e_1,\dots,e_8$ for $U\oplus 8A_1$.
The set  $P(\M^{(2)})$ consists of $e_1,\dots, e_8$,
$e_1^\prime =c_1-e_1,\,e_2^\prime=c_1-e_2, \dots,\,e_8^\prime=c_1-e_8$,
$f=-c_1+c_2$ and $f^\prime =c_1+c_2-(e_1+e_2+\cdots +e_8)/2$.
The Weyl vector $\rho=3c_1+2c_2-(e_1+e_2+\cdots +e_8)/2$.
See $\Gamma(P(\M^{(2)}))$ in Figure \ref{u22D4}.

\begin{figure}
\begin{center}
\includegraphics[width=8cm]{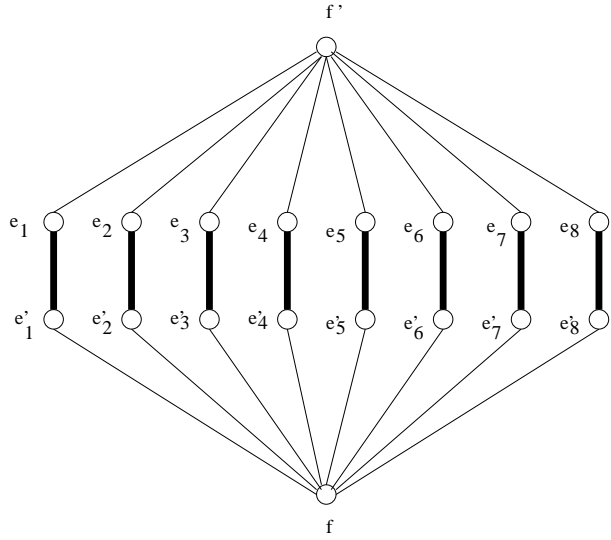}
\end{center}
\caption{The graph $\Gamma(P(\M^{(2)}))$ for $U(2)\oplus 2D_4$.}
\label{u22D4}
\end{figure}

Let $S=\langle -2 \rangle\oplus nA_1$, $2\le n\le 8$. Then calculation of
$P(\M^{(2)})$ is equivalent to the calculation of classes of exceptional
curves in Picard lattices of the rank $n+1$ for non-singular Del Pezzo surfaces.
Then the Weyl vector $\rho$ is equivalent to the anti-canonical class.

For $S=\langle -2 \rangle \oplus nA_1$, $2\le n\le 8$, with the standard basis
$h,e_1,\dots,e_n$,
the Weyl vector $\rho=3h-e_1-e_2-\cdots -e_n$, and
\begin{equation}
P(\M^{(2)})=\{\delta\in S\ |\ \delta^2=2\ \& \ \delta\cdot \rho=-1\}\,.
\label{DelPezzo}
\end{equation}
Then $P(\M^{(2)})$ consists of all elements below which one
can get by all permutations of $e_1,\dots, e_n$. They are
$e_1$, $h-e_1-e_2$ for $n\ge 2$; $2h-e_1-e_2-e_3-e_4-e_5$ for $n\ge 5$;
$3h-2e_1-e_2-e_3-e_4-e_5-e_6-e_7$ for $n\ge 7$;
$4h-2e_1-2e_2-2e_3-e_4-e_5-e_6-e_7-e_8$,
$5h-2e_1-2e_2-2e_3-2e_4-2e_5-2e_6-e_7-e_8$, $6h-3e_1-2e_2-\cdots -2e_8$
for $n=8$. For example, see \cite[Ch. 4, Sect. 4.2]{Man}. Thus,
$P(\M^{(2)})$ consists of $240$ elements for $n=8$; $56$ elements for
$n=7$; $27$ elements for $n=6$; $16$ elements for $n=5$; $10$ elements for
$n=4$; $6$ elements for $n=3$; $3$ elements for $n=2$; $1$ element for
$n=1$.   It is hard
to draw the corresponding graphs for big $n$. For $n=2$ and $n=3$ we draw
them in Figures \ref{del3} and \ref{del4}.

\begin{figure}
\begin{center}
\includegraphics[width=4cm]{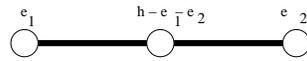}
\end{center}
\caption{The graph $\Gamma(P(\M^{(2)}))$ for $\langle -2 \rangle \oplus 2A_1$.}
\label{del3}
\end{figure}

\medskip

\begin{figure}
\begin{center}
\includegraphics[width=3cm]{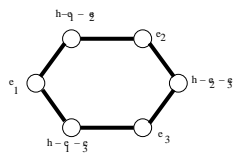}
\end{center}
\caption{The graph $\Gamma(P(\M^{(2)}))$ for $\langle -2 \rangle \oplus 3A_1$.}
\label{del4}
\end{figure}

This proves Theorem \ref{th:el2refW} for $\rk S\ge 7$.
Below we consider remaining cases.

\medskip

\noindent
{\it Remaining cases of $\rk S=6$.}

\medskip

Let $S=S_{6,64,b}=U(4)\oplus D_4$. See \cite[Sec. 6.4]{Nik2}.
Let $c_1,c_2,e_1,e_2,e_3,e_4$ be
its standard basis where $w=e_1+e_2+e_3+2e_4$ is the maximal root of $D_4$.
Then  $P(\M^{(2)})$ consists of $e_1,\dots,e_4$, $e_0=c_1-w$,
$e_0^\prime=c_2-w$, and
$e_i^\prime=c_1+c_2-2e_1-2e_2-2e_3-4e_4-e_i$, $i=1,2,3$.
The Weyl vector is
$\rho=(3c_1+3c_2)/2-3e_1-3e_2-3e_3-5e_4$. See $\Gamma(P(\M^{(2)}))$
in Figure \ref{u4d4}.

\begin{figure}
\begin{center}
\includegraphics[width=6cm]{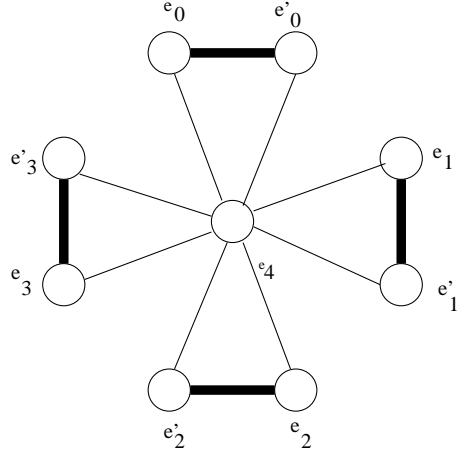}
\end{center}
\caption{The graph $\Gamma(P(\M^{(2)}))$ for $U(4) \oplus D_4$.}
\label{u4d4}
\end{figure}

Let $S=S_{6,81}=U(3)\oplus 2A_2$. Let $c_1,c_2,e_1,e_2,e_3,e_4$ be
its standard basis. The set $P(\M^{(2)})$ consists of $e_i$, $i=1,2,3,4$;
$c_1-e_1-e_2$, $c_1-e_3-e_4$, $c_2-e_1-e_2$, $c_2-e_3-e_4$;
$c_1+c_2-e_1-e_2-e_3-e_4-e_i$, $i=1,2,3,4$.  The Weyl vector
$\rho=c_1+c_2-e_1-e_2-e_3-e_4$. The $P(\M^{(2)})$  has the Gram matrix
\begin{equation}
-\left(\begin{array}{rrrrrrrrrrrr}
-2 & 1 & 0 & 0 & 1 & 0 & 1 & 0 & 3 & 0 & 1 & 1 \\
 1 & -2& 0 & 0 & 1 & 0 & 1 & 0 & 0 & 3 & 1 & 1 \\
 0 & 0 & -2& 1 & 0 & 1 & 0 & 1 & 1 & 1 & 3 & 0 \\
 0 & 0 & 1 & -2& 0 & 1 & 0 & 1 & 1 & 1 & 0 & 3 \\
 1 & 1 & 0 & 0 & -2& 0 & 1 & 3 & 0 & 0 & 1 & 1 \\
 0 & 0 & 1 & 1 & 0 & -2& 3 & 1 & 1 & 1 & 0 & 0 \\
 1 & 1 & 0 & 0 & 1 & 3 & -2& 0 & 0 & 0 & 1 & 1 \\
 0 & 0 & 1 & 1 & 3 & 1 & 0 & -2& 1 & 1 & 0 & 0 \\
 3 & 0 & 1 & 1 & 0 & 1 & 0 & 1 & -2& 1 & 0 & 0 \\
 0 & 3 & 1 & 1 & 0 & 1 & 0 & 1 & 1 & -2& 0 & 0 \\
 1 & 1 & 3 & 0 & 1 & 0 & 1 & 0 & 0 & 0 & -2& 1 \\
 1 & 1 & 0 & 3 & 1 & 0 & 1 & 0 & 0 & 0 & 1 &-2
\label{gram6,64,b}
\end{array}\right)
\end{equation}
which is very regular.

Thus, we have considered all lattices of the rank 6 of Sect.
\ref{subsec:clas.ell2-refl},
only the lattices of Theorem \ref{th:el2refW} have the lattice Weyl vectors.

\medskip

{\it Remaining cases of $\rk S=5$.}

\medskip

Let $S=S_{5,16}=\langle -4\rangle \oplus D_4$. See \cite[Sec. 8.5]{Nik2}.
For the standard basis
$h,e_1,e_2,e_3,e_4$ (for $D_4$, we use the same as above), we get that
$P(\M^{(2)})$ consists of $e_i$, $i=1,2,3,4$, and
$e_4^\prime=h-2e_1-2e_2-2e_3-3e_4$. The lattice Weyl vector is
$\rho=(5/2)h-3e_1-3e_2-3e_3-5e_4$. See Figure \ref{min4d4} for the Gram graph.

\begin{figure}
\begin{center}
\includegraphics[width=5cm]{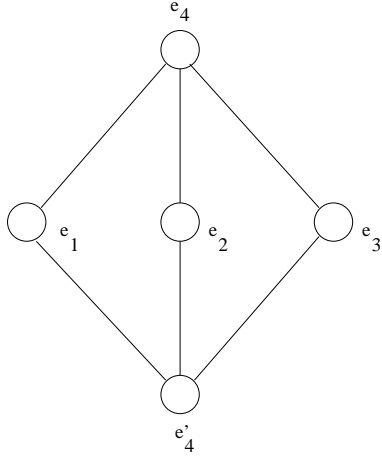}
\end{center}
\caption{The graph $\Gamma(P(\M^{(2)}))$ for $\langle -4 \rangle \oplus D_4$.}
\label{min4d4}
\end{figure}

Let $S=S_{5,32,b}=\langle -8\rangle\oplus D_4$. See \cite[Sec. 8.6]{Nik2}.
Then $P(\M^{(2)})$  consists of $e_i$, $i=1,2,3,4$, and
$f_i=h-2e_1-2e_2-2e_3-4e_4-e_i$, $i=1,2,3$.
The lattice Weyl vector is $\rho=(3/2)h-3e_1-3e_2-3e_3-5e_4$.
See Figure \ref{min8d4} for the Gram graph.

\begin{figure}
\begin{center}
\includegraphics[width=5cm]{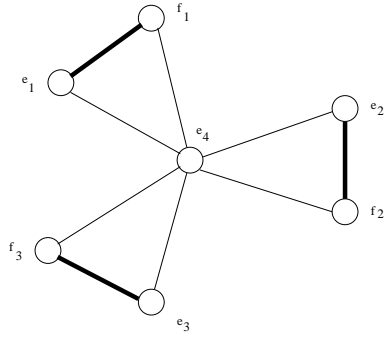}
\end{center}
\caption{The graph $\Gamma(P(\M^{(2)}))$ for $\langle -8 \rangle \oplus D_4$.}
\label{min8d4}
\end{figure}

Let $S=S_{5,64}=\langle -16 \rangle\oplus D_4$.
See \cite[Sec. 8.5]{Nik2}.
Then $P(\M^{(2)})$
consists of $e_i$, $i=1,2,3,4$;
$f_i=h-3e_1-3e_2-3e_3-5e_4-e_i$, $i=1,2,3$; $f_4=h-3e_1-3e_2-3e_3-6e_4$.
The lattice Weyl vector is $\rho=h-3e_1-3e_2-3e_3-5e_4$.
See Figure \ref{min16d4} for the Gram graph.

\begin{figure}
\begin{center}
\includegraphics[width=6cm]{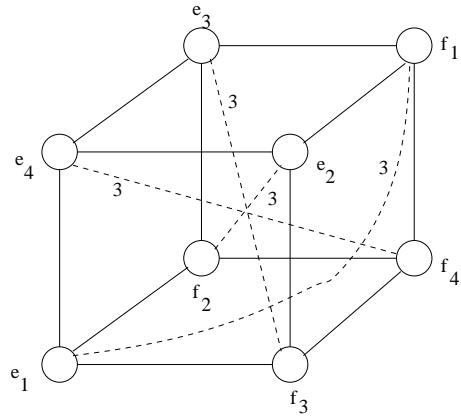}
\end{center}
\caption{The graph $\Gamma(P(\M^{(2)}))$ for $\langle -16 \rangle \oplus D_4$.}
\label{min16d4}
\end{figure}

\medskip

Let $S=S_{5,128}=U(4)\oplus 3A_1$. See details in
\cite[Sec. 8.3]{Nik2}.
For the standard basis $c_1,c_2,e_1,e_2,e_3$, the Weyl vector
$\rho=(c_1+c_2-e_1-e_2-e_3)/2$, and
$$
P(\M^{(2)})=\{\delta\in S\ |\ \delta^2=2\ \&\ \delta\cdot \rho=-1\}.
$$
The set
$P(\M^{(2)})$ consists of $e_i$, $c_1-e_i$, $c_2-e_i$, $i=1,2,3$;
$c_1+c_2-2e_i-e_j$, $1\le i\not=j\le 3$; $2c_1+c_2-2e_1-2e_2-2e_3+e_i$,
$c_1+2c_2-2e_1-2e_2-2e_3+e_i$,
$i=1,2,3$;  $2c_1+2c_2-2e_1-2e_2-2e_3-e_i$, $i=1,2,3$. It has 24 elements,
and its Gram graph is very symmetric.

\medskip

For the remaining lattice $S=\langle 6 \rangle\oplus 2A_2$ of rank 5 of Sect.
\ref{subsec:clas.ell2-refl}, the Gram graph $\Gamma(P(\M^{(2)}))$
is described in \cite[Sec. 8.6]{Nik2}, and it has no the lattice
Weyl vector.

\medskip

{\it Remaining cases of $\rk S=4$.} We use Vinberg's algorithm \cite{Vin2}
to calculate $P(\M^{(2)})$, and either to find the lattice Weyl vector or
to prove that it does not exist.

\medskip

Let $S=S_{4,12}=U(2)\oplus A_2$. For the standard basis $c_1$, $c_2$,
$e_1$, $e_2$, the set $P(\M^{(2)})$ consists of
$e_1$, $e_2$, $f_1=c_1-e_1-e_2$, $f_2=c_2-e_1-e_2$. The lattice
Weyl vector is $\rho=(3/2)(c_1+c_2)-e_1-e_2$. See
Figure \ref{u2a2} for the Gram graph.

\begin{figure}
\begin{center}
\includegraphics[width=4cm]{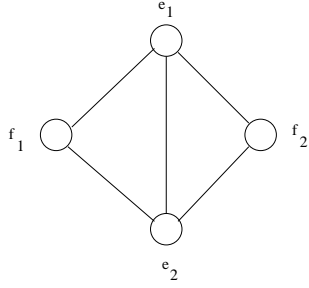}
\end{center}
\caption{The graph $\Gamma(P(\M^{(2)}))$ for $U(2)\oplus A_2$.}
\label{u2a2}
\end{figure}

Let $S=S_{4,16,b}=\langle -4 \rangle \oplus A_3$. For the
standard basis $h$, $e_1$, $e_2$, $e_3$, the set
$P(\M^{(2)})$ consists of $e_1$, $e_2$, $e_3$,
$f_1=h-2e_1-2e_2-e_3$, $f_3=h-e_1-2e_2-2e_3$. The lattice Weyl vector
$\rho=(3/2)h-(3/2)e_1-2e_2-(3/2)e_3$. See Figure \ref{min4a3} for
the Gram graph.

\begin{figure}
\begin{center}
\includegraphics[width=4cm]{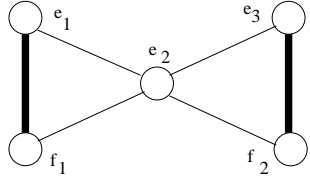}
\end{center}
\caption{The graph $\Gamma(P(\M^{(2)}))$ for $\langle -4 \rangle\oplus A_3$.}
\label{min4a3}
\end{figure}

Let $S=S_{4,27,a}=U(3)\oplus A_2$. For the standard basis $c_1$, $c_2$,
$e_1$, $e_2$, the set $P(\M^{(2)})$ consists of
$e_1$, $e_2$, $f_1=c_1-e_1-e_2$, $f_2=c_2-e_1-e_2$. The lattice
Weyl vector is $\rho=c_1+c_2-e_1-e_2$. See
Figure \ref{u3a2} for the Gram graph.

\begin{figure}
\begin{center}
\includegraphics[width=4cm]{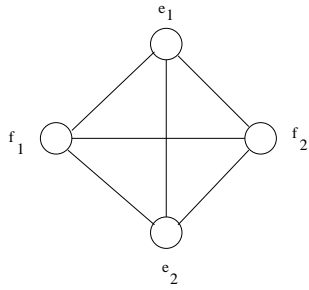}
\end{center}
\caption{The graph $\Gamma(P(\M^{(2)}))$ for $U(3)\oplus A_2$.}
\label{u3a2}
\end{figure}

Let
$S=S_{4,27,b}
=\left\langle\begin{array}{rr}
 0 & -3  \\
 -3 & 2
\end{array}\right\rangle \oplus A_2$.
For the standard basis $c$, $e_1$, $e_2$, $e_3$,
the set $P(\M^{(2)})$ consists of
$e_1$, $e_2$, $e_3$, $f_1=c-e_2-e_3$, $f_2=c+e_1-e_2-2e_3$,
$f_3=c+e_1-2e_2-e_3$. The lattice
Weyl vector is $\rho=c+e_1-e_2-e_3$. See
Figure \ref{s427b} for the Gram graph.

\begin{figure}
\begin{center}
\includegraphics[width=6cm]{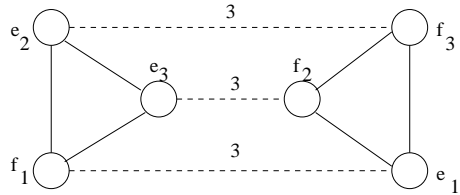}
\end{center}
\caption{The graph $\Gamma(P(\M^{(2)}))$ for
$S_{4,27,b}$.}
\label{s427b}
\end{figure}

\medskip

Let $S=S_{4,64,a}=U(4)\oplus 2A_1$.
For the standard basis $c_1$, $c_2$,
$e_1$, $e_2$, the set $P(\M^{(2)})$ consists of
$e_1$, $e_2$, $c_1-e_1$, $c_1-e_2$, $c_2-e_1$,
$c_2-e_2$, $c_1+c_2-2e_1-e_2$, $c_1+c_2-e_1-2e_2$.
The lattice
Weyl vector is $\rho=(c_1+c_2-e_1-e_2)/2$. The Gram matrix
$\Gamma(P(\M^{(2)}))$ is
\begin{equation}
-\left(\begin{array}{rrrrrrrr}
-2 & 0 & 2 & 0 & 2 & 0 & 4 & 2 \\
 0 & -2& 0 & 2 & 0 & 2 & 2 & 4 \\
 2 & 0 & -2& 0 & 2 & 4 & 0 & 2 \\
 0 & 2 & 0 & -2& 4 & 2 & 2 & 0 \\
 2 & 0 & 2 & 4 & -2& 0 & 0 & 2 \\
 0 & 2 & 4 & 2 & 0 & -2& 2 & 0 \\
 4 & 2 & 0 & 2 & 0 & 2 & -2& 0 \\
 2 & 4 & 2 & 0 & 2 & 0 & 0 & -2
\label{gram4,64,a}
\end{array}\right)
\end{equation}
which is very regular.

\medskip

Let $S=S_{4,64,b}=\langle -8 \rangle\oplus 3A_1$. For the
standard basis $h$, $e_1$, $e_2$, $e_3$, the set $P(\M^{(2)})$
consists of $e_1$, $e_2$, $e_3$, $h-e_1-2e_2$, $h-e_1-2e_3$,
$h-2e_1-e_2$, $h-e_2-2e_3$, $h-2e_1-e_3$, $h-2e_2-e_3$,
$2h-3e_1-2e_2-2e_3$, $2h-2e_1-3e_2-2e_3$, $2h-2e_1-2e_2-3e_3$.
The lattice Weyl vector is $\rho=(h-e_1-e_2-e_3)/2$. The Gram
matrix  $\Gamma(P(\M^{(2)}))$ is
\begin{equation}
-\left(\begin{array}{rrrrrrrrrrrr}
-2 & 0 & 0 & 2 & 2 & 4 & 0 & 4 & 0 & 6 & 4 & 4 \\
 0 & -2& 0 & 4 & 0 & 2 & 2 & 0 & 4 & 4 & 6 & 4 \\
 0 & 0 & -2& 0 & 4 & 0 & 4 & 2 & 2 & 4 & 4 & 6 \\
 2 & 4 & 0 & -2& 6 & 0 & 4 & 4 & 0 & 2 & 0 & 4 \\
 2 & 0 & 4 & 6 & -2& 4 & 0 & 0 & 4 & 2 & 4 & 0 \\
 4 & 2 & 0 & 0 & 4 & -2& 6 & 0 & 4 & 0 & 2 & 4 \\
 0 & 2 & 4 & 4 & 0 & 6 & -2& 4 & 0 & 4 & 2 & 0 \\
 4 & 0 & 2 & 4 & 0 & 0 & 4 & -2& 6 & 0 & 4 & 2 \\
 0 & 4 & 2 & 0 & 4 & 4 & 0 & 6 & -2& 4 & 0 & 2 \\
 6 & 4 & 4 & 2 & 2 & 0 & 4 & 0 & 4 & -2& 0 & 0 \\
 4 & 6 & 4 & 0 & 4 & 2 & 2 & 4 & 0 & 0 & -2& 0 \\
 4 & 4 & 6 & 4 & 0 & 4 & 0 & 2 & 2 & 0 & 0 &-2
\label{gram4,64,b}
\end{array}\right)
\end{equation}
which is very regular.

\medskip

Let $S=S_{4,108}=U(6)\oplus A_2$.
For the standard basis $c_1$, $c_2$,
$e_1$, $e_2$, the set $P(\M^{(2)})$ consists of
$e_1$, $e_2$, $f_3=c_1-e_1-e_2$, $f_4=c_2-e_1-e_2$,
$f_5=c_1+c_2-2e_1-3e_2$, $f_6=c_1+c_2-3e_1-2e_2$.
The lattice Weyl vector is $\rho=(c_1+c_2)/2-e_1-e_2$.
See Figure \ref{u6a2} for the Gram graph.

\begin{figure}
\begin{center}
\includegraphics[width=8cm]{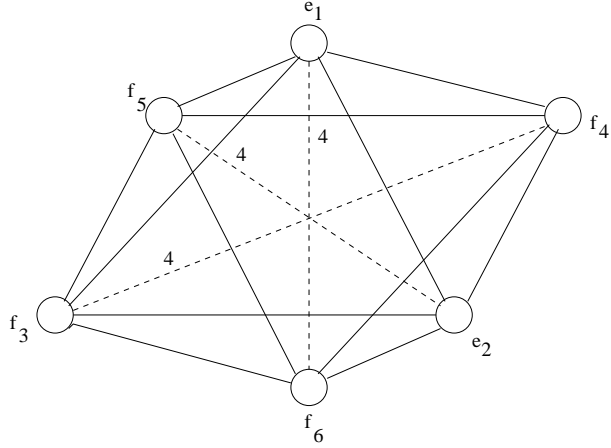}
\end{center}
\caption{The graph $\Gamma(P(\M^{(2)}))$ for
$U(6)\oplus A_2$.}
\label{u6a2}
\end{figure}

\medskip

Let $S=S_{4,28}=
\left\langle\begin{array}{rrrr}
 -2 & -1 & -1 & -1  \\
 -1  & 2 & 0 & 0 \\
 -1  & 0 & 2 & 0 \\
 -1  & 0 & 0 & 2
\end{array}\right\rangle$. For the standard basis $h$,
$e_1$, $e_2$, $e_3$, the set $P(\M^{(2)})$ consists of
$e_1$, $e_2$, $e_3$, $f_4=h-e_1$, $f_5=h-e_2$, $f_6=h-e_3$.
The lattice Weyl vector $\rho=h$. This case is anisotropic: the
polyhedron  $\M^{(2)}$ is compact, it has no vertices at infinity.
See Figure \ref{s4,28} for the Gram graph.

\begin{figure}
\begin{center}
\includegraphics[width=6cm]{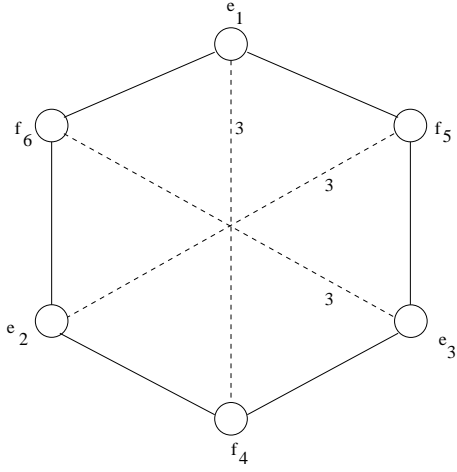}
\end{center}
\caption{The graph $\Gamma(P(\M^{(2)}))$ for
$S_{4,28}$.}
\label{s4,28}
\end{figure}

\medskip

Let us show that remaining three elliptically 2-reflective
hyperbolic lattices of rank 4 of Sec. \ref{subsec:clas.ell2-refl} don't
have a lattice Weyl vector.

\medskip

Let $S=U(3)\oplus 2A_1$. For the standard basis $c_1$, $c_2$, $e_1$, $e_2$,
the set $P(\M^{(2)})$ consists of
$e_1$, $e_2$, $f_3=c_1-e_1$, $f_4=c_1-e_2$, $f_5=c_2-e_1$, $f_6=c_2-e_2$,
$f_7=2c_1+2c_2-3e_1-2e_2$, $f_8=2c_1+2c_2-2e_1-3e_2$. These calculations
are important as itself.

Considering $\rho$ for first 6 these elements, one can see that
$\rho=-2c_1^\ast-2c_2^\ast-e_1^\ast-e_2^\ast$.
But, then $\rho\cdot f_7=-3$. Thus, $\rho$ does not exist.

\medskip

Let $S=\langle -4 \rangle\oplus \langle 4 \rangle \oplus A_2$.
For the standard basis $h$, $e$, $e_1$, $e_2$,
the set $P(\M^{(2)})$ consists of $e_1$, $e_2$,
$f_3=h-e-e_1-e_2$, $f_4=h+e-e_1-e_2$, $f_5=h-e_1-2e_2$,
$f_6=h-2e_1-e_2$. These calculations are important as itself.

Considering $\rho$ for first 4 these elements, one can see that
$\rho=(3/4)h-e_1-e_2$. But, then $\rho\cdot f_5=0$. Thus, $\rho$
does not exist.

\medskip

Let
$$
S=\left\langle\begin{array}{rrrr}
 -12 & -2 & 0 & 0   \\
  -2  & 2 & -1 & 0  \\
  0  & -1 & 2 & -1 \\
  0  & 0 & -1 & 2 \\
\end{array}\right\rangle\ .
$$
For the standard basis $h$, $e_1$, $e_2$, $e_3$, the set
$P(\M^{(2)})$ consists of $e_1$, $e_2$, $e_3$,
$f_4=h-2e_1-2e_2-e_3$, $f_5=h-2e_2-3e_3$, $f_6=h-e_1-3e_2-2e_3$,
$f_7=2h-2e_1-4e_2-5e_3$, $f_8=2h-3e_1-4e_2-4e_3$. This case is
anisotropic: the polyhedron  $\M^{(2)}$ is compact, it has no
vertices at infinity. These calculations
are important as itself.

Considering $\rho$ for first 4 these elements, one can see that
$\rho=(3h-3e_1-7e_2-6e_3)/5$. But, then $\rho\cdot f_6=0$. Thus, $\rho$
does not exist.

\medskip

{\it Remaining cases of $\rk S=3$.}

Firstly, let us consider isotropic cases.

Let $M_0=U\oplus A_1$
(the lattice $M_0=S_{3,2}$ in notations of Theorem \ref{th:el2refW}).
For the standard basis $c_1$, $c_2$ for $U$, and $b$ for $A_1$,
the set $P(\M^{(2)})$ is $a=-c_1+c_2$, $b$, $c=c_1-b$.
The lattice Weyl vector $\rho=3c_1+2c_2-b/2$ and $\rho^2=-23/2$.
The set  $P(\M^{(2)})$ has the Gram matrix $\Gamma(P(\M^{(2)}))$
equals to
\begin{equation}
A_{1,0}=
\left(\begin{array}{rrr}
2 & 0 & -1 \\
0 & 2 & -2 \\
-1&-2 & 2
\end{array}\right) \ ,
\label{A_{1,0}}
\end{equation}
and the Graph graph which is shown in Figure \ref{ua1}.

\begin{figure}
\begin{center}
\includegraphics[width=6cm]{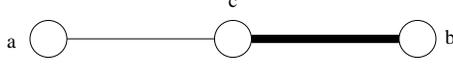}
\end{center}
\caption{The graph $\Gamma(P(\M^{(2)}))$ for
$U\oplus A_1$.}
\label{ua1}
\end{figure}

\medskip

All isotropic elliptically reflective hyperbolic lattices of rank 3
are sublattices of $M_0=U\oplus A_1$ of finite index
which are described in \cite{Nik5}. Let $[v_1,v_2,\dots,v_n]$ be a
sublattice generated by $v_1,\dots,v_n$. Let
$M_{k,l,m}=[ka,\ lb,$  $mc]\subset M_0$
where $k$, $l$, $m\in \bn$.

By \cite{Nik5}, up to the action of $W^{(2)}(M_0)=O^+(M_0)$,
all elliptically reflective sublattices
of $M_0$ are $M_{1,1,m}$,
$m=1$, $2$, $3$, $4$, $6$, $8$;
$M_{1,l,1}$, $l=2$, $3$, $4$, $5$, $6$, $9$;
$M_{k,1,1}$,   $k=4$, $5$, $6$, $7$, $8$, $10$, $12$;
$M_{2,1,2}$; $M_{4,1,2}$; $M_{6,1,2}$;
$M^\prime_{4,1,2}=[2a+c,b,2c]$; $M^\prime_{6,1,2}=[3a+c,b,2c]$
(24 sublattices). See \cite[Table 3]{Nik5}. For these sublattices,
only the following are isomorphic as lattices: $M_{4,1,1}\cong M_{2,1,2}$,
$M_{8,1,1}\cong M_{4,1,2}$, $M_{12,1,1}\cong M_{6,1,2}$,
$M_{6,1,1}\cong M^\prime_{6,1,2}$.
Thus, there are 20 isotropic non-isomorphic such lattices.
See \cite[Theorem 2.5]{Nik5}. The last isomorphism was missed in this
Theorem.
For all these 24 sublattices, the fundamental polygons $\M^{(2)}$ and
$P(\M^{(2)})$ are calculated in terms of $\Delta^{(2)}(M_0)$ in
\cite[Figures 5--10]{Nik5}. For $M^\prime_{6,1,2}=[3a+c,b,2c]$,
the correct polygon will be $PQRT_1$ in Figure 10 (see
\cite[Table 3]{Nik5}).

Using these results, one can find all these lattices which have
the lattice Weyl vector, and identify them with the isotropic lattices
of the rank three of Theorem \ref{th:el2refW}. Below, we do these
calculations.

The case $S=M_{1,1,1}=M_0=S_{3,2}$ was considered above.

Let $S=M_{1,1,2}$ (equals to $S_{3,8,a}=U(2)\oplus A_1$).
Then $P(\M^{(2)})$ is $b+2c$, $a$, $b$
with the Gram matrix
\begin{equation}
A_{2,0}=
\left(\begin{array}{rrrr}
 2 & -2 & -2 \\
-2 &  2 &  0 \\
-2 &  0 &  2
\end{array}\right)\ .
\label{A_{2,0}}
\end{equation}
The lattice Weyl vector is $\rho=a+(5/2)b+3c$ and $\rho^2=-7/2$.
This is equal to $S_{3,8,a}=U(2)\oplus A_1$. For its standard
basis $c_1$, $c_2$, $e$, the set $P(\M^{(2)})$ consists of
$e$, $c_1-e$, $c_2-e$ with the Gram matrix $A_{2,0}$, the
$\rho=c_1+c_2-e/2$ with $\rho^2=-7/2$. These lattices
are isomorphic since they have equal determinants, and their matrices
in their generators above are the same: $A_{2,0}$.

Let $S=M_{1,1,3}$. Then $P(\M^{(2)})$ is $a$, $b$, $2b+3c$,
$2a+3b+6c$. There is no the lattice Weyl vector.

Let $S=M_{1,1,4}$ (equals to
$S_{3,32,b}=\langle-8\rangle \oplus 2A_1$).
Then $P(\M^{(2)})$ is $b$, $3b+4c$, $a+2b+4c$, $a$.
with the Gram matrix
\begin{equation}
A_{2,I}=
\left(\begin{array}{rrrr}
 2 & -2 & -4 &  0 \\
-2 &  2 &  0 & -4 \\
-4 &  0 &  2 & -2 \\
 0 & -4 & -2 &  2
\end{array}\right)\ .
\label{A_{2,I}}
\end{equation}
The lattice Weyl vector is $\rho=(1/2)a+(3/2)b+2c$ and $\rho^2=-1$.
This is equal to $S_{3,32,b}=\langle-8\rangle \oplus 2A_1$. For its
standard basis $h$, $e_1$, $e_2$,
the set $P(\M^{(2)})$ consists of
$e_1$, $h-e_1-2e_2$, $h-2e_1-e_2$, $e_2$
with the Gram matrix $A_{2,I}$, the $\rho=(h-e_1-e_2)/2$ with
$\rho^2=-1$.

Let $S=M_{1,1,6}$. Then $P(\M^{(2)})$ is $a$, $b$, $5b+6c$,
$3a+16b+24c$, $4a+15b+24c$, $2a+3b+6c$.
There is no the lattice Weyl vector.

Let $S=M_{1,1,8}$ (equals to
$S_{3,128,b}=\langle-32 \rangle \oplus 2A_1$).
Then $P(\M^{(2)})$ is $a$, $3a+4b+8c$,
$4a+9b+16c$, $4a+15b+24c$, $3a+16b+24c$, $a+12b+16c$, $7b+8c$, $b$
with the Gram matrix
\begin{equation}
A_{2,III}=
\left(\begin{array}{rrrrrrrr}
 2  & -2 & -8  & -16 & -18 & -14 & -8  &   0 \\
-2  &  2 &  0  &  -8 & -14 & -18 & -16 &  -8 \\
-8  &  0 &  2  &  -2 &  -8 & -16 & -18 & -14 \\
-16 & -8 & -2  &   2 &   0 & -8  & -14 & -18 \\
-18 & -14& -8  &   0 &   2 & -2  & -8  &  -16 \\
-14 & -18& -16 &  -8 &  -2 &  2  & 0   &  -8 \\
 -8 & -16& -18 & -14 &  -8 &  0  & 2   &  -2 \\
  0 & -8 & -14 & -18 & -16 & -8  & -2  &   2 \\
\end{array}\right)\ .
\label{A_{2,III}}
\end{equation}
The lattice Weyl vector is $\rho=(1/4)a+b+(3/2)c$ with $\rho^2=-1/8$.
This is equal to $S_{3,128,b}=\langle-32\rangle \oplus 2A_1$. For its
standard basis $h$, $e_1$, $e_2$,
the set $P(\M^{(2)})$ consists of
$e_1$, $h-e_1-4e_2$, $2h-4e_1-7e_2$, $3h-8e_1-9e_2$, $3h-9e_1-8e_2$,
$2h-7e_1-4e_2$, $h-4e_1-e_2$, $e_2$
with the Gram matrix $A_{2,III}$, the $\rho=(3/16)h-e_1/2-e_2/2$ with
$\rho^2=-1/8$.

Let $S=M_{1,2,1}$ (equals to
$S_{3,8,b}=(\langle -24 \rangle\oplus A_2)[1/3,-1/3,1/3]$).
Then $P(\M^{(2)})$ is $c$, $2b+c$, $a$
with the Gram matrix
\begin{equation}
A_{1,I}=
\left(\begin{array}{rrr}
 2 & -2 & -1 \\
-2 &  2 & -1 \\
-1 & -1 &  2
\end{array}\right)\ .
\label{A_{1,I}}
\end{equation}
The lattice Weyl vector is $\rho=a+3b+3c$ with $\rho^2=-4$.
This is equal to
$S_{3,8,b}=(\langle -24 \rangle\oplus A_2)[1/3,-1/3,1/3]$.
For the standard basis $h$, $e_1$, $e_2$ of
$\langle -24 \rangle \oplus A_2$,
the set $P(\M^{(2)})$ is
$e_2$, $(h-4e_1-5e_2)/3$, $e_1$ with the Gram matrix $A_{1,I}$, the
$\rho=h/2-e_1-e_2$ with $\rho^2=-4$.

Let $S=M_{1,3,1}$ (equals to $S_{3,18}=U(3)\oplus A_1$).
Then $P(\M^{(2)})$ is $3b+2c$, $a$, $c$
with the Gram matrix
\begin{equation}
A_{3,0}=
\left(\begin{array}{rrr}
2  & -2 & -2 \\
-2 &  2 & -1 \\
-2 & -1 &  2
\end{array}\right)\ .
\label{A_{3,0}}
\end{equation}
The lattice Weyl vector is $\rho=(2/3)a+(5/2)b+(7/3)c$ with $\rho^2=-13/6$.
This is equal to $S_{3,18}=U(3)\oplus A_1$. For its standard
basis $c_1$, $c_2$, $e$, the set $P(\M^{(2)})$ is
$e$, $c_1-e$, $c_2-e$ with the Gram matrix $A_{3,0}$, the
$\rho=(2/3)c_1+(2/3)c_2-e/2$ with $\rho^2=-13/6$.

Let $S=M_{1,4,1}$. Then $P(\M^{(2)})$ is $c$, $a$, $3a+12b+8c$,
$4b+3c$. There is no the lattice Weyl vector.

Let $S=M_{1,5,1}$. Then $P(\M^{(2)})$ is $-5b-4c$, $5a+40b+29c$,
$16a+150b+111c$,  $4a+45b+34c$, $a+20b+16c$, $10b+9c$.
There is no the lattice Weyl vector.

Let $S=M_{1,6,1}$ (equals to $S_{3,72}=\langle -24 \rangle\oplus A_2$).
Then $P(\M^{(2)})$ is $a$, $a+6b+4c$, $6b+5c$,
$c$ with the Gram matrix
\begin{equation}
A_{3,I}=
\left(\begin{array}{rrrr}
 2 & -2 & -5 & -1 \\
-2 &  2 & -1 & -5 \\
-5 & -1 &  2 & -2 \\
-1 & -5 & -2 &  2
\end{array}\right)\ .
\label{A_{3,I}}
\end{equation}
The lattice Weyl vector is $\rho=(1/3)a+2b+(5/3)c$ with $\rho^2=-2/3$.
This is equal to
$S_{3,72}=\langle -24 \rangle\oplus A_2$.
For its standard basis $h$, $e_1$, $e_2$,
the set $P(\M^{(2)})$ is
$e_1$, $h-3e_1-4e_2$, $h-4e_1-3e_2$, $e_2$ with
the Gram matrix $A_{3,I}$, the  $\rho=h/3-e_1-e_2$ with $\rho^2=-2/3$.

Let $S=M_{1,9,1}$.
Then $P(\M^{(2)})$ is $c$, $a$, $2a+9b+6c$, $7a+54b+39c$,
$8a+72b+53c$,  $4a+45b+34c$, $5a+72b+56c$, $3a+54b+43c$, $9b+8c$.
There is no the lattice Weyl vector.

Let $S=M_{4,1,1}$ (equals to $S_{3,32,a}=U(4)\oplus A_1$).
Then $P(\M^{(2)})$ consists of $b$, $c$, $4a+3b+4c$
with the Gram matrix
\begin{equation}
A_{1,II}=
\left(\begin{array}{rrr}
 2 & -2 & -2 \\
-2 &  2 & -2 \\
-2 & -2 &  2
\end{array}\right)\ .
\label{A_{1,II}}
\end{equation}
The lattice Weyl vector is $\rho=2a+2b+(5/2)c$ with $\rho^2=-3/2$.
This is equal to $S_{3,32,a}=U(4)\oplus A_1$. For its standard
basis $c_1$, $c_2$, $e$, the set $P(\M^{(2)})$ is
$e$, $c_1-e$, $c_2-e$ with the Gram matrix $A_{1,II}$, the
$\rho=(c_1+c_2-e)/2$ with $\rho^2=-3/2$.

Let $S=M_{5,1,1}$.
Then $P(\M^{(2)})$ is $b$, $c$, $20a+15b+24c$, $5a+4b+5c$.
There is no the lattice Weyl vector.

Let $S=M_{6,1,1}$.
Then $P(\M^{(2)})$ is $b$, $c$, $12a+9b+14c$, $6a+5b+6c$.
There is no the lattice Weyl vector.

Let $S=M_{8,1,1}$ (equals to $S_{3,128,a}=U(8)\oplus A_1$).
Then $P(\M^{(2)})$ is $b$, $c$, $8a+6b+9c$, $8a+7b+8c$ with
the Gram matrix
\begin{equation}
A_{2,II}=
\left(\begin{array}{rrrr}
 2 & -2 & -6 & -2 \\
-2 &  2 & -2 & -6 \\
-6 & -2 &  2 & -2 \\
-2 & -6 & -2 &  2
\end{array}\right)\ .
\label{A_{2,II}}
\end{equation}
The lattice Weyl vector is $\rho=2a+(7/4)b+(9/4)c$ with $\rho^2=-1/2$.
This is equal to $S_{3,128,a}=U(8)\oplus A_1$. For its standard
basis $c_1$, $c_2$, $e$, the set $P(\M^{(2)})$ is
$e$, $c_1-e$, $c_1+c_2-3e$, $c_2-e$ with the Gram matrix $A_{2,II}$, the
$\rho=(c_1+c_2)/4-e/2$ with $\rho^2=-1/2$.

Let $S=M_{10,1,1}$.
Then $P(\M^{(2)})$ is $b$, $c$, $20a+15b+24c$, $30a+23b+34c$,
$60a+47b+66c$, $40a+32b+43c$, $40a+33b+42c$, $10a+9b+10c$.
There is no the lattice Weyl vector.

Let $S=M_{12,1,1}$ (equals to $S_{3,288}=U(12)\oplus A_1$).
Then $P(\M^{(2)})$ is $b$, $c$, $12a+9b+14c$, $24a+19b+26c$,
$24a+20b+25c$, $12a+11b+12c$
 with
the Gram matrix
\begin{equation}
A_{3,II}=
\left(\begin{array}{rrrrrr}
 2  & -2  & -10 & -14 & -10 &  -2 \\
-2  &  2  & -2  & -10 & -14 & -10 \\
-10 & -2  &  2 &   -2 & -10 & -14 \\
-14 & -10 & -2 &   2  &  -2 & -10 \\
-10 & -14 & -10&  -2  &   2 & -2  \\
-2  & -10 & -14&  -10 &  -2 &  2
\end{array}\right).
\label{A_{3,II}}
\end{equation}
The lattice Weyl vector is $\rho=2a+(5/3)b+(13/6)c$ with $\rho^2=-1/6$.
This is equal to $S_{3,228}=U(12)\oplus A_1$. For its standard
basis $c_1$, $c_2$, $e$, the set $P(\M^{(2)})$ is
$e$, $c_2-e$, $c_1+2c_2-5e$, $2c_1+2c_2-7e$,
$2c_1+c_2-5e$, $c_1-e$ with the Gram matrix $A_{3,II}$, the
$\rho=(c_1+c_2)/6-e/2$ with $\rho^2=-1/6$.

Let $S=M_{2,1,2}$.
Then $P(\M^{(2)})$ is $b$, $b+2c$, $2a+b+2c$
with the Gram matrix $A_{1,II}$. The lattice Weyl vector is
$\rho=a+(3/2)b+2c$ with $\rho^2=-1/6$. This case is isomorphic to
$M_{4,1,1}$ above.

Let $S=M_{4,1,2}$.
Then $P(\M^{(2)})$ is $b$, $b+2c$, $4a+3b+6c$, $4a+3b+4c$
with the Gram matrix $A_{2,II}$. The lattice Weyl vector is
$\rho=a+b+(3/2)c$ with $\rho^2=-1/2$. This case is isomorphic to
$M_{8,1,1}$ above.

Let $S=M_{6,1,2}$.
Then $P(\M^{(2)})$ is $b$, $b+2c$, $6a+5b+10c$, $12a+9b+16c$,
$12a+9b+14c$, $6a+5b+6c$ with the Gram matrix $A_{3,II}$.
The lattice Weyl vector is $\rho=a+(5/6)b+(4/3)c$ with $\rho^2=-1/6$.
This case is isomorphic to $M_{12,1,1}$ above.

Let $S=M^\prime_{4,1,2}=[2a+c,b,2c]$ (equals to $S_{3.32,c}=
U(8)[1/2,1/2]\oplus A_1$).
Then $P(\M^{(2)})$ is $b$, $b+2c$, $4a+3b+6c$, $4a+3b+4c$
with the Gram matrix $A_{2,II}$ (see \eqref{A_{2,II}}).
The lattice Weyl vector is $\rho=a+b+(3/2)c$ with $\rho^2=-1/2$.
These are the same as for
$M_{4,1,2}$, but $M_{4,1,2}\subset M^\prime_{4,1,2}$ is only a
sublattice of the index two. The lattice
$M^\prime_{4,1,2}$ is not generated
by its elements with square $2$.
This is equal to $S_{3,32,c}=U(8)[1/2,1/2]\oplus A_1$. For the standard
basis $c_1$, $c_2$, $e$, of $U(8)\oplus A_1$,
the set $P(\M^{(2)})$ is  $e$, $c_1-e$, $c_1+c_2-3e$, $c_2-e$
with the Gram matrix $A_{2,II}$, the
$\rho=(c_1+c_2)/4-e/2$ with $\rho^2=-1/2$.

Let $S=M^\prime_{6,1,2}=[3a+c,b,2c]$.
Then $P(\M^{(2)})$ is $b$, $b+2c$, $6a+5b+10c$, $12a+9b+16c$,
$3a+b+3c$. Their Gram matrix is the same as for $P(\M^{2})$
of the lattice $M_{6,1,1}$ above. Thus, these lattices are isomorphic,
and there are no the lattice Weyl vector.

Now, let us consider anisotropic cases. According to \cite{Nik5},
there are 6 anisotropic elliptically reflective lattices. For all of
them the sets $P(\M^{2})$ and their Gram matrices are calculated in
\cite{Nik5}. Using these calculations, one can find the lattice
Weyl vector $\rho$ or prove that it does not exist. We give these
calculations below. For Gram matrices below
we use notations $B_i$ from \cite{GN8}. .

Let $S=S_{3,12}=\langle -4 \rangle\oplus A_2$ (it is $S_5$ in notations
of \cite{Nik5}). For its standard basis
$h$, $e_1$, $e_2$, the set $P(\M^{(2)})$ is
$e_2$, $h-2e_1-e_2$, $h-e_1-2e_2$, $e_1$ with the Gram matrix
\begin{equation}
B_1=
\left(\begin{array}{rrrr}
 2 &  0 & -3 & -1 \\
 0 &  2 & -1 & -3 \\
-3 & -1 &  2 &  0 \\
-1 & -3 &  0 &  2
\end{array}\right)\ .
\label{B_1}
\end{equation}
The lattice Weyl vector is $\rho=h-e_1-e_2$ with $\rho^2=-2$.

Let $S=S_{3,24}=\langle -6\rangle \oplus 2A_1$
(it is $S_1$ in \cite{Nik5}).
For its standard basis $h$, $e_1$, $e_2$, the set
$P(\M^{(2)})$ is $e_1$, $e_2$, $h-2e_1$, $2h-3e_1-2e_2$,
$2h-2e_1-3e_2$, $h-2e_2$ with the Gram matrix
\begin{equation}
B_3=\left(\begin{array}{rrrrrr}
 2  &  0 & -4 & -6 & -4 &  0 \\
 0  &  2 &  0 & -4 & -6 & -4 \\
-4  &  0 &  2 &  0 & -4 & -6 \\
-6  & -4 &  0 &  2 &  0 & -4 \\
-4  & -6 & -4 &  0 &  2 &  0 \\
 0  & -4 & -6 & -4 &  0 &  2
\end{array}\right)\ .
\label{B_3}
\end{equation}
The lattice Weyl vector is $\rho=(h-e_1-e_2)/2$ with $\rho^2=-1/2$.

Let $S=S_{3,36}=\langle -12\rangle \oplus  A_2$ (it is $S_3$
in \cite{Nik5}).
For its standard basis $h$, $e_1$, $e_2$, the set
$P(\M^{(2)})$ is $e_1$, $e_2$, $h-3e_1-2e_2$, $h-2e_1-3e_2$
with the Gram matrix
\begin{equation}
B_2=
\left(\begin{array}{rrrr}
 2 & -1 & -4 & -1 \\
-1 &  2 & -1 & -4 \\
-4 & -1 &  2 & -1 \\
-1 & -4 & -1 &  2
\end{array}\right)\ .
\label{B_2}
\end{equation}
The lattice Weyl vector is $\rho=h/2-e_1-e_2$ with $\rho^2=-1$.

Let $S=S_{3,108}=\langle -36\rangle \oplus A_2$ (it is $S_2$ in
\cite{Nik5}).
For its standard basis $h$, $e_1$, $e_2$, the set
$P(\M^{(2)})$ is $e_1$, $e_2$, $h-5e_1-3e_2$, $2h-9e_1-8e_2$,
$2h-8e_1-9e_2$, $h-3e_1-5e_2$
with the Gram matrix
\begin{equation}
B_4=\left(\begin{array}{rrrrrr}
 2  &  -1 & -7 & -10 & -7 &  -1 \\
-1  &   2 & -1 & -7 & -10 &  -7 \\
-7  &  -1 &  2 & -1 & -7  & -10 \\
-10 & -7 & -1 &  2 & -1  &  -7 \\
 -7 & -10& -7 & -1 &  2 &   -1 \\
 -1 & -7 &-10 & -7 & -1   &  2
\end{array}\right)\ .
\label{B_4}
\end{equation}
The lattice Weyl vector is $\rho=h/4-e_1-e_2$ with $\rho^2=-1/4$.

Let $S=(\langle -60 \rangle \oplus A_2)[1/3,-1/3,1/3]$
(it is $S_4$ in \cite{Nik5}).
For the standard basis $h$, $e_1$, $e_2$ of
$\langle -60 \rangle \oplus A_2$, the set
$P(\M^{(2)})$ is $e_1$, $e_2$, $(h-7e_1-5e_2)/3$,
$(2h-8e_1-13e_2)/3$
with the Gram matrix $U_4$ of the lattice $S_4$ in
\cite[Theorem 1.2]{Nik5}. There is no the lattice Weyl vector.

Let $S=(\langle -132 \rangle \oplus A_2)[1/3,-1/3,1/3]$
(it is $S_6$ in \cite{Nik5}).
For the standard basis $h$, $e_1$, $e_2$ of
$\langle -132 \rangle \oplus A_2$, the set
$P(\M^{(2)})$ is $e_1$, $e_2$, $(h-10e_1-5e_2)/3$,
$(2h-17e_1-16e_2)/3$, $h-7e_1-9e_2$, $(2h-11e_1-19e_2)/3$
with the Gram matrix $U_6$ of the lattice $S_6$ in
\cite[Theorem 1.2]{Nik5}. There is no the lattice Weyl vector.

\medskip

These completes the proof of Theorem \ref{th:el2refW} with description of
the corresponding Gram matrices (equivalently, Gram graphs)
$\Gamma(P(\M^{(2)}))$, and the lattice Weyl vectors $\rho$.

\medskip

{\it The concluding remark.}

By \cite[Theorems 1.2.1, 1.3.1]{GN5}, there are the only two more
Gram matrices $A_{1,III}$ and $A_{3,III}$ for fundamental chambers $\M$
of reflection subgroups $W\subset W^{(2)}(S)$ of finite index
for elliptically $2$-reflective hyperbolic lattices $S$ of rank 3
with lattice Weyl vectors for $P(\M)$.
They are as follows.

\medskip

For the lattice $M_{1,1,3}$ above, let us take
\begin{equation}
P(\M)=\{2a+3b+6c,\ a+6b+9c,\ 5b+6c,\ b,\ a\}
\label{P(M)A_{1,III}}
\end{equation}
It has the Gram matrix
\begin{equation}
A_{1,III}=
\left(\begin{array}{rrrrr}
 2 & -2 & -6 & -6 & -2 \\
-2 &  2 &  0 & -6 & -7 \\
-6 &  0 &  2 & -2 & -6 \\
-6 & -6 & -2 &  2 &  0 \\
-2 & -7 & -6 &  0 &  2
\end{array}\right),
\label{A_{1,III}}
\end{equation}
and the lattice Weyl vector $\rho=(1/3)a+(7/6)b+(5/3)c$ with $\rho^2=-7/18$.
The polygon $\M$ is obtained from the described above polygon
$\M^{(2)}$ for $M_{1,1,3}$ by the reflection at
$2b+3c$. Thus, $[W^{(2)}(M_{1,1,3}):W]=2$.

\medskip

For the lattice $M_{1,6,1}$ above, let us take
$$
P(\M)=\{a,\,3a+12b+8c,\,5a+30b+21c,\,7a+54b+39c,\,8a+72b+53c,
$$
$$
8a+84b+63c,\,7a+84b+64c,\,5a+72b+56c,
$$
\begin{equation}
3a+54b+43c,\,a+30b+25c,\,12b+11c,\,c\}\ .
\label{P(M)A_{3,III}}
\end{equation}
It has the Gram matrix
$$
A_{3,III}=
$$
\begin{equation}
-\left(\begin{array}{rrrrrrrrrrrr}
 -2  & 2 & 11  & 25 & 37 & 47 & 50  &  46 & 37 & 23 & 11 & 1  \\
2  &  -2 &  1 &  11 & 23 & 37 & 46 &   50 & 47 & 37 & 25 & 11 \\
11 & 1 &  -2  &  2  & 11 & 25 & 37 &   47 & 50 & 46 & 37 & 23 \\
25 & 11& 2  & -2 &   1 & 11  & 23 &  37 & 46 & 50 & 47 & 37 \\
37 & 23& 11 &  1 &   -2 &  2  & 11 &  25 & 37 & 47 & 50 & 46 \\
47 & 37& 25 &  11&   2 &  -2  &  1 &  11 & 23 & 37 & 46 & 50 \\
50 & 46& 37 & 23 &  11 &  1  &   -2 &   2 & 11 & 25 & 37 & 47 \\
46 & 50& 47 & 37 &  25 &  11 &  2 &   -2 & 1  & 11 & 23 & 37 \\
37 & 47& 50 & 46 &  37 &  23 &  11&  1 &  -2  & 2  & 11 & 25 \\
23 & 37& 46 & 50 &  47 &  37 &  25&   11& 2  &  -2  & 1  & 11 \\
11 & 25& 37 & 47 &  50 &  46 &  37&   23& 11 & 1  &  -2  & 2  \\
1  & 11& 23 & 37 &  46 &  50 &  47&   37& 25 & 11 & 2  &  -2
\end{array}\right)
\end{equation}
and the lattice Weyl vector $\rho=(1/6)a+(7/4)b+(4/3)c$ with $\rho^2=-1/24$.
The polygon $\M$ is obtained from the described above polygon
$\M^{(2)}$ for $M_{1,6,1}$ by the group $D_3$ of the order 6 generated
by reflections in $a+6b+4c$ and $6b+5c$.
Thus, $W^{(2)}(M_{1,6,1})/W\cong D_3$.

\begin{remark}
\label{rem:K3finautgroup+rho}
{\rm By Remark \ref{rem:K3finautgroup},
ellptically $2$-reflective hyperbolic lattices $S$ with
lattice Weyl vector from the list of Theorem \ref{th:el2refW},
give all Picard lattices $S_X=S(-1)$ of
K3 surfaces $X$ over $\bc$ with finite automorphism group and
$\rk S_X\ge 3$ such that all non-singular rational curves $E$ on $X$
have the same degree $E\cdot h$ with respect to an
ample element $h=t\rho\in S_X$
for some $t>0$ from $\bq$ where $\rho\in S_X\otimes \bq$
is the lattice Weyl vector.

See Remark \ref{rem:K3discr} below about
their arithmetic mirror symmetric K3 surfaces.}
\end{remark}

\begin{remark}
\label{rem:finWeylvec}
{\rm Finiteness (or almost finiteness)
of the set of hyperbolic reflection groups
$W\subset W(S)$ of restricted arithmetic type
and $P(\M)$ for the fundamental chamber $\M$ of $W$
with lattice Weyl vector $\rho$ of elliptic ($\rho^2<0$),
parabolic ($\rho^2=0$) and hyperbolic ($\rho^2>0$) types
was proved in \cite{Nik8}, \cite{Nik10}.

For $\rk S=3$, such cases of elliptic type were classified by
D. Allcock in \cite{Al2}. }
\end{remark}

\subsection{Lorentzian Kac--Moody superalgebras\\
corresponding to lattices of Theorem \ref{th:el2refW}}
\label{subsec:KM1}

Let $S$ be one of lattices of Theorem \ref{th:el2refW},
$\M^{(2)}$ the fundamental chamber for $W^{(2)}(S)$, and
$P(\M^{(2)})$ the set of perpendicular vectors to $\M^{(2)}$
with square $2$. The Gram matrix (or the corresponding graph)
$\Gamma(P(\M^{(2)}))$ is a hyperbolic symmetric generalized Cartan matrix
$A(S)$ with the lattice Weyl vector $\rho(S)$, described in
Sec. \ref{subsec:grmatel2refW}.
In \cite{GN1} --- \cite{GN8}, for lattices $S$ of the rank three and with
the generalized Cartan matrices $A_{i,j}$, $i=1,\,2,\,3,\,4$,
$j=0,\,I,\,II$, and some other parabolic cases below,
we additionally constructed appropriate automorphic forms $\Phi(S)$
on appropriate IV type symmetric domains. Together $A(S)$ and $\Phi(S)$
defined the corresponding Lorentzian Kac--Moody Lie superalgebras
$\goth g (S)$ which are graded by the hyperbolic lattice $S$.

They are as follows.

$A_{1,0}$: The lattice is $S_{3,2}=U\oplus A_1$,
The automorphic form is $\Phi_{1,0,\overline{0}}=\Delta_{35}$.

\medskip

$A_{2,0}$: The lattice is
$S_{3,8,a}=U(2)\oplus A_1$.
The automorphic form is $\Phi_{2,0,\overline{0}}=\Delta_{11}$.

\medskip

$A_{3,0}$:
The lattice is $S_{3,18}=U(3)\oplus A_1$.
The automorphic form is $\Phi_{3,0,\overline{0}}=D_6\Delta_1$.

\medskip

$A_{4,0,\overline{0}}=A_{1,II}$:
The lattice is $S_{3,32,a}=U(4)\oplus A_1$.
The automorphic form is $\Phi_{4,0,\overline{0}}=\Delta_5^{(4)}$.

\medskip

$A_{1,I}$:
The lattice is $S_{3,8,b}=(\langle -24 \rangle \oplus A_2)[1/3,-1/3,1/3]$.
The automorphic form is
$\widetilde{\Phi}_{1,I,\overline{0}}=
\Phi_{1,0,\overline{0}}(Z)/\Phi_{1,II,\overline{0}}(2Z)=
\Delta_{35}(Z)/\Delta_{5}(2Z)$.

\medskip

$A_{2,I}$:
The lattice $S_{3,32,b}=\langle -8 \rangle\oplus 2A_1$.
The automorphic form is
$\widetilde{\Phi}_{2,I,\overline{0}}=
\Phi_{2,0,\overline{0}}(Z)/\Phi_{2,II,\overline{0}}(2Z)=
\Delta_{11}(Z)/\Delta_{2}(2Z)$.
\medskip

$A_{3,I}$:
The lattice is $S_{3,72}=\langle -24 \rangle \oplus A_2$.
The automorphic form is
$\widetilde{\Phi}_{3,I,\overline{0}}=
\Phi_{3,0,\overline{0}}(Z)/\Phi_{3,II,\overline{0}}(2Z)=
D_{6}(Z)\Delta_{1}(Z)/\Delta_{1}(2Z)$.

\medskip

$A_{4,I}=A_{2,II}$:
The lattice is $S_{3,128,a}=U(8)\oplus A_1$ (or $S_{3,32,c}=U(8)[1/2,1/2]$
$\oplus A_1$ which has the same
elements of square $2$).
The automorphic form is
$\widetilde{\Phi}_{4,I,\overline{0}}=
\Phi_{4,0,\overline{0}}(Z)/\Phi_{4,II,\overline{0}}(2Z)=\Delta_{5}^{(4)}(Z)/\Delta_{1/2}(2Z)$.

\medskip

$A_{1,II}$:
The lattice is $S_{3,32,a}=U(4)\oplus A_1$.
The automorphic form is $\Phi_{1,II,\overline{0}}=\Delta_5$.

\medskip

$A_{2,II}$: The lattice is
$S_{3,128,a}=U(8)\oplus A_1$.
The automorphic form is $\Phi_{2,II,\overline{0}}=\Delta_2$.

\medskip

$A_{3,II}$: The lattice is
$S_{3,288}=U(12)\oplus A_1$.
The automorphic form is $\Phi_{3,II,\overline{0}}=\Delta_1$.

\medskip

{\it Additional parabolic cases:}

\medskip

$A_{4,II}$ is $\Gamma(P(\M^{(2)}))$ for $U(16)\oplus A_1$.
The lattice is  $U(16)\oplus A_1$.
The automorphic form is $\Phi_{4,II,\overline{0}}=\Delta_{1/2}$.

\medskip

$A=\Gamma(P(\M^{(2)}))$ for $U\oplus \langle 4 \rangle$ (the $\M^{(2)}$
is infinite polygon with angles $\pi/2$). The lattice is $U\oplus \langle 4 \rangle$.
The automorphic form is $\Phi_{2,\overline{1}}=\Psi_{12}^{(2)}$.

\medskip

$A=\Gamma(P(\M^{(2)}))$ for $U\oplus \langle 6 \rangle$ (the $\M^{(2)}$
is infinite polygon with angles $\pi/3$). The lattice is $U\oplus \langle 6 \rangle$.
The automorphic form is $\Phi_{3,\overline{1}}=\Psi_{12}^{(3)}$.

$A=\Gamma(P(\M^{(2)})$ for $U\oplus \langle 8 \rangle$ (the $\M^{(2)}$
is infinite polygon with angles $0$). The lattice is $U\oplus \langle 8 \rangle$.
The automorphic form is $\Phi_{4,\overline{1}}=\Psi_{12}^{(4)}$.

\medskip

Using these basic automorphic forms, in \cite{GN1} --- \cite{GN8},
we constructed many other Lorentzian Kac--Moody superalgebras.
Roughly speaking, they are obtained by products and quotients of
these basic forms.

We want to extend these results to other lattices of Theorem
\ref{th:el2refW}, especially to the higher ranks $\ge 4$.
In Sections \ref{pullback}--\ref{seriesD8},
we construct $2$-reflective automorphic forms for 2-reflective
hyperbolic lattices of Theorem \ref{th:el2refW}.
Using different base functions, we get six series of such
automorphic forms.
\smallskip

1) For the lattices $U\oplus K$,
$$
K=A_1,\,2A_1,A_2;
\ 3A_1,\, A_3;\
4A_1,\,2A_2,\,A_4,\,D_4;
\ A_5, D_5;
$$
$$
3A_2,\,2A_3,\,A_6,\,D_6,\,E_6;\
A_7,\,D_7,\,E_7; \ 2D_4,\,D_8,\,E_8,\,2E_8
$$
and $U(2)\oplus 2D_{4}$,
it is done in Theorem \ref{thm-Phi}
and Theorem \ref{thm-affineKM}.
\smallskip

2) For the lattices $\langle -2 \rangle \oplus kA_1$, $2\le k\le 9$ (the case
$k=9$ is parabolic), it is done in Theorem \ref{LiftD}.
\smallskip

3) For the lattices $U(4)\oplus kA_1$, $1\le k\le 4$
(the case $k=4$ is parabolic), it is done
in Theorem \ref{U(4)pluskA_1}.
\smallskip

4) For the lattices $U(3)\oplus A_2$, $U(3)\oplus 2A_2$, $U(3)\oplus 3A_2$
(the last case is parabolic), it is done in Theorem \ref{lift3A2}.
\smallskip

5) For the lattices $U(2)\oplus D_4$ and $U(4)\oplus D_4$,
it is done in Theorem \ref{U(2)D4} and Theorem \ref{Thm-U(4)+D_4}.
\smallskip

6) For the $2$-reflective lattices of parabolic type
$U\oplus K$,
$$
K=A_1(2),\ A_1(3),\ A_1(4),\ D_2(2),\
A_2(2),\ A_2(3),\ A_3(2),\  D_4(2),\ E_8(2),
$$
it is done in Theorem \ref{thmLeech}.

\section{The strongly reflective modular forms}
\label{pullback}

Lorentzian Kac-Moody algebras give automorphic corrections of hyperbolic
Kac-Moody algebras  since their Kac-Weyl-Borcherds denominator functions
are automorphic forms with respect to arithmetic orthogonal
groups of signature $(n,2)$ (see Section \ref{sec:data1-5}). Here we
give the  general set-up for
construction of corresponding  automorphic  forms which we shortly call
as {\it automorphic corrections} of the hyperbolic root systems.
We note that the signature $(2,n)$ is  usually used in algebraic geometry
and the theory of automorphic forms. The signatures  $(n,0)$,
$(n,1)$ and $(n,2)$ are natural in the theory of Lie algebras.

Let $T$ be an integral lattice with a
quadratic form of signature $(n,2)$ and let
\begin{equation}\label{DL}
\Omega (T)=\{[Z] \in \PP(T\otimes \CC) \mid
  (Z, Z)=0,\ (Z,\Bar Z)<0\}^+
\end{equation}
be the associated $n$-dimensional Hermitian domain of type IV
(here $+$  denotes one of its two connected components)
and $\Omega(T)^\bullet$ its affine cone.
We denote by $\Orth^+(T)$ the
index $2$ subgroup of the integral orthogonal group $\Orth(T)$
preserving $\Omega (T)$.

\begin{definition}
Suppose that $T$ has signature $(n,2)$, with $n\ge 3$.
Let $k\in \ZZ$ and let $\chi \colon \Gamma\to \CC^*$
be a character of a subgroup $\Gamma\subset \Orth^+(T)$ of finite index.
A holomorphic function $F\colon \Omega(T)^\bullet\to \CC$ on the affine
cone $\Omega (T)^\bullet$ over $\Omega(T)$ is called a \emph{modular form}
of \emph{weight} $k$ and character $\chi$ for the group $\Gamma$ if
\begin{equation*}
F(tZ)=t^{-k}F(Z)\quad \forall\,t\in \CC^*,
\end{equation*}
\begin{equation*}
F(gZ)=\chi(g)F(Z)\quad  \forall\,g\in \Gamma.
\end{equation*}
A modular form is called a \emph{cusp form}
if it vanishes at every cusp.
\end{definition}

We denote the linear spaces of modular and cusp forms of weight $k$
and character $\chi$ by $M_k(\Gamma,\chi)$ and $S_k(\Gamma,\chi)$
respectively.
We recall that a cusp is defined by  an isotropic line or plane in $T$.
For applications, one of the most important  subgroups of $\Orth^+(T)$ is
the stable orthogonal group
\begin{equation}\label{def-TO}
\Tilde{\Orth}^+(T) =\{g\in \Orth^+(T)\mid g|_{T^*/T}=\id\}
\end{equation}
where $T^*$ is the dual lattice of $T$.

For any $v\in L\otimes \QQ$ such that $v^2=(v,v)>0$ we define the
\emph{rational quadratic divisor}
\begin{equation}\label{div-Dv}
\cD_v=\cD_v(T)=\{[Z] \in \Omega (T)\mid (Z,v)=0\}\cong \Omega (v^\perp_T)
\end{equation}
where $v^\perp_T$ is an even integral lattice of signature $(n-1,2)$.
Therefore, $\cD_v$ is also a homogeneous domain of type~IV.  We note
that $\cD_v(T)=\cD_{t v}(T)$ for any $t\ne 0$.  The theory of
automorphic Borcherds products (see \cite{B4}--\cite{B5}
and \cite{GN6}, \cite{CG1}, \cite{G12} for the Jacobi variant of these products)
gives a method of constructing automorphic forms with rational quadratic divisors.

The reflection with respect to the hyperplane defined by a
non-isotropic vector $v\in T^*$ is given by
\begin{equation}\label{sigma_r}
\sigma_v\colon l\Mapsto l-\frac{2(l,v)}{(v,v)}v.
\end{equation}
If $v\in T^*$ and $(v,v)>0$, the divisor $\cD_v(T)$ is called a
\emph{reflective divisor} if $\sigma_v\in \Orth(T)$.
In what follows we consider the  divisor of a modular form $F$ as a divisor of
$\Omega (T)$ since $F$ is homogeneous on $\Omega (T)^\bullet$.

\begin{definition}
\label{reflform}
A modular form $F\in M_k(\Gamma,\chi)$ is called {\bf reflective} if
\begin{equation}\label{eq-div-rmf}
\supp(\divv_{\Omega (T)} F) \subset  \bigcup_{\substack{\pm v\in T \vspace{1\jot}\\
 v \ {\rm is \ primitive}\vspace{1\jot}\\
\sigma_v\in \Gamma\text{ or }-\sigma_v\in \Gamma}} \cD_v(T).
\end{equation}
We call $F$  $2$-reflective if all $v$  above are of square $2$.
We call $F$  strongly reflective if multiplicity of any irreducible component
of $\divv F$ is equal to one.
We say that a strongly reflective modular form  $F$ is
 a modular form with the complete  $2$-divisor if
\begin{equation}\label{eq-divPhi}
\divv_{\Omega (T)} F = \Sum_{\substack{v\in R_{2}(T)/\{\pm 1\} }}
\cD_v(T)
\end{equation}
where $R_{2}(T)$ is the set of $2$-vectors (roots) in $T$.
\end{definition}

Our main goal is to construct {\it strongly reflective}
modular forms with the {\it complete  $2$-divisor}
related to the hyperbolic root systems described in Section \ref{sec:data1-3}.
\medskip

\noindent
\begin{example} \label{borchform}
{\it The Borcherds modular form $\Phi_{12}$}
{\rm (see \cite{B4}).
This is the unique, up to a constant, modular form of
the singular (i.e. the minimal possible) weight $12$ and  character
$\det$ with respect to
$\Orth^+(II_ {26,2})$
$$
\Phi_{12}\in M_{12}(\Orth^+(II_ {26,2}), \det)
$$
where $II_ {26,2}$ is the unique (up to an isomorphism) even
unimodular lattice of signature $(26,2)$.
It was proved in \cite{B4}   that
$$
\divv_{\Omega (II_ {26,2})} \Phi_{12} =
\Sum_{\substack{v\in R_{2}(II_ {26,2})/\{\pm 1\}}} \cD_v(II_ {26,2}).
$$
We note that  all  $2$-vectors in  $II_ {26,2}$
form only one orbit with respect to $\Orth^+(II_ {26,2})$.}
\end{example}

\begin{example} \label{Siegelmodform}
{\rm If
$$
T_{2t}^{(5)}=2U\oplus \latt{2t}\qquad
{\rm where}\quad
U\cong
\begin{pmatrix}
\ \, \,0&-1\\ -1&\ \,\,0
\end{pmatrix}, \quad t\in \mathbb N,
$$
of signature $(3,2)$,
then the modular forms with respect to $\Tilde{\SO}^{+}(T_{2t}^{(5)})$
coincide with Siegel modular forms of genus two with respect to the paramodular
group $\Gamma_t\subset \Sp_2(\QQ)$ (see \cite{G2}, \cite{GN6}).
In particular, if $t=1$ we obtain the Siegel
modular forms with respect to $\Sp_2(\ZZ)$.
A large well defined class of strongly reflective modular forms
for $\Gamma_t$ was described
in \cite{GN1}--\cite{GN8}. See also \cite{CG1} where
all Siegel modular forms with the simplest diagonal divisor
were classified for all Hecke congruence subgroups of all paramodular groups.}
\end{example}

All reflective  modular forms  have a Borcherds product expansion.
It follows from the results of  J.H. Bruinier who proved existence of a
Borcherds product expansion  for modular forms with a divisor
which is sum of rational quadratic divisors if the lattice is not very exotic
(see \cite{Bru}).
To construct strongly reflective  modular forms for the reflective hyperbolic lattices
with a lattice Weyl vector we use the method of quasi pull-back of the Borcherds form $\Phi_{12}$ which was proposed in \cite[pp. 200--201]{B1}.
It  was successfully applied to the theory of
moduli spaces in \cite{BKPS},
\cite{GHS1}--\cite{GHS4}. See \cite[\S 8]{GHS4}
on the detailed description of this construction.

The statements of the next theorem were  proved in
\cite[Theorem 1.2]{BKPS} and  \cite[Theorems 8.3 and 8.18]{GHS4}.

\begin{theorem}\label{thm:qpb}
Let $T\emb II_{26,2}$ be a primitive  sublattice of
signature $(n,2)$, $n\ge 3$, and let $\Omega (T)\emb\Omega (II_{26,2})$ be
the corresponding embedding of the homogeneous domains.
The set of $2$-roots
\begin{equation*}
R_{2}(T^\perp)=\{v\in II_{26,2}\mid v^2=2,\ (v, T)=0\}
\end{equation*}
in the orthogonal complement is finite.
We put $N(T^\perp)=\# R_{2}(T^\perp)/2$. Then the function
\begin{equation}\label{qpb}
  \left. \Phi_{12}|_T=
    \frac{\Phi_{12}(Z)}{
      \prod_{v\in R_{2}(T^\perp)/{\pm 1}} (Z, v)}
    \ \right\vert_{\Omega (T)^\bullet}
  \in M_{12+N(T^\perp)}(\Tilde{\Orth}^+(T),\, \det),
\end{equation}
where in the product over $v$ we fix a finite system of
representatives in $R_{2}(T^\perp)/{\pm 1}$.  The modular form
$\Phi_{12}|_T$ vanishes only on rational quadratic divisors of type
$\cD_u(T)$ where $u\in T^*$ is the orthogonal projection
of a $2$-root $r\in II_{26,2}$ to $T^*$ satisfying  $0<(u,u)\le 2$.
If the set $R_{2}(T^\perp)$ of $2$-roots in $T^\perp$ is non-empty
then the quasi pull-back
$\Phi_{12}|_T\in S_{12+N(T^\perp)}(\Tilde{\Orth}^+(T),\, \det)$
is a cusp form.
\end{theorem}

In \cite{G12} we proposed  twenty four Jacobi type constructions
of the Bor\-cherds function $\Phi_{12}$ based on the twenty four
one dimensional  boundary \linebreak
components of the Baily-Borel
compactification of the modular variety \linebreak
$\Orth^+(II_{26,2})\setminus \Omega (II_{26,2})$.
These components correspond exactly to the  classes of
positive definite  even unimodular lattices of rank $24$.
They are the $23$ Niemeier lattices $N(R)$ uniquely determined by their root
sublattices  $R$ of rank $24$
\begin{gather*}
3E_8,\ E_8\oplus D_{16},\  D_{24},\  2D_{12},\  3D_8,\
4D_6,\ 6D_4,\\
A_{24},\  2A_{12},\  3A_8,\  4A_6,\  6A_4,\  8A_3,\  12A_2,\  24A_1,\\
E_7\oplus A_{17},\ 2E_7\oplus D_{10}, \ 4E_6,\ E_6\oplus D_{7}\oplus A_{11},\\
A_{15}\oplus D_{9},\ 2A_{9}\oplus D_{6},\ 2A_{7}\oplus D_{5},\ 4A_{5}\oplus D_{4}
\end{gather*}
and the Leech lattice $\Lambda_{24}=N(\emptyset)$ without roots
(see \cite[Chapter 18]{CS}). We note that
$II_{26,2}\cong 2U\oplus N(R)$.
The quasi pull-backs
of $\Phi_{12}$ considered in the  different
one-dimensional boundary components give the first series
of strongly reflective modular forms which determine
the Lorentzian Kac-Moody algebras of some reflective lattices
considered in Sect. \ref{sec:data1-3}.

The next theorem is a particular case of a more general result
proved in \cite{G16} and a generalisation of
\cite[Theorem 3.4] {GH}.

\begin{theorem}\label{Tmain}
Let $K$ be a primitive sublattice of $N(R)$ containing  a direct  summand of
the same rank of a root lattice $R$
of a Niemeier lattice $N(R)$ or a primitive sublattice of the Leech lattice
$N(\emptyset)=\Lambda_{24}$. We assume that
$K$ satisfies the following condition:
$$
({\rm Norm}_2)\qquad \forall\, \bar c \in K^*/K\
\ (\bar c^2\not\equiv 0\  \mathrm{mod}\  2\ZZ)\ \
\exists\, h_c\in \bar c\,:\,0<h_c^2<2.
$$
We consider  $T=2U\oplus K $ as a sublattice of the
corresponding model of $II_ {26,2}=2U\oplus N(R)$.
Then  $\Phi_{12}|_T$ is a strongly reflective
modular form with the complete $2$-divisor. More exactly
$$
\Phi_{12}|_T\in M_{k}(\Tilde{\Orth}^+(T),\det)
$$
where $k=12+|R_{2}(K^\perp)|/2$ and
\begin{equation}\label{div-formula}
\divv \Phi_{12}|_T=
\sum _{v\in R_{2}(T)/{\pm 1}} \cD_v(T).
\end{equation}
\end{theorem}

\begin{remark} \label{remarkNorm}
{\rm In  the discriminant group $A_K=K^*/K$, if
$h\in \bar c\in K^*/K$
then $(h,h)\equiv (\bar c, \bar c)=q_K(\bar c)\mod 2\ZZ$
is well defined modulo  $2$.
The condition (${\rm Norm}_2$) claims that there exists
an element $h_c$ in every $\bar c$ with the smallest possible norm.}
\end{remark}
\begin{proof}
The quasi pull-back $\Phi_{12}|_T$ is a modular form with respect
to the character $\det$. For any $2$-vector $v\in T$
the reflection $\sigma_v$ is in $\Tilde{\Orth}^+(T)$.
Therefore $\Phi_{12}|_T$ vanishes on the walls of all
$2$-reflections in $T$.

For any $2$-vector $v\in II_ {26,2}$  we write
$v=\alpha+\beta$ where
$$
\alpha=\pr_{T^*}(v)\in T^*,
\ \beta\in (T^\perp)^*=(K^\perp_{N} )^*\quad{\rm and }\quad
\alpha^2+\beta^2=2,\quad  \beta^2\ge 0.
$$
Then  we have
\begin{equation*}
\Omega (T)\cap\cD_v(II_ {26,2})=
\begin{cases}
\cD_\alpha(T),&\text{ if }\  \alpha^2>0, \\
\emptyset,&\text{ if }\  \alpha^2\le  0,\ \alpha\ne 0,\\
\Omega (T),&\text{ if }\  \alpha=0, \text{ i.e.\ } v\in T^\perp.
\end{cases}
\end{equation*}
We note that if $\beta^2=0$ then $v\in T$ because
$T$ is primitive in $II_ {26,2}$. In this case, we get
the divisor $\cD_v(T)$ in $\Omega (T)$.

Let $0< \alpha^2<2$.
Since $K$ satisfies $(\rm{Norm}_2)$-condition
and $K^*/K =T^*/T$, there exists
$h\in K^*$ such that  $h\in \alpha+K$ and $h^2=\alpha^2$.
We have $h+\beta\in v+K\subset II_ {26,2}$.
Therefore,
$$
h+\beta\in \bigl(K^* \oplus (K^\perp_{N} )^*\bigr)
\cap \bigl(2U\oplus N(R) \bigr)=N(R)
$$
and $(h+\beta)^2=2$.
It follows that $h+\beta$ is a $2$-root in $N(R) $
which does not belong to $K \oplus K^\perp_N $.
This contradicts to the condition on the roots in
$K\subset N(R)$.
\end{proof}

In order to apply the last theorem, we have to fix
models of irreducible $2$-roots lattices $R$.
\begin{equation}\label{Dn}
D_n=\{\sum_{i=1}^n x_ie_i\in \bigoplus_{i=1}^n\ZZ e_i
=\ZZ^n\ |\, x_1+\dots+x_n\in 2\ZZ,\   (e_i,e_j)=\delta_{i,j}\}
\end{equation}
is the maximal even sublattice of the odd unimodular lattice
$\mathbb Z^n$.
Then
$$
A_n=\{\Sum_{i=1}^{n+1} x_ie_i\in \ZZ^{n+1}\,|\, x_1+\dots+x_{n+1}=0\}
\subset D_{n+1}.
$$
In particular,
$A_1\cong \latt{2}$,
$A_1\oplus A_1\cong D_2$ and  $A_3\cong  D_3$. We note that
$D_1\cong \latt{4}=A_1(2)$ is not a root lattice.
$A_n=(1,\dots,1)^\perp_{\ZZ^{n+1}}$, therefore  $A_n^*/A_n$ is the  cyclic group
$C_{n+1}$ of order $n+1$. $D_n^*/D_n$
is isomorphic to $C_4$ for odd $n$ and to
$C_2\times C_2$  for even $n$.
The classes of the discriminant  group  $R^*/R$ of these root lattices
are generated by the following elements having
{\it the minimal  possible norm} in the corresponding classes modulo $A_n$ or $D_n$:
$$
D_n^*/D_n=
\{\,0,\  e_n,\  (e_1+\dots+e_n)/2,\   (e_1+\dots+e_{n-1}-e_n)/2 \,\}+D_n,
$$
\begin{multline*}
A_n^*/A_n=\\
\{\,
\varepsilon_i=\frac{1}{n+1}
(\ \underbrace{\,i,\dots, i,\,}_{\text{$n+1-i$}}\
\underbrace{\,i-n-1,\dots,i-n-1\,}_{\text{$i$}}\ ),\ 1\le i\le n+1\,\}+A_n.
\end{multline*}
Then the nontrivial classes of $D_n^*/D_n$ have representatives of norm
$1$ and $\frac{n}4$. For $A_n$ we see  that
if $n\le 7$ then $(\varepsilon_i, \varepsilon_i)\le 2$
and  $(\varepsilon_i, \varepsilon_i)=2$ only for $n=7$ and $i=4$.
\smallskip

\begin{example} \label{lattN8}
{\rm The following even lattice of determinant $2^6$
is important for $K3$ surfaces with symplectic involutions:
\begin{equation}\label{N8}
N_8=\latt {8A_1, h=(a_1+...+a_8)/2} \cong D_8^*(2),\quad
(a_i,a_j)=2\delta_{i,j},\ h^2=4
\end{equation}
which is usually called {\it Nikulin's lattice}.
Then $\bar h= h+8A_1$ is an isotropic element of the discriminant
group $(8A_1)^*/(8A_1)$
$$
8A_1\subset N_8\subset N_8^*\subset 8A_1^*,\quad \det N_8=2^6
\quad{\rm and}\quad  N_8^*/N_8\cong  {\bar h}^\perp/\bar h.
$$
From the last representation of $N_8^*/N_8$ it  follows that the lattice $N_8$
satisfies the condition $({\rm Norm}_2)$ of Theorem \ref{Tmain}. }
\end{example}

\begin{theorem}\label{thm-Phi}
For the lattice $T=2U\oplus K $ where $K$ is one
of the following $24$ lattices
$$
{}\quad A_1,\ 2A_1,\ 3A_1,\ 4A_1,\ N_8;\ \ A_2,\ 2A_2,\ 3A_2;\ \ A_3,\ 2A_3;\ \
A_4,\ A_5,\ A_6,\  A_7;
$$
$$
k=35,\ 34,\ \ 33,\  \ \ 32,\ \ \ 28; \ \ \ 45,\ \  42,\ \ \  39;\ \ \  \ 54,\ \  48;\ \
\ \ 62,\ \   69,\ \  75,\ \  80;
$$
$$
{}\quad\ \  \ D_4,\ \ 2D_4,\ D_5,\ D_6,\ D_7,\ D_8;\ \ E_6,\ \ E_7,\ \ E_8,\ \ 2E_8;
$$
$$
k=72,\ \ \ 60,\ \ \ 88,\ 102,\ 114,\ 124;\ 120,\ 165,\ 252,\ 132
$$
there exists a strongly reflective modular form $$
\Phi_{k, K}=\Phi_{12}|_{2U\oplus K \emb 2U\oplus N_K }
\in S_k(\Tilde\Orth^+(2U\oplus K ), \det)
$$
with the complete $2$-divisor where $N_K$ is the Niemeier lattice
whose root system $R$ contains $K$ as a direct summand. All  these functions
$\Phi_{k, K}$
are cusp forms.
\end{theorem}

The K\"ocher principle together with the standard  divisors argument
give the next result.

\begin{corollary}
For the lattices $K$ of Theorem \ref{thm-Phi}
the reflective modular form $\Phi_{k,K}$ is the only, up to a constant,
cusp form in $S_k(\Tilde\Orth^+(K), \det)$. In particular,
$\Phi_{k,K}$ is  a (new) eigenfunction of all Hecke operators
acting on modular forms.
\label{Hecke1}
\end{corollary}

\noindent
\begin{remark} \label{rem1and2}
{\rm 1) The cusp forms in Corollary \ref{Hecke1}
are generalisations of the Ramanujan delta-function $\Delta(\tau)$.
The strongly reflective modular form for
$K=A_1$ is, in fact, the Igusa modular form
$\Delta_{35}\in S_{35}(\Sp_2(\ZZ))$ which is the only genus $2$ Siegel modular form
of odd weight up to a factor.
The corresponding Lorentzian Kac-Moody algebra, the algebra
with the smallest possible Cartan matrix, was defined
in \cite{GN4}. We can say that  all $2$-reflective modular forms
of Theorem \ref{thm-Phi} are of $\Delta_{35}$-type.

2) In the next section we show that all quasi pull-backs of Theorem \ref{thm-Phi}
have integral Fourier coefficients and we show how to describe their
Borcherds products and multiplicities of imaginary roots of
the corresponding Lorent\-zian Kac-Moody algebras. }
\end{remark}

\begin{proof}
Any root lattice $K$,  mentioned in Theorem \ref{thm-Phi}, satisfies
$({\rm Norm}_2)$ of Theorem \ref{Tmain}.
This follows from the description of the discriminant forms of
the irreducible root lattices given above.
Moreover there exists a Niemeier lattice $N(R)$ such that $K$
is a direct summand of the root lattice $R$.
The lattice $4A_1$, $3A_2$, $2A_3$, $A_7$, $2D_4$ and $D_8$ are not
maximal but their even extensions contain a new root which is not
possible in $N(R)$. Therefore the embedding $K\emb N(R)$ is
primitive sublattice for all root lattices $K$ of the theorem.

The lattice $N_8\emb N(24A_1)$ is the extension of $8A_1$ in
$N(24A_1)$ by one octave of the Golay code. Therefore, $N_8$ is primitive
in $N(24A_1)$ and  it satisfies the condition $({\rm Norm}_2)$.

The  quasi pull-back $\Phi_{12}|_{2U\oplus K }$
is a strongly reflective modular form with the complete
$2$-divisor according to  Theorem \ref{Tmain}.
All these functions are cusp forms according to
\cite[Theorem 8.18]{GHS4}.
\end{proof}

The last theorem gives a nice example  of two different automorphic corrections
of the same hyperbolic root system.
\begin{proposition}\label{Pr-2-forms}
The hyperbolic Kac-Moody algebra defined  by the $2$-root system
of the lattice $U\oplus D_4 $
has two different automorphic corrections, i.e. there are two
$2$-reflective modular forms with this lattice as the hyperbolic
lattice of a zero dimensional cusp.
\end{proposition}
\begin{proof}
These two modular forms  are related to the different non isomorphic
extensions $U\oplus(U\oplus D_4)$ and $U(2)\oplus(U\oplus D_4 )$
of the given hyperbolic lattice.
The first Lorentzian Kac-Moody algebra is defined by  the modular form
from Theorem \ref{thm-Phi}
$$
\Phi_{72, D_4}=\Phi_{12}|_{2U\oplus D_4 \emb 2U\oplus N(6D_4 )}.
$$
To define the second  Lorentzian Kac-Moody algebra  we use the isomorphism
$$
U\oplus N_8 \cong U(2)\oplus D_4 \oplus D_4 .
$$
(These lattices are indefinite, $2$-elementary and have isomorphic  discriminant forms, see \cite{Nik2}.)
We consider the embedding
$$
U(2)\oplus U\oplus D_4 \emb \bigl(U(2)\oplus U\oplus D_4 \bigr)\oplus D_4
\cong 2U\oplus N_8 .
$$
The arguments of the proof of Theorem \ref{Tmain} show that
\begin{equation}\label{F2U(2)+D4}
\Phi_{40,\, U\oplus U(2)\oplus D_4}=
\Phi_{28}^{(N_8)}|_{U(2)\oplus U\oplus D_4 \emb 2U\oplus N_8 }
\end{equation}
is a strongly $2$-reflective modular form of weight $40$  from the space
of cusp forms
$S_{40}(\Tilde\Orth^+(U(2)\oplus U\oplus D_4 ), \det)$.
\end{proof}

In Theorem \ref{thm-Phi}, we used $23$ Niemeier lattices with non-trivial root systems.
We can construct nine strongly reflective modular forms
using the Leech lattice. The next result  is proved in  \cite{G16}.
\begin{theorem}\label{thmLeech}
For the lattice $T=2U\oplus K$, where $K$ is one
of the following nine sublattices  of the Leech
lattice $\Lambda_{24}$
$$
A_1(2),\ A_1(3),\ A_1(4),\ D_2(2)=\latt{4}\oplus\latt{4},\
A_2(2),\ A_2(3),\ A_3(2),\  D_4(2),\ E_8(2)
$$
the corresponding pull-back
for $T=2U\oplus K \emb 2U\oplus \Lambda_{24} $
$$
\Phi_{12}|_{T}=P_{12}\in M_{12}(\Tilde\Orth^+(T), \det)
$$
is a strongly reflective (non cusp)  modular  form
of weight $12$ with the complete $2$-divisor.
\end{theorem}

\begin{remark} \label{remarkA1(2)A1(4)}
{\rm The reflective Siegel modular forms
corresponding to the lattices $K=A_1(2)$, $A_1(3)$ and  $A_1(4)$
were constructed in  \cite[\S 4.2]{GN4} by another method.
Theorem \ref{thmLeech} can be considered as a generalisation
of the fact indicated in \cite[\S 4.2, Remark 4.4]{GN4}. }
\end{remark}

\smallskip

The last theorem implies

\begin{corollary}\label{parabolic}
Let $K$ be one of the positive definite lattices
of  Theorem \ref{thmLeech}. Then the $2$-root system of the hyperbolic lattice
$U\oplus K $ is reflective  with a lattice Weyl vector of norm $0$,
i.e. it has parabolic type.
\end{corollary}
We plan to consider in more details the corresponding Lorentzian Kac-Moody algebras
in a separate paper.

In this section  we constructed
$33$ reflective modular forms with respect to the groups of type
$\Tilde\Orth^+(2U\oplus K )$ and one modular form for
$\Tilde\Orth^+(U\oplus U(2)\oplus D_4 )$. (We note that in many cases
$\Tilde\Orth^+(2U\oplus K )$ is not the  maximal modular group of the reflective
modular form.)
The main theorem of \cite{GH} gives a result
on the geometric type of the corresponding modular varieties.
\begin{corollary}\label{uniruled}
For all $34$ lattices $T$ of signature $(n,2)$ from Theorem \ref{thm-Phi},
Proposition \ref{Pr-2-forms} and Theorem \ref{thmLeech}, the modular variety
$\Tilde\Orth^+(T)\setminus \Omega (T)$ is at least uniruled.
The same is true for the modular variety $\Tilde{\rm SO}^+(T)\setminus \Omega (T)$
if the rank of $T$ is odd.
\end{corollary}

\section{Jacobi type representation of Borcherds \\products
and the lattice Weyl vector}
\label{JBP}

It is known that one can consider the Kac--Weyl denominator function
of the affine Lie algebra  $\hat{\mathfrak g}(K)$ with positive part
$K$ of the root system as a product of eta- and theta-functions
(see \cite{K1}, \cite{KP}) or as a Jacobi form.
The Borcherds product of $34$ reflective modular forms
constructed  in  Sect. \ref{pullback} is equal to  the right
hand side of the Kac-Weyl-Borcherds denominator formula of the
corresponding Lorentzian Kac-Moody algebra (see Sect. \ref{sec:data1-5}).
In this section we consider a Jacobi type representation of the Borcherds
products of the reflective modular forms.
This gives a description of the multiplicities of imaginary
roots of the corresponding Lorentzian Kac-Moody algebras as Fourier
coefficients of some Jacobi forms of weight $0$.

It is hard to get explicit formulae for the Fourier expansion
of the quasi pull-backs constructed in Theorem \ref{thm-Phi},
Proposition \ref{Pr-2-forms} and Theorem \ref{thmLeech}.
In \cite{G12} we proposed twenty four  Jacobi type constructions
of the Borcherds modular form $\Phi_{12}$. This approach gives
similar formulae for the Borcherds products of all modular forms
of Sect. \ref{pullback}. In particular we give
simple explicit formulae for the first two  Fourier-Jacobi
coefficients of these reflective forms.

\begin{theorem}\label{thm-affineKM}
Let $K$ be one of  lattices  of Theorem \ref{thm-Phi}.

1) All Fourier coefficients of the reflective form
$\Phi_{k,K}$ are integral.
\smallskip

2) $\Phi_{k,K}$ has the  Borcherds product described in
(\ref{phiS}) and (\ref{BP}) below.
\smallskip

3) The lattice Weyl vector of the Lorentzian Kac-Moody algebra
with the hyperbolic $2$-root system of $U\oplus K $
and the denominator function $\Phi_{k,K}$ is given by the formula
\begin{equation}\label{Weyl-vector}
\rho_{U\oplus K }=(A,B,C)=
\bigl(1+h(K),\,
-\frac{1}{2}\sum_{v\in R_2(K)_{>0}} v,\, h(K)\bigr)
\end{equation}
where $h(K)$ is the Coxeter number of $K$
and $B=-\frac 12\sum_{v\in R_2(K)_{>0}} v$ is the direct sum of the Weyl vectors
of the irreducible components of the root system of the  positive definite lattice
$K$.
\smallskip

4) The first non-zero Fourier-Jacobi coefficient
of $\Phi_{k,K}$ has index $h(K)$ and it is equal to
$$
\eta(\tau)^{(h(K)+1)(24-\rk(K))}\cdot \eta(\tau)^{\rk K} \prod_{v\in R(K)_{>0}}
\frac{\vartheta(\tau, (v,\mathfrak{z}))}{\eta(\tau)}
$$
where $\mathfrak{z}\in K\otimes \mathbb C$,
$\vartheta(\tau,z)$ is the Jacobi theta-series defined in
(\ref{vth-product}) and (\ref{vth-sum}). The second Fourier-Jacobi coefficient
of index $h(K)+1$ is given in (\ref{BNR}).
\end{theorem}

\begin{remark} \label{remafflie}
{\rm We note that for a root lattice $K>0$
$$
\eta(\tau)^{\rk K} \prod_{v\in R(K)_{>0}}
\frac{\vartheta(\tau, (v,\mathfrak{z}))}{\eta(\tau)}
$$
is {\it the Kac--Weyl denominator function of the affine Lie algebra }
$\hat{\mathfrak g}(K)$} where $\tau\in \mathbb H^+$, $q=\exp(2\pi i \tau)$,
$$
\eta(\tau)=q^{\frac{1}{24}}\prod_{n\ge 1}(1-q^n)
$$
is the Dedekind eta-function.
\end{remark}

In order to define Fourier and Fourier-Jacobi expansions of modular forms,
we have to fix  a tube realisation of the homogeneous domain $\Omega (T)$
related to boundary components of its Baily--Borel compactification.
In this paper, we shall use automorphic forms
mainly  for the lattices $T$ of signature  $(n_0+2,2)$ of the simplest possible type
$$
T=U'\oplus (U\oplus K )=U\oplus S
$$
where $U'\cong U=\begin{pmatrix}\ 0&-1\\-1&\ 0\end{pmatrix}$ is the unimodular
hyperbolic plane
and $K$ is  a positive definite  even integral lattice of rank $n_0$
and $S=U\oplus K$ is a hyperbolic lattice of signature $(n_0+1,1)$.

Let $[\cZ]\in \Omega (T)$.
Using the basis $\latt{e',f'}_\ZZ=
U={\begin{pmatrix} 0&-1\\ -1&0\end{pmatrix}}$
we write $\cZ=z'e'+\widetilde Z+zf'$ with
$\widetilde Z\in S\otimes \CC$. We note that $z\ne 0$. (If $z=0$, the real
and imaginary parts of $\widetilde Z$ form two orthogonal vectors of negative norm
in the hyperbolic lattice $S\otimes \RR$.)
Thus $[\cZ]=[\frac{1}2(Z,Z) e'+Z+f']$.
Using the similar basis $\latt{e,f}_\ZZ=U$ of the second hyperbolic plane in $T$,
we see that
$\Omega (T)$ is isomorphic to the tube domains
\begin{equation*}\label{H(S)}
\cH({S})=\{Z\in S\otimes \CC\, |\ -(\hbox{\rm Im\,} Z,\, \hbox{\rm Im\,} Z)>0\}^+
\end{equation*}
and
\begin{multline}\label{H(S)2}
\cH({K})=\{Z=\omega e+\gz+\tau f\in S\otimes \CC\, |\,
\\
\tau,\,\omega\in\HH^+ ,\ \gz\in K\otimes \CC,\
2\,\hbox{\rm Im\,} \tau \cdot \hbox{\rm Im\,}\omega
-(\hbox{\rm Im\,} \gz,\, \hbox{\rm Im\,} \gz)_{K}>0\}.
\end{multline}
We fix the isomorphism $[{\rm{pr}}]:\cH(K)\to\Omega (T)$
defined
by the  $1$-dimensional  cusp $\latt{e',e}$ fixed above
\begin{equation*}\label{pr-Z}
Z=(\omega e+\gz+\tau f)\, {\mapsto}\,
{\rm{pr}}(Z)=(\tfrac{(Z,Z) }2 e'+\omega e+\gz+\tau f+f')\,
\mapsto\ \left[{\rm{pr}}(Z)\right].
\end{equation*}

For a primitive  isotropic vector $c\in T$ and any $a \in c^\perp_T$,
one defines the Eichler transvection
\begin{equation*}
t(c,a)\colon v\Mapsto v+(a,v)c-(c,v)a-\frac{1}{2}(a,a)(c,v)c
\in \Tilde\SO^+(T).
\end{equation*}
If  $Z\in \cH(S)$ and $l\in S=(e')^\perp_T/\ZZ e'$, then
$t(e',l)({\rm pr}[Z])={\rm pr}[Z+l]$ is a translation in $\cH(S)$.
Therefore, any modular form $F\in M_k(\Tilde\SO^+(T))$
is periodic, i.e.  $F(Z+l)=F(Z)$ for any $l\in S$.
One defines  the Fourier expansion
of $F$ at the zero-dimensional cusp $\latt{e'}$
\begin{equation}\label{F-exp}
F(Z)=\sum_{l\in S^*,\,-(l,l)\ge 0} f(l)\exp{(-2\pi i \,(l,Z))}
\end{equation}
and its Fourier-Jacobi expansion at the one-dimensional cusp $\latt{e',e}$
\begin{equation}\label{F-exp2}
F(\tau, \gz,\omega)=
\sum_{m\ge 0} \varphi_m(\tau,\gz)\exp{(2\pi i\, m\omega)}.
\end{equation}
(See a general description  of a Fourier expansion at an arbitrary cusp
in \cite[\S 2.3]{GN3} and \cite[\S 8.2-8.3]{GHS4}.)
The Fourier-Jacobi coefficients $\varphi_m(\tau,\gz)$ of
$F\in M_k(\Tilde\SO^+(T))$, where $T=2U\oplus K$, are examples of holomorphic
Jacobi forms of weight $k$ and index $m$ for the even integral lattice $K>0$.
We note that we use the positive orientation of the indices
defined by (\ref{F-exp}) in the Fourier expansion of the Jacobi forms
$\varphi_m(\tau,\gz)$.

\begin{definition} (See \cite{G2}, \cite{CG2}.)
A holomorphic (respectively, weak  or nearly holomorphic)
Jacobi form of weight
$k\in \ZZ$ and index $m\in \NN$ for an even integral positive definite lattice
$K$ is  a holomorphic function
$$
\phi: \HH\times (K\otimes \mathbb C)\to \CC
$$
satisfying the functional equations
\begin{align*}\label{def-JF}
\varphi(\frac{a\tau+b}{c\tau+d},\,\frac{\gz}{c\tau+d})&
=(c\tau+d)^{k}\exp\bigl (\pi i \frac{cm(\gz,\gz)}{c\tau+d}\bigr)
\varphi(\tau ,\,\gz ),
\\\vspace{2\jot}
\varphi(\tau,\gz+\lambda\tau+\mu)&
=\exp\bigl(-\pi i m\bigl( (\lambda,\lambda)\tau+ 2(\lambda,\gz)\bigr)\bigr)
\varphi(\tau,\gz )
\end{align*}
for any
$A=\left(\begin{smallmatrix}
a&b\\c&d\end{smallmatrix}\right)
\in{\SL}_2(\mathbb Z)$,
$\lambda,\,\mu\in K$
and having a Fourier expansion
$$
\varphi(\tau ,\,\gz )=
\sum_{n\in \ZZ,\ \ell \in K^*}
f(n,\ell)\,\exp\bigl(2\pi i (n\tau-(\ell,\gz) \bigr)
$$
where  $f(n,l)\ne 0$ implies $N_m(n,\ell):=2nm-(\ell,\ell)\ge 0$
for holomorphic, $n\ge 0$ for weak
and
$N_m(n,\ell)=2nm-(\ell,\ell)>> -\infty$
for nearly holomorphic Jacobi forms.
We denote the space of holomorphic   (reps. weak or nearly-holomorphic)
Jacobi forms by  $J_{k,K;m}\subset J^{w}_{k,K;m}\subset J^{!}_{k,K;m}$.
If $m=1$, we write simply $J_{k,K}$, etc.
\end{definition}

In \cite{G12}, we showed that any Jacobi form of weight $0$ in $J^{!}_{0,K}$
with integral Fourier coefficients defines an automorphic  Borcherds
 product which is a (meromorphic) automorphic form with respect to
$\Tilde \Orth^+(2U\oplus K )$ with a character.

A Niemeier lattice  $N(R)$ is defined by its non-empty root system
$R=R_1\oplus\dots \oplus R_m$. All components $R_i$  have  the same
Coxeter number $h(R)=h(R_i)$.
In Theorem \ref{thm-Phi} above, the lattice  $K$ is  a
direct component of $R$ and we put $h(K)=h(R)$.

We introduce the Jacobi theta-series $\vartheta_{N}$ of the even
unimodular  positive definite lattice $N=N(R)$ of rank $24$ (see \cite{G2})
$$
\vartheta_{N}(\tau, \mathfrak{z})=
\sum_{\ell\in N}\exp{\bigl(\pi i (\ell,\ell)\tau-2\pi i
(\ell, \mathfrak{z} )\bigr)}
\in J_{12,N}
$$
and  a nearly holomorphic Jacobi form of weight $0$ with integral
Fourier coefficients
\begin{equation}\label{eq-Jth}
\varphi_{0,N}(\tau, \mathfrak z)=
\frac{\vartheta_{N}(\tau, \mathfrak{z})}
{\Delta(\tau)}
=\sum_{\substack{n\ge-1,\ \ell\in N\\ 2n-\ell^2\ge -2}}
a_{N}(n,\ell)q^nr^\ell
\in J_{0,N}^{!}
\end{equation}
where $q=\exp{(2\pi i \tau)}$, $r^\ell=\exp(-2\pi i(\ell, \gz))$
and  $\Delta(\tau)=q\prod_{n\ge 1}(1-q^n)^{24}$
is the Ramanujan delta-function.
The Fourier expansion starts with
$$
\varphi_{0,N}(\tau, \mathfrak z)=
q^{-1}+24+\sum_{v\in R_2(N)}e^{2\pi i (v, \mathfrak z)}+q(\dots),
$$
where $R_2(N)=\{v\in N,\, v^2=2\}$ is the set of roots of the Niemeier lattice.
One has similar formula for the Jacobi theta-series
$\vartheta_{\Lambda_{24}}(\tau, \mathfrak{z})$ of the Leech lattice,
but, in this  case, the Fourier expansion of
$\varphi_{0,\Lambda_{24}}$  does not contain the sum with respect to the roots in $q^0$-term.

We define the pullback of $\varphi_{0,N}$ on the lattice $K\emb N=N(R)$.
In other words, we write  $\gz_N=\gz+\gz'$
with $\gz_N\in N\otimes \CC$, $\gz\in K\otimes \CC$,
$\gz'\in K^\perp_N\otimes \CC$ and we put
\begin{equation}\label{phiS}
\varphi_{0,K}(\tau, \gz)
=\varphi_{0,N}|_{K}=\varphi_{0,N}(\tau,\gz_N)|_{\gz'=0}=
\sum_{n\ge -1,\ \ell\in K^*} a_K(n,\ell)q^nr^\ell
\end{equation}
$$
{}=q^{-1}+24+h(K)(24-\rk K)+
\sum_{v\in R_2(K)}e^{2\pi i (v, \mathfrak z)}+q(\dots)
\in J^{!}_{0,K}
$$
where
$$
a_K(n,\ell)=\sum_{\substack{\ell_1\in (K^\perp_N)^*,\  \ell+\ell_1\in N\\
2n-\ell^2-\ell_1^2\ge -2}} a_{N}(n, \ell+\ell_1).
$$
We note that $h(K)(24-\rk K)=|R_2(K^\perp_N)|$ since all irreducible components
of the root system of the Niemeier lattice $N$ have the same Coxeter number.

In \cite[Theorem 3.1]{G12}, we proved that any Jacobi form of weight $0$
$$
\psi(\tau, \gz)=\sum_{n\in \ZZ,\,\ell\in K^*} f(n,\ell)q^nr^\ell
\in J^{!}_{0,K}
$$
with integral Fourier coefficients $f(n,\ell)$ with  indices $(n,\ell)$
of negative hyperbolic  norm $N(n,\ell)=2n-(\ell,\ell)<0$
determines the (meromorphic) Borcherds product
\begin{equation}\label{BP}
\cB_{\psi}(\tau, \gz, \omega)=q^{A}r^{B}s^{C}
\prod_{\substack{n,m\in \ZZ,\,\ell\in K^* \vspace {1pt} \\
(n,\ell,m)>0}} \bigl(1-q^nr^\ell s^{m}\bigr)^{f(nm,\ell)}
\end{equation}
where $Z=(\tau, \gz,\omega)\in \cH(K)\cong \Omega (T)$,
$q=\exp{(2\pi i \tau)}$, $s=\exp{(2\pi i \omega)}$
and $r^\ell=\exp{(-2\pi i (\ell, \gz))}$,
$(n,\ell,m)>0$ means  that $m>0$, or
$m=0$ and $n>0$, or $m=n=0$ and  $\ell>0$ (in the sense of
the root system in $K$) and
$$
A=\frac{1}{24}\sum_{\ell\in K^*}f(0,\ell),\ \
B=-\frac{1}{2}\sum_{\ell >0 }
f(0,\ell)\ell\in \frac{1}2K^*,\ \
$$
$$
C=\frac{1}{2\rk K}\sum_{\ell \in K^*}
f(0,\ell)(\ell,\ell).
$$
We note that for the formula of the Borcherds product given above,
we fix an ordering in $K$. A positive vector $u\in K^*$
has a positive scalar product with a fixed vector in $T\otimes \RR$
which is not orthogonal to the vectors $\ell\in K^*$ such that
$f(0,\ell)\ne 0$.
We can fix such a vector  once at a boundary component, for example
in  $2U\oplus\Lambda_{24}$.

The Borcherds product $\cB_{\psi}$ has also a Jacobi type representation
\begin{equation*}\label{eq-BHecke}
\cB_{\psi}(Z)=\widetilde\psi_{K;C}(Z)\exp{\bigl( -\sum_{m\ge 1}
m^{-1}\tilde{\psi}| T_-(m)(Z)\bigr)}
\end{equation*}
where $\tilde{\psi}(Z)=\psi(\tau,\gz)e^{2\pi i \omega}$,
$T_-(m)$ is a  Hecke operator of type (\ref{Hecke}) below,
$$
\Tilde\psi_{K;C}(Z)=
\eta(\tau)^{f(0,0)}
\prod_{\ell>0}
\biggl(\frac{\vartheta(\tau,(\ell, \gz))}
{\eta(\tau)}\biggr)^{f(0,\ell)}e^{2\pi i C\omega}
$$
and
\begin{equation}\label{vth-product}
\vartheta(\tau,z)
=q^{1/8}(\zeta^{1/2}-\zeta^{-1/2})
\prod_{n\ge 1}\,(1-q^{n}\zeta)(1-q^n\zeta^{-1})(1-q^n)
\end{equation}
is a Jacobi theta-series of characteristic $2$ with  $q=e^{2\pi i \tau}$ and
$\zeta=e^{2\pi i z}$ (see (\ref{vth-sum}) for its Fourier expansion).
In particular, we have $23$ formulae for the Borcherds
modular forms
\begin{equation}\label{Phi12}
\Phi_{12}(Z)= B_{\varphi_{0, N(R)}}(Z).
\end{equation}
See \cite{G12} for more details where  we used the signature $(2,n)$
typical in the theory of moduli spaces of $K3$ surfaces.

According to (\ref{phiS}), (\ref{BP}) and  (\ref{Phi12})
the quasi pull-back $\Phi_{k,K}$ is written as the Borcherds product defined
by $\varphi_{0,K}(\tau, \gz)$
$$
\Phi_{k,K}=\Phi_{12}|_{2U\oplus K }(\tau, \gz, \omega)=
B_{\varphi_{0,K}}(\tau, \gz, \omega)\in
S_k(\Tilde\Orth^+(2U\oplus K ), \det).
$$
All Fourier coefficients of the Borcherds product $\Phi_{k,K}$ are
integral since all Fourier coefficients  $a(n,\ell)$ of the Jacobi form
$\varphi_{0,K}$ are integral.
Moreover,  for  $\varphi_{0,K}$ the first factor   $q^{A}r^{B}s^{C}$ of
the Borcherds product $B_{\varphi_{0,K}}$ has a very simple expression.
We have
$$
(A,B,C)=
\bigl(1+h(K),\,
-\frac 12\sum_{v\in R_2(K)_{>0}} v,\  h(K)\bigr)
$$
where $B=-\frac 12\sum_{v\in R_2(K)_{>0}}v$ is the direct sum of the Weyl vectors
of the irreducible components of the root systems of $K$.

The Jacobi type representation of the Borcherds product $B_{\varphi_{0,K}}$
gives its  first two non-zero Fourier-Jacobi coefficients
of indices $h(K)$ and $h(K)+1$
\begin{equation}\label{BNR}
\cB_{\varphi_{0,K}}(\tau, \gz, \omega)=
\eta(\tau)^{h(K)(24-\rk K)} \prod_{v\in R_{2}(K)_{>0}}
\frac{\vartheta(\tau, (v,\mathfrak{z}))}
{\eta(\tau)}\,e^{2\pi i\, h(K)\,\omega}\times
\end{equation}
$$
\times \bigl(\Delta(\tau)-\vartheta_{N(R)}|_K(\tau, \mathfrak{z})\,
e^{2\pi i\,\,\omega}+ \dots\bigr)
$$
where the product is taken over all positive roots $v$ of $K$.
In order to finish the proof of Theorem \ref{thm-affineKM} we have to
use that $\Delta(\tau)=\eta(\tau)^{24}$.

\section{Reflective towers of Jacobi liftings}
\label{seriesD8}

\subsection{The  Jacobi lifting and Fourier coefficients\\ of modular forms}
\label{jacliftFourier}

In Sect. \ref{JBP}, we calculated the Borcherds products  of
the reflective modular forms of
Theorem \ref{thm-Phi}. Using the differential operators  (see \cite[\S 8]{GHS4}),
we  can write some expressions  for their Fourier coefficients
in terms of Fourier coefficients of $\Phi_{12}$ but such formulae
are not, in fact, explicit because they contain rather complicated summations.
In this section, we consider reflective modular forms
related to the seventeen  lattices
$$
\latt{-2}\oplus k\latt{2}\  (2\le k\le 9),
\quad U(2)\oplus D_4,\quad U(4)\oplus D_4,
$$
$$
\quad U(4)\oplus kA_1\  (1\le k\le 4)\quad{\rm and}\quad
U(3)\oplus kA_2\ (1\le k\le 3).
$$
The corresponding strongly $2$-reflective  modular forms
have rather simple Fourier expansions.
We construct them using the Jacobi (additive)  lifting  of holomorphic Jacobi forms.
This construction was proposed in \cite{G1}-\cite{G2}
and was extended to the Jacobi modular forms of half-integral index
in \cite{GN6} and \cite{CG2}.

We take  a holomorphic Jacobi form of integral  weight $k$ for an arbitrary
even positive definite lattice $K$
$$
\varphi_k(\tau, \gz)=
\sum_{n\in \NN,\,\ell\in K^*} f(n,\ell)q^nr^\ell
\in J_{k,K}
$$
with $f(0,0)=0$.
Then the lifting of  $\varphi_k$  is defined as
\begin{equation}\label{lift}
\Lift(\varphi_k)(\tau, \gz, \omega)=
\sum_{\substack{ n,m>0,\,\ell\in K^*\\
\vspace{0.5\jot} 2nm-(\ell,\ell)\ge 0}}\
\sum_ {d|(n,\ell,m)}
d^{k-1}f(\tfrac{nm}{d^2},\tfrac{\ell}d)\,e^{2\pi i (n\tau+(\ell,\gz)+m\omega)}
\end{equation}
where $d|(n,\ell,m)$ denotes a positive integral  divisor of the vector
in $U\oplus K^* $.
Then
$$
\Lift(\varphi_k)\in M_k(\Tilde\Orth^+(2U\oplus K )).
$$
We note that {\it the Fourier coefficients of the lifting are integral if
all Fourier coefficients of the holomorphic Jacobi form $\varphi_k$ are integral}.
There is a natural  sufficient condition in terms of the discriminant form
of the lattice $K$ (see \cite[Theorem 4.2]{GHS1})
which implies that the lifting  of a Jacobi cusp form is a cusp form.
In the last case {\it the norm of the Weyl
vector of the automorphic correction will be automatically positive}.

\begin{example}
{\bf ${\mathbf 1}$-reflective modular form of singular weight for \linebreak
$2U\oplus D_8 $.}
{\rm (See \cite{G10} and \cite{CG2}.)
We can consider  the Jacobi theta-series (\ref{vth-product}) having the following
Fourier expansion
\begin{equation}\label{vth-sum}
\vartheta(z)=\vartheta(\tau,z)
=\sum_{m\in \ZZ}\,\biggl(\frac{-4}{m}\biggr)\, q^{{m^2}/8}\,\zeta^{{m}/2}
\in J_{\frac 12, \frac 12}(v_\eta^3\times v_H),
\end{equation}
as a Jacobi form of weight $\frac12$ and index $\frac12$.
In the last formula $\bigl(\frac{-4}{m}\bigr)$ denotes the quadratic  Kronecker symbol.
The Jacobi forms of half-integral indices
were introduced in \cite{GN6}.
(See \cite{CG2} for the lattice case.)

We define the following Jacobi form of singular weight $4$ for $D_8$
$$
\vartheta_{D_8}(\tau, \gz)=\vartheta(z_1)\cdot\ldots\cdot \vartheta(z_8)
\in J_{4,D_8}
$$
where $\gz=(z_1,\dots,z_8)\in D_8\otimes \CC$ where the coordinates
$z_i$  correspond to the euclidean basis of  the model (\ref{Dn}) of $D_n$.
For any $\ell\in D_8^*\subset \frac{1}2\ZZ^8$, we put
$\bigl(\frac{-4}{2\ell}\bigr)=\prod_{i=1}^8 \bigl(\frac{-4}{2l_i}\bigr)$.
Then
\begin{equation}
\vartheta_{D_8}(\tau, \gz)
=\sum_{\ell\in \frac{1}2\ZZ^8}\biggl(\frac{-4}{2\ell}\biggr)
\exp{\bigl(\pi i( (\ell, \ell)\tau+2(\ell, \gz))\bigr)}.
\end{equation}
According to (\ref{lift}), we have
$$
\Lift(\vartheta_{D_8})=\sum_{\substack{
 n,\,m\in \NN,\, \ell\in \frac 12\ZZ^8\\
\vspace{0.5\jot}  2nm-(\ell,\ell)=0}}
\ \sum_{d|(n,\,\ell,\,m)}
d^3 \biggl(\frac{-4}{2\ell/d}\biggr)
e^{2\pi i (n\tau+ (\ell,\gz)+m\omega))}
$$
where $d$ is a divisor of $(n,\ell,m)$ in $U\oplus D_8^* $.
(See a more detailed formula in  \cite[(17)]{G10}.)
This lifting is invariant with respect to the permutations
of $z_1,\dots,z_8$ and anti-invariant with respect to
the reflections $\sigma_{e_i}$. Therefore,
\begin{equation}\label{DeltaD8}
\Delta_{4, D_8}=\Lift(\vartheta(z_1)\cdot \ldots\cdot
\vartheta(z_8))\in M_4(\Orth^+(2U\oplus D_8 ),  \chi_2)
\end{equation}
where $\chi_2$ is a character of order $2$. It was proved
in \cite{G10} that $\Delta_{4, D_8}$ is
{\it strongly reflective with the divisor determined by all
$1$-reflections in $2U\oplus D_8^* $.}
In fact, $\Delta_{4, D_8}$ gives a simple  additive construction
of the Borcherds-Enriques modular form
$\Phi_4^{(BE)}\in M_4(\Orth^+(U\oplus U(2)\oplus E_8(2)), \chi_2)$
from \cite{B5}. Moreover, this automatically  gives the explicit
description of the character.

We have the following Jacobi type Borcherds product
$\Delta_{4, D_8}=B_{\varphi_{0,D_8}}$ (see \cite[\S 3]{G10}
and (\ref{BP}) above) with
\begin{equation}\label{phi0D8}
\varphi_{0,D_8}=2^{-1}(\vartheta_{D_8}|T_{-}(2))/{\vartheta_{D_8}}=
r_1+r_1^{-1}+\dots +r_8+r_8^{-1}+8+q(\dots)
\end{equation}
$$
=8\prod_{i=1}^8
\frac{\vartheta(2\tau,2z_i)}{\vartheta(\tau,z_i)}
+\frac{1}2
\prod_{i=1}^8
\frac{\vartheta(\frac{\tau}2,z_i)}{\vartheta(\tau,z_i)}
+\frac{1}2
\prod_{i=1}^8
\frac{\vartheta(\frac{\tau+1}2,z_i)}{\vartheta(\tau,z_i)}
$$
where $r_i=\exp(2\pi i z_i)$. We recall that for a Jacobi form of weight $k$ we have
by definition
\begin{equation}\label{Hecke}
\psi_{k}|T_{-}(m)(\tau,\gz)=
\sum_{\substack{ad=m\\ b\mod d}}
\ a^{k}\psi_{k}\bigl(\frac{a\tau+b}d,\  a\gz\bigr).
\end{equation}
}
\end{example}

In the rest of the section, we analyse the towers of strongly reflective
modular forms based on three modular forms of singular weight
for the lattices $2U\oplus D_8 $,  $2U\oplus 4A_1 $
and $2U\oplus 3A_2 $ constructed in \cite{G10}.
\medskip

\subsection{The singular modular form for  $2U\oplus D_8 $ and the elliptic
$2$-reflective lattices $\latt{-2}\oplus k\latt{2}$, $2\le k\le 8$.}
\label{susec:2UD_8}

{\it In this subsection  we construct a tower of eight reflective modular forms}.
In Theorem \ref{th:el2refW} of Section \ref{sec:data1-3},
we have  the series of $2$-reflective
lattices $\latt{-2}\oplus k\latt{2}$ where $2\le k\le 8$.
The corresponding automorphic corrections will be defined by
the reflective modular forms with respect to the orthogonal groups
of $U(2)\oplus \latt{-2}\oplus k\latt{2}$.

\begin{theorem}\label{LiftD}
The automorphic correction of the $2$-root system of \newline
$\latt{-2}\oplus (k+1)\latt{2}$ ($1\le k\le 7$) is
defined by
$U(2)\oplus \latt{-2}\oplus (k+1)\latt{2}$ ($1\le k\le 7$)
and by the modular form
\begin{equation*}\label{DeltaDk}
\Delta_{12-k, D_k}= \Lift(\psi_{12-k,\,D_k})\in
S_{12-k}(
{\Orth}^+(U(2)\oplus \latt{-2}\oplus (k+1)\latt{2}), \chi_2)
\end{equation*}
where
\begin{equation}\label{thetaDk}
\psi_{12-k,\,D_k}(\tau, \gz)=
\eta(\tau)^{24-3k}\ \vartheta(z_1)\cdot \ldots\cdot
\vartheta(z_k)\quad(2\le k\le 7)
\end{equation}
and
$$
\psi_{11,\,D_1}=\eta(\tau)^{21}\ \vartheta(2z).
$$
\end{theorem}

\begin{remark} \label{Delta11,D1}
{\rm We want to remark about $\Delta_{11,\,D_1}$.
The last case of $\Delta_{12-k, D_k}$ for $k=1$  is special
because $D_1=\latt{4}$. One gets it as a degeneration of $D_2$.
The function
$\Delta_{11, D_1}=\Lift(\eta(\tau)^{21}\vartheta(\tau, 2z))$
is one the basic reflective Siegel modular forms
$\Delta_{11}\in S_{11}(\Gamma_2)$
with respect to the paramodular group $\Gamma_2$ in the classification
of Lorentzian Kac-Moody algebras of rank $3$ in \cite{GN6}, \cite{GN8}.
Therefore, the $D_8$-tower of reflective modular forms
considered in this subsection starts with the reflective  Siegel form
$\Delta_{11}$. }
\end{remark}

In this subsection, we show how to calculate the Fourier expansions
of the reflective modular forms from Theorem \ref{LiftD},
and propose three ways to write the Borcherds products of them.
In the next lemma, we describe a  trick, {\it the duality argument},
which is very useful for our considerations.

\begin{lemma}\label{8A1} Let  $1\le k\le 8$.

1) The next three  groups are canonically isomorphic if $k\ne 4$
\begin{equation*}
\Orth(\latt{-2}\oplus (k+1)\latt{2})=
\Orth(U\oplus k\latt{1})=
\Orth(U\oplus D_k ).
\end{equation*}
For $k=4$, $\Orth(\latt{-2}\oplus 5\latt{2})$
is isomorphic to a double extension of
$\Tilde\Orth^+(U\oplus D_4 )$.

2) The reflections with respect to  $2$-vectors of $\latt{-2}\oplus (k+1)\latt{2}$
correspond to the reflections with respect to $4$-reflective vectors of
$U\oplus D_k $ or $1$-vectors of $U\oplus D_k^* $.
If $k\ne 4$,  then  all $1$-reflective vectors
of $2U\oplus D_k^* $ belong to the unique
$\Tilde\Orth^+(2U\oplus D_k )$-orbit which is equal
to the set of  $1$-vectors in $2U\oplus k\latt{1}$.

3) If $k=4$, then there are three such
$\Tilde\Orth^+(2U\oplus D_4 )$-orbits, and  one  of them coincides
with the  $1$-vectors in $2U\oplus k\latt{1}$.
\end{lemma}
\begin{proof}
Let $M$ be an integral quadratic lattice. Then
\begin{equation}\label{M*(n)}
\Orth(M)=\Orth(M^*)= \Orth(M^*(n)), \qquad n\in \NN.
\end{equation}
By $M^*(n)$ we denote  the renormalisation by the factor $n$
of the quadratic form of the dual lattice $M^*$.
If $M$ is odd, we denote by $M_{ev}$ the maximal even sublattice
of $M$. Then $\Orth(M)$ can be considered as  a subgroup
of $\Orth(M_{ev})$.
The following isomorphism is valid
$$
\latt{-2}\oplus (k+1)\latt{2}\cong
U(2)\oplus k\latt{2}, \qquad k\ge 1,
$$
since for $k=1$ one has
$
\latt a\oplus \latt b\oplus \latt c
=\latt{a+b, a+c}\oplus \latt{a+b+c}
$
where $a^2=-2$ and $b^2=c^2=2$.
Thus, for $M=\latt{-2}\oplus (k+1)\latt{2}$ we get
$$
\Orth(M)=
\Orth(M^*(2))= \Orth(U\oplus k\latt{1})
\subset \Orth(U\oplus D_k )
$$
because $D_k$ is the maximal even sublattice of $k\latt 1$.
For the discriminant group we have
$$
D_k^*/D_k\cong
\begin{cases}\ZZ/2\ZZ\times \ZZ/2\ZZ
\ &{\rm if}\ k\equiv 0 \mod 2,\\
\ZZ/4\ZZ\ &{\rm if}\ k\equiv 1 \mod 2
\end{cases}
$$
and  $D_k$ for $k\le 8$
has the unique  extension to the odd unimodular lattice $\ZZ^k$
if and only if $k\ne 4$.
This proves the first isomorphism of the lemma.
The renormalisation $M^*(2)$
explains the relation between the reflective vectors of the lattices.

Let us assume that $u\in M$,  $u^2=\pm 4$ and
$\sigma_u\in O(M)$. In this case, $v=u/2\in M^*$.
The $1$-vectors $v$ in $2U\oplus D_k^* $ are classified
by their images in the discriminant group $D_k^*/D_k$
(see the Eichler criterion in \cite{G2} and \cite{GHS4}).
It follows that there exists only one
$\Tilde\Orth^+(2U\oplus D_k )$-orbit of such vectors
if  $k\ne 4$.

If $k=4$, then all classes $e_1$, $(e_1+e_2+e_3\pm e_4)/2 \mod D_4$
of the discriminant group contain $1$-vector. This gives three orbits.
A permutation of $e_i$ and $(e_1+e_2+e_3\pm e_4)/2$ could be realised
in $\Orth(2U\oplus 4\latt{1})$, and the reflection $\sigma_4$
permutes the last two vectors.
\end{proof}

According to the last lemma, the problem of construction
of  automorphic corrections of the hyperbolic root systems
$\latt{-2}\oplus (k+1)\latt{2}$, $1\le k\le 8$ from Theorem \ref{th:el2refW}
is reduces to construction of a $1$-reflective (or equivalently
$4$-reflective) modular form for  the lattice $U\oplus D_k$
which
vanishes along the   walls  of all reflections $\sigma_v$
for $v\in 2U\oplus D_k^* $, $v^2=1$,
$v\equiv e_1 \mod 2U\oplus D_k $.
We obtain the $D_8$-tower of the  reflective modular forms by taking
the  quasi pull-backs of $\Delta_{4,D_8}$ (see (\ref{DeltaD8}))
for  $D_k\oplus D_{8-k}\emb D_8$.
\smallskip

For  $k=9$,  the lattice $\latt{-2}\oplus 9\latt{2}$
{\it is a $2$-reflective  lattice of parabolic type with a lattice Weyl vector
of norm zero,
and $\Delta_{4,D_8}$ is its automorphic correction}.
\smallskip

We put $\gz_k=\sum_{i=1}^k z_ie_i\in D_k\otimes \CC$, $2\le k\le 8$
(see (\ref{Dn})).
A cusp form of the  $D_8$-tower is
the quasi pull-back of $\Delta_{4,D_8}=\Phi_4^{(BE)}$ for $z_k=\dots = z_8=0$
$$
\Delta_{12-k,D_k}=\Delta_{4,D_8}|_{z_k=\dots =z_8=0}.
$$
All of them are $1$-reflective modular forms,
and
$$
\Delta_{12-k, D_k}=
\Lift(\eta(\tau)^{24-3k}\ \vartheta(z_1)\cdot \ldots\cdot
\vartheta(z_k))\in S_{12-k}({\Tilde\Orth}^+(2U\oplus D_k ))
$$
(see \cite[\S 3]{G10}).
If $k\ne 4$, then  there is an extension of order $2$
$$
\Tilde{\Orth}^+(2U\oplus D_k )\subset {\Orth}^+(2U\oplus D_k ).
$$
$\Delta_{12-k,D_k}(\tau, z_1,\dots,z_k,\omega)$
is invariant with respect to the permutations of the variables
$(z_1,\dots,z_k)$, and anti-invariant with respect to
the reflection $z_i\to -z_i$. It follows that
\begin{equation}\label{DeltaDk2}
\Delta_{12-k, D_k}
\in S_{12-k}({\Orth}^+(2U\oplus D_k ), \chi_2)
\end{equation}
where   $\chi_2: {\Orth}^+(2U\oplus D_k)\to \{\pm 1\}$
is defined by the relation
$$
\chi(g)=1\  \Leftrightarrow\
g|_{A_{2U\oplus D_k}}=\id.
$$
The form $\Delta_{12-k, D_k}(\tau,\gz_k,\omega)$ ($2\le k\le 8$)
vanishes on $z_i=0$ ($1\le i\le k$) as the lifting of the product
of the  Jacobi theta-series
by a power of $\eta(\tau)$.
This is exactly the union of walls which we are looking for.
The Borcherds product of $\Delta_{12-k, D_k}$ is defined by a Jacobi form
which can be written  in two ways, as
$$
\varphi_{0,D_k}=\varphi_{0,D_8}|_{z_{k+1}=\dots=z_8=0}=
r_1+r_1^{-1}+\dots +r_k^{-1}+(24-2k)+q(\dots)
$$
or using the quotient $2^{-1}(\psi_{12-k,\,D_k}|T_{-}(2))/\psi_{12-k,\,D_k}$
(see (\ref{phi0D8})).
In the proof of Proposition \ref{DeltaD7-Phi} below,
we give  the third formula for  the same Jacobi form.

According to Lemma \ref{8A1},
we have
\begin{multline*}\label{Delta8A1}
\Delta_{12-k, D_k}(\tau,z_1,\dots,z_k,\omega)
\in S_{12-k}({\Orth}^+(2U\oplus k\latt{1}),\chi_2)=\\
S_{12-k}({\Orth}^+(U(2)\oplus \latt{-2}\oplus (k+1)\latt{2}),\chi_2)
\end{multline*}
is strongly $1$-reflective with respect to
${\Orth}^+(2U\oplus  k\latt{1})$ and
is strongly $2$-reflective with respect to
${\Orth}^+(U(2)\oplus \latt{-2}\oplus k\latt{2})$.
Theorem \ref{LiftD} is proved.
\smallskip

We study the first cusp form of the modular $D_8$-tower
in more details.
We note that
$$
(2\pi i)^{-1}\frac{\partial\vartheta(\tau ,z)}{\partial z}\big|_{z=0}=
 \sum_{n>0}
\biggl(\frac{-4}{n}\biggr)nq^{n^2/8}
=\eta(\tau )^3.
$$
Using  the exact formula for the Fourier coefficients
of the Jacobi form,  we find the Fourier expansion of
the Jacobi lifting  (see (\ref{lift}))
$$
\Delta_{5,D_7}=
\Lift(\eta^3\vartheta(z_1)\cdot\ldots\cdot \vartheta(z_7))=
\sum_{\substack{
n,m, N\in \NN,\,\ell \in D_7^*\\
\vspace{1.0\jot}
8nm-(2\ell, 2\ell)=N^2}}
$$
\begin{equation}\label{FexpD7}
N\sum_{d|(n,\ell,m)}
d^3\biggl(\frac{-4}{N/d}\biggr)
\,\biggl(\frac{-4}{2\ell/d}\biggr)\,
\exp\bigl(2\pi i (n\tau+(\ell,\gz_7)+m\omega)\bigr)
\end{equation}
where
$\ell/d\in D_7^*=\latt{\ZZ^7,(\frac12,\ldots, \frac 12)}$ and
$
\bigl(\frac{-4}{2\ell}\bigr)=
\bigl(\frac{-4}{2l_1}\bigr)\dots \bigl(\frac{-4}{2l_7}\bigr)$.
One can find the Fourier expansion of other functions of the $D_8$-tower
in terms of Fourier coefficients of  $\eta(\tau)^{3(8-k)}$.

Now we give a new product construction of $\Delta_{5,D_7}$ using
the strongly reflective forms
$$
\Phi_{{124}, D_8}=\Phi_{12}|_{D_8\emb N(3D_8)}\quad
{\rm and}\quad
\Phi_{114, D_7}=\Phi_{12}|_{D_7\emb N(D_7\oplus E_6\oplus A_{11})}
$$
of  Theorem  \ref{thm-Phi}.

\begin{proposition}\label{DeltaD7-Phi}
$$
\Delta_{5,D_7}^2=\frac{\Phi_{124, D_8}|_{D_7\emb D_8}}
{\Phi_{114, D_7}}.
$$
\end{proposition}
\begin{proof}
We consider the quasi pull-back
$$
\Phi_{{124}, D_8}|_{D_7\emb D_8}\in
S_{124}(\Tilde{\Orth}^+(2U\oplus D_7 ), \det)
$$
where $D_7=\latt{e_8}^\perp_{D_8}\emb D_8$.
The arguments of the proof of Theorem \ref{Tmain}
give
$$
\divv \Phi_{{124}, D_8}|_{D_7\emb D_8}=
\Sum_{\substack {\pm u\in 2U\oplus D_7
\vspace{0.5\jot} \\
u^2=2}}
\cD_u+
\Sum_{\substack {\pm v\in 2U\oplus D_7
\vspace{0.5\jot} \\
v^2=4,\,  v/2\in D_7^*}}
2\cD_v.
$$
(For a general result of this type see \cite{G16}.)
Then
$$
\divv \frac{\Phi_{{124}, D_8}|_{D_7\emb D_8}}
{\Phi_{114, D_7}}=
\Sum_{\substack {\pm v\in 2U\oplus D_7
\vspace{0.5\jot} \\
v^2=4,\,  v/2\in D_7^*}}
2\cD_v.
$$
For the reflective modular forms  of  Theorem  \ref{thm-Phi},
we found  the Jacobi type Borcherds products in Sect. \ref{JBP}.
We get the following weak  Jacobi form of weight $0$
with integral Fourier coefficients
\begin{align*}
2\phi_{0, D_7}(\tau, \gz_7)&=
\Delta(\tau)^{-1}
\bigl(\vartheta_{N(3D_8)}(\tau, \gz)|_{\gz\in D_7\emb D_8}
-\vartheta_{N(D_7\oplus E_6\oplus A_{11})}
(\tau, \gz)|_{\gz\in D_7}\bigr)\\
{}&=2(r_1+r_1^{-1}+\dots+r_7+r_7^{-1})+20+q(\dots)
\in J_{0, D_7}^{weak}
\end{align*}
where $r_i=\exp(2\pi i z_i)$. The formula for the divisor
of the quotient shows that the last  expansion contains
all Fourier coefficients with negative hyperbolic norms
of their indices. According to (\ref{BP}), $B_{\phi_{0, D_7}}$
is strongly $4$-reflective of weight $5$. Using the
K\"ocher principle, we obtain
\begin{equation*}\label{Phi-D7}
\Delta_{5,D_7}=\Lift(\eta^3\vartheta(z_1)\cdot\ldots\cdot \vartheta(z_7))
=B_{\phi_{0, D_7}}=
\sqrt{\frac{\Phi_{{124}, D_8}|_{D_7\emb D_8}}
{\Phi_{114, D_7}}}.
\end{equation*}
We can find  similar expressions for all functions
$\Delta_{12-k, D_k}$.
\end{proof}

\begin{remark}\label{autdiscr}
{\rm The modular form $\Delta_{4+\deg V ,D_{8-\deg V}}$
is equal  to the automorphic dicriminant of the moduli spaces
of the K\"ahler moduli of a Del Pezzo surfaces $V$ of degree
$1\le \deg V\le 6$ (see \cite{Y}, \cite{G10}).
The explicit  formula of type (\ref{FexpD7})
gives us the generating function of the imaginary simple roots
of the corresponding Lorentizian Kac-Moody algebra defined by
this automorphic dicriminant. }
\end{remark}

\begin{remark}
\label{Scheit}
{\rm Some other reflective modular forms of singular weight
similar to Borcherds forms $\Phi_{12}$  and $\Phi_4^{BE}$ above
were found by N. Scheithauer (see \cite{Sch}). The reflective
forms of singular weight in his class are modular with respect
to congruence subgroups.}
\end{remark}

\medskip

\subsection{The $D_8$-tower of Jacobi liftings and
$U(2)\oplus D_4$}
\label{D8towerjaclift}
In this subsection, we construct three new reflective modular forms.
In Proposition \ref{Pr-2-forms}, we proved that the $2$-root system of
$U\oplus D_4 $ has two different automorphic corrections.
The second example of this type is given in the next theorem.
\begin{theorem}\label{U(2)D4}
The hyperbolic $2$-root system of  $U(2)\oplus D_4$
has two different automorphic corrections.
One of them has (nearly) a Jacobi lifting construction.
\end{theorem}
\begin{proof}
We construct two $2$-reflective modular forms of different weights
with respect to the stable orthogonal groups
of the lattices
$$
U\oplus U(2)\oplus D_4 \quad {\rm and}
\quad  U(2)\oplus U(2)\oplus D_4 .
$$
The first automorphic correction was given
in Proposition \ref{Pr-2-forms}:
$$
\Phi_{40,\,U(2)\oplus D_4}=
\Phi_{28, N_8}|_{U\oplus U(2)\oplus D_4 }
\in
S_{40}(\Tilde\Orth^+(U\oplus U(2)\oplus D_4 ), \det).
$$
The second  automorphic correction is related
to the $D_4$-modular form
$$
\Delta_{8, D_4}=\Lift(\eta^{12}(\tau)\vartheta(z_1)\vartheta(z_2)
\vartheta(z_3)\vartheta(z_4))
\in S_8(\Orth^+(2U\oplus D_4 ),\chi_2)
$$
which is a $4$-reflective modular form in the reflective
$D_8$-tower of Theorem \ref{LiftD}.
It  vanishes along the $4$-vectors of the orbit
$\Tilde\Orth^+(2U\oplus D_4 )(2e_1)$ (see (\ref{Dn})).
We proved in Lemma \ref{8A1} that the lattice $2U\oplus D_4 $
contains three  $\Tilde\Orth^+(2U\oplus D_4 )$-orbits of $4$-vectors.
The trick of Lemma \ref{8A1}  shows that
$$
\Orth(2U(2)\oplus D_4 )=
\Orth(2U\oplus D_4^*(2))= \Orth(2U\oplus D_4 )
$$
since $D_4^*(2)\cong D_4$.
The $2$-vectors of $2U(2)\oplus D_4 $ correspond to all
$4$-vectors of $2U\oplus D_4 $. We get them from the first orbit
$\Tilde\Orth^+(2U\oplus D_4 )(2e_1)$ using the reflections
$$
\sigma_1=\sigma_{(-e_1-e_2-e_3+e_4)/2},\quad
\sigma_2=\sigma_{(e_1+e_2+e_3-e_4)/2}\in
\Orth^+(2U\oplus D_4 )
$$
with respect to $1$-vectors in $D_4^*$. We have
$$
\sigma_1(e_4)=(e_1+e_2+e_3-e_4)/2,\quad
\sigma_2(e_4)=(e_1+e_2+e_3+e_4)/2.
$$
Therefore, the product of three Jacobi liftings is a strongly reflective
modular form with the complete $4$-divisor
\begin{equation}\label{F2U(2)+D4,2}
F_{24, U(2)\oplus D_4}=\Delta_{8, D_4}\cdot (\Delta_{8, D_4}|\sigma_1)\cdot
(\Delta_{8, D_4}|\sigma_2)\in S_{24}(\Orth^+(2U\oplus D_4 ),\chi_2).
\end{equation}
We note that $\Delta_{8, D_4}|\sigma_1=\Lift(\varphi_{8,D_4}^{(1)})$
and $\Delta_{8, D_4}|\sigma_2=\Lift(\varphi_{8,D_4}^{(2)})$
are $4$-reflective where
$$
\varphi_{8,D_4}^{(1)}=\eta^{12}(\tau)\vartheta(\tfrac{-z_1+z_2+z_3+z_4}2)
\vartheta(\tfrac{z_1-z_2+z_3+z_4}2)
\vartheta(\tfrac{z_1+z_2-z_3+z_4}2)
\vartheta(\tfrac{z_1+z_2+z_3-z_4}2),
$$
$$
\varphi_{8,D_4}^{(2)}=\eta^{12}(\tau)\vartheta(\tfrac{z_1+z_2+z_3+z_4}{2})
\vartheta(\tfrac{z_1+z_2-z_3-z_4}{2})
\vartheta(\tfrac{z_1-z_2-z_3+z_4}{2})
\vartheta(\tfrac{z_1-z_2+z_3-z_4}{2})
$$
(see \cite[Example 1.8]{CG2}).
The Jacobi type Borcherds product for $F_{24, U(2)\oplus D_4}$
can be easily constructed  from the corresponding product for
$\Delta_{8, D_4}$.
\end{proof}

\begin{remark} \label{2,4,reflections}
{\rm The products of  reflective forms of $D_m$-type
from Theorem \ref{thm-Phi} and the functions constructed in the subsection 6.2
$$
\Phi_{k_m, D_m}\cdot\Delta_{12-m, D_m}  \qquad (2\le m\le 8)
$$
are strongly reflective modular forms with the complete reflective
divisor determined by all reflections in
$2U\oplus D_m$ ($m\ne 4$).
These modular forms determine Lorentzian Kac--Moody algebras
with the maximal  Weyl groups generated by all  $2$- and $4$-reflections
of the hyperbolic lattices $2U\oplus D_m$ ($m\ne 4$).

For $m=4$ we get more complicated formula for the root system
for the strongly reflective modular form with the complete reflective
divisor ``of type'' $F_4$
$$
\Phi_{72, D_4}\cdot \Delta_{8, D_4}\cdot (\Delta_{8, D_4}|\sigma_1)\cdot
(\Delta_{8, D_4}|\sigma_2).
$$
}
\end{remark}

\medskip

\subsection{The singular modular form for $2U\oplus 4A_1$
and\\ $U(4)\oplus D_4$.}
\label{singforms}
In this subsection, we construct four reflective  modular forms,
in particular, the automorphic correction of  the $2$-root system
of $U(4)\oplus D_4$.

\begin{theorem}\label{Thm-U(4)+D_4}
There is a strongly $2$-reflective modular form
$$
F_6\in M_6(\Orth^+(2U(4)\oplus D_4 ),\chi_2)
$$
with complete $2$-divisor.
\end{theorem}

We construct this reflective form using
the singular modular form for $2U\oplus 4A_1 $ (see \cite[\S 5]{G10}).
As in Lemma \ref{8A1},
$$
\Orth(2U(4)\oplus D_4)=\Orth(2U\oplus D_4^*(4))=\Orth(2U\oplus D_4(2))
$$
since $D_4^*(2)=D_4$.
\begin{lemma}\label{-2 in D_4(-2)}
The $2$-vectors of $2U(4)\oplus D_4 $ correspond
to $\frac12$-vectors
of the dual lattice $2U\oplus D_4(2)^*$.
\end{lemma}
\begin{proof}
Any $2$-vector $v\in M=2U(4)\oplus D_4 \subset M^*$ is  primitive
in $M^*$ since $D_4^*$ is odd integral. After renormalisation
by $4$, we have $\frac{v}4\in (M^*(4))^*\cong 2U\oplus D_4(2)^*$,
$(\frac{v}4)^2=\frac12$.
Any $\frac12$-vector in $2U\oplus D_4(2)^*$
is primitive in this lattice  since  the minimal possible norm
in  $D_4(2)^*$  is equal to $\frac12$.
\end{proof}

By definition,  $D_4\subset \ZZ^4$. Therefore,
$2U\oplus D_4(2)\subset 2U\oplus 4A_1 $
of index $2$ and
$$
\Tilde\Orth^+(2U\oplus D_4(2))\subset \Tilde\Orth^+(2U\oplus 4A_1 ).
$$
\begin{lemma}\label{orbitsD4-2}
We put $D_4(2)\subset 4A_1=\oplus_{i=1}^4\ZZ a_i$.
There are twenty four  $\Tilde\Orth^+(2U\oplus D_4(2))$-orbits
of $\frac 12$-vectors in the dual lattice  $2U\oplus D_4(2)^*$
and four $\Tilde\Orth^+(2U\oplus 4A_1 )$-orbits
of $\frac 12$-vectors in $2U\oplus 4A_1 $.
\end{lemma}
\begin{proof}
We proved above that all $\frac 12$-vectors are primitive
in the corresponding dual lattices. According to the Eichler criterion
(see \cite{G2}, \cite{GHS4}), the orbit of a $\frac 12$-vector
with respect to the stable  orthogonal group is defined by its image
in the discriminant group.
It is clear that there are four such orbits $\frac{a_i}2$
$(1\le i \le 4)$ in $2U\oplus 4A_1 $.

We have
$$
D_4(2)^*/D_4(2)\cong \frac12 D_4^*/D_4\cong D_4^*/2D_4\cong
(D_4^*/D_4)/(D_4/2D_4).
$$
Analysing the last quotient, we see that the discriminant groups
\linebreak
$D_4(2)^*/D_4(2)$ contains $24$ (respectively $4$, $12$, $24$)
elements of norm $\frac 12\mod 2$ (respectively of norms
$0,1,\frac 32 \mod 2$). In the case of norm $\frac 12$,
their representatives are
$\pm a_i/2$, $(\pm a_1\pm a_2\pm a_3\pm a_4)/4$.
\end{proof}

The product of Jacobi theta-series
$\vartheta(z_1)\cdots \vartheta(z_n)$ can be considered as a Jacobi
form for $D_n$ (see Example 3.1)
or a Jacobi form of half-integral index  for $nA_1$.
For example,
$$
\psi_{2,4A_1}(\tau,\gz_4)=
\vartheta(\tau,z_1)\vartheta(\tau,z_2)\vartheta(\tau,z_3)
\vartheta(\tau,z_4)\in J_{2,D_4}(v_\eta^{12})
$$
is a $D_4$-Jacobi form with character $v_\eta^{12}:SL_2(\mathbb Z)\to \{\pm 1\}$.
The same product
$$
\psi_{2,4A_1}(\tau,\gz_4)=
\vartheta(\tau,z_1)\vartheta(\tau,z_2)\vartheta(\tau,z_3)
\vartheta(\tau,z_4)\in J_{2,4A_1;\frac{1}2}(v_\eta^{12}\times v_H)
$$
is a  Jacobi form of index $\frac{1}2$ with respect to the lattice
$4A_1$ where $v_H$ is the unique non-trivial binary character of the Heisenberg
group $H(4A_1)$  (see \cite{CG2}).
For such Jacobi forms, the corresponding lifting
contains only Hecke operators of indices congruent to a constant modulo
the conductor of the character.
According to the lifting constructions (see \cite[Theorem 1.12]{GN4}
and  \cite[Theorem 2.2]{CG2}), the following function is defined:
\begin{equation}\label{Delta4A_1}
\Delta_{2,4A_1}=\Lift(\psi_{2,\,4A_1})\in M_2(\Orth^+(2U\oplus 4A_1 ),\chi_2).
\end{equation}
All Fourier coefficients with primitive indices of  this lifting
are equal to $\pm 1$ or $0$
$$
\Delta_{2,4A_1}=
\sum_{\substack{\ell=(l_1,\dots, l_4)\in \ZZ^4,\
l_i\equiv 1 \,{\rm mod \,}2}}
$$
$$
\sum_{\substack{
 n,\,m\in \NN\\
\vspace{0.5\jot} n\equiv m\equiv 1\,{\rm mod\,}2\\
\vspace{0.5\jot}  4nm-l_1^2-l_2^2-l_3^2-l_4^2=0}}
\sigma_1((n,\ell,m))
\biggl(\frac{-4}{\ell}\biggr)\exp(\pi i (n\tau+l_1z_1+\dots+l_4z_4+m\omega))
$$
where $(n,\ell ,m)$ is the greatest common divisor
and
$\bigl(\frac{-4}{\ell}\bigr)=\bigl(\frac{-4}{l_1l_2l_3l_4}\bigr)$
is the Kronecker symbol.
It  was proved in \cite[Theorem 5.1]{G10} that
$$
\divv\Delta_{2,\, 4A_1}
=\Sum_{\substack
{\pm v\in 2U\oplus 4A_1^*,
\ v^2=\frac{1}2
}} \cD_v(2U\oplus 4A_1).
$$
In other words, $\Delta_{2,4A_1}$ is a strongly $2$-reflective modular form
which vanishes along the divisors defined by one of two
$\Orth^+(2U\oplus 4A_1 )$-orbits of $2$-vectors in $2U\oplus 4A_1$ ($(2v)^2=2$).

The Borcherds product of this modular form is defined by the Jacobi form
(see (\ref{phi0D8}) and (\ref{Hecke}))
$$
\varphi_{0,\,4A_1}(\tau,\gz_4)=
3^{-1}\frac{\psi_{2,\,4A_1}|\, T_{-}(3)}{\psi_{2,\,4A_1}}
\in J_{0, 4A_1}^{(weak)}.
$$

We can  consider $\Delta_{2,\, 4A_1}$ as a modular form
with respect to  $\Tilde\Orth^+(2U\oplus D_4(-2))$.
We note that $\psi_{2,4A_1}(\tau,-\gz_4)=\psi_{2,4A_1}(\tau,\gz_4)$,
and the same property has its lifting. Therefore
$$
\Delta_{2,\, 4A_1}\in M_2(\Tilde\Orth^+(2U\oplus D_4(2)),\chi_2).
$$
More exactly, $\Delta_{2,\, 4A_1}$ is anti-invariant under the action
of  $4$ reflections with respect to $\pm a_i/2$ and invariant
with respect to any permutation of $a_i$. Therefore,
$\Delta_{2,\, 4A_1}$ vanishes along  the devisors defined by
only $8$ of $24$ vectors of square $\frac12$ from Lemma \ref{orbitsD4-2}.
As in (\ref{F2U(2)+D4}), we put
\begin{equation}\label{FU(4)+D4}
F_{6}=\Delta_{2, 4A_1}\cdot (\Delta_{2, 4A_1}|\sigma_1)\cdot
(\Delta_{2, 4A_1}|\sigma_2)\in M_{6}(\Orth^+(2U\oplus D_4(2)),\chi_2)
\end{equation}
where
$\sigma_1=\sigma_{(a_1+a_2+a_3-a_4)/4},
\sigma_2=\sigma_{(a_1+a_2+a_3+a_4)/4}\in
\Orth^+(2U\oplus D_4(2))$.
We note that $\Delta_{2, 4A_1}|\sigma_1$ and $\Delta_{2, 4A_1}|\sigma_2$
are reflective. (Compare with the functions from the previous subsection.)
The product of three Jacobi liftings is a strongly reflective
modular form with the complete $\frac12$-divisor.
Moreover, $F_6$ is anti-invariant with respect to twenty four
$\frac12$-reflections in $D_4(2)$ and invariant with respect to all
permutations of $a_i$. Therefore, $F_6$ is modular with respect
to the full orthogonal group $\Orth^+(2U\oplus D_4(2))$
since $\Orth (D_4(2))=\Orth (D_4)$.
Theorem \ref{Thm-U(4)+D_4} is proved.
\medskip

\noindent
\subsection{The reflective tower of Jacobi liftings for $U(4)\oplus kA_1$,
$k\le 3$.}
\label{U(4)kA1}
\medskip

In this subsection, we construct seven reflective modular forms.
This $4A_1$-tower is based on the reflective modular form of singular weight
$\Delta_{2,4A_1}$ (see (\ref{Delta4A_1}))
and starts with  the  Igusa modular form $\Delta_5$ considered
as Borcherds product in \cite{GN1}.

\begin{theorem} \label{U(4)pluskA_1}
The automorphic correction
of the $2$-root system $U(4)\oplus kA_1 $ ($1\le k\le 3$)
is given by
\begin{equation}\label{4A_1-tower}
\Delta_{6-k,kA_1}=\Lift(\eta^{3k}\vartheta(z_1)\vartheta(z_2)
\vartheta(z_3))
\in S_{6-k}(\Orth(2U\oplus kA_1 )).
\end{equation}
\end{theorem}

Similar to Lemma \ref{8A1} and Lemma \ref{-2 in D_4(-2)},
we get
\begin{lemma}\label{4A1} Assume $1\le k\le 4$. Then
$$
\Orth(2U(4)\oplus kA_1 )=\Orth(2U\oplus kA_1 ).
$$
The $2$-reflective vectors of $2U(4)\oplus kA_1 $
correspond to $\frac 12$-reflective vectors of $2U\oplus kA_1^* $.
These vectors belong to $k$ different $\Tilde \Orth^+(2U\oplus kA_1 )$-orbits
or one $\Orth^+(2U\oplus kA_1 )$-orbit.
\end{lemma}

We note that the basic  reflective modular form of this tower
$\Delta_{2,4A_1}=\Lift(\vartheta(z_1)\vartheta(z_2)\vartheta(z_3)
\vartheta(z_4))$ defines {\it an automorphic correction
of the  hyperbolic $2$-root system $U(4)\oplus 4A_1 $
of parabolic type  with a lattice Weyl vector of norm $0$}.
We take three consecutive quasi pull-backs of this modular form
for $z_4=0$, $z_3=0$ and  $z_2=0$ and get  three strongly reflective
cusp forms for $k=1,\,2,\,3$
$$
\Lift(\eta^{3k}\vartheta(z_1)\cdots
\vartheta(z_{4-k}))
\in S_{2+k}(\Tilde\Orth^+(2U\oplus (4-k)A_1 ))
$$
with the complete $\frac 12$-divisor described in  Lemma \ref{4A1}
(see \cite[\S 5]{G10}).
The Fourier expansion of the quasi pull-backs can be written
in terms of the Fourier coefficients of $\eta(\tau)^{3k}$.
Here we give  the formula for the first cusp modular form of this tower
which contains only elementary functions:
$$
\Delta_{3,3A_1}(\tau,\gz_3,\omega)=\Lift(\eta^{3}\vartheta(z_1)\vartheta(z_2)
\vartheta(z_3))=
$$
$$
\sum_{\substack{
n\equiv m\equiv 1\,{\rm mod\,}2\\
\vspace{1\jot}
l_i\equiv 1 {\rm mod\,} 2,\\
\vspace{1\jot}
4nm-l_1^2-l_2^2-l_3^2=N^2}}
N\biggl(\frac{-4}{Nl_1l_2l_3}\biggr)\sigma_1((n,\ell,m))
\exp(\pi i (n\tau+l_1z_1+l_2z_2+l_3z_3+m\omega)).
$$
As in the case the reflective modular  $D_8$-tower
(see (\ref{phi0D8}) and Proposition \ref{DeltaD7-Phi}),
we give different constructions of the Borcherds product
for $\Delta_{3,3A_1}$.
First (see (\ref{Hecke}) and  \cite[\S 5]{G10}), we have two
formulae for the weak Jacobi form
$$
\varphi_{0,3A_1}(\tau,\gz_3)=3^{-1}\frac{\psi_{5,\,3A_1}|\, T_{-}(3)}{\psi_{5,\,3A_1}}
=\varphi_{0,4A_1}(\tau,\gz_4)|_{z_4=0}\in J_{0, 3A_1}^{weak}
$$
such that $\Delta_{3,3A_1}=B_{\phi_{0,3A_1}}$, see (\ref{BNR}).
Similar formulae are valid for $\phi_{0,2A_1}$ and $\phi_{0,A_1}$.

Second, we can construct the modular forms of the reflective $4A_1$-tower
using some $2$-reflective modular forms of Theorem \ref{thm-Phi}.

\begin{proposition}\label{Phi-3A1}
$$
\Delta_{3,3A_1}^2=
\frac
{\Phi_{39, 3A_2}|_{3A_1\emb 3A_2 }}{\Phi_{33, 3A_1}},\quad
\Delta_{4,2A_1}^2=
\frac
{\Phi_{42, 2A_2}|_{2A_1\emb 2A_2 }}{\Phi_{34, 2A_1}},
$$
$$
\Delta_{5,A_1}^2=
\frac
{\Phi_{45,A_2}|_{A_1\emb A_2 }}{\Phi_{35, A_1}}.
$$
\end{proposition}

\begin{proof}
We consider the case $k=3$.
We embed $A_1=\latt{u}$ in $A_2\latt{u,v}$ where $u^2=v^2=2$.
Then $A_1\perp \latt{6}\subset A_2$, and two pairs  $\pm v$ and $\pm (u+v)$
of $A_2$-roots have ``small'' orthogonal  projections of norm $1$ on $u$.
We take this embedding for $3$ copies $3A_1\emb 3A_2$.
We note that the lattices $3A_2$  and $3A_1$ satisfy the $({\rm Norm}_2)$ condition
of Theorem \ref{Tmain}.
Therefore, the arguments from the proof of Theorem \ref{Tmain}
show that the pull-back $\Phi_{39, 3A_2}|_{3A_1\emb 3A_2 }$
has weight $39$ and  vanishes along all $2$-divisors
and additionally along  $1$-divisors corresponding to
$1$-vectors of $2U\oplus 3A_1^* $. The $1$-divisors have multiplicity
$2$. Dividing this pull-back by the strongly reflective form
${\Phi_{33, 3A_1}}$, we get  $\Delta_{3,3A_1}^2$ according to the K\"ocher principle.
\end{proof}

\begin{remark} \label{remDelta5Phi}
{\rm 1) We note that $\Delta_{5,A_1}^2$ is equal to the Igusa
modular form $\Psi_{10}\in S_{10}(\Sp_2(\ZZ))$ which is the first Siegel
cusp form of weight $10$.  The Borchers product formula for $\Psi_{10}$
was constructed in \cite{GN1} (see also \cite{GN6} for other constructions
of the Igusa modular form).
Proposition \ref{Phi-3A1} gives a new model of  this very important function.
The function $\Delta_{5,A_1}$ is the automorphic correction of
the hyperbolic root system with Cartan matrix $$
\begin{pmatrix} \ \,2&-2&-2\\-2&\ \,\,2&-2\\-2&-2&\ \,\,2\end{pmatrix}
$$
(see \cite{GN1}).
\smallskip

2) The quotient
$$
\frac{\Phi_{32+k, (4-k)A_1}}{\Delta_{2+k,\,(4-k)A_1}}
$$
for $0\le k\le 3$ is a holomorphic strongly reflective modular form.
It  defines a Lorentzian Kac--Moody algebra of a hyperbolic root system
of $U\oplus (4-k)A_1$  whose Weyl group is smaller than
the full Weyl group  generated by all $2$-reflections
in $U\oplus (4-k)A_1$. The Cartan matrix of such  Lorentzian Kac--Moody algebra
for $U\oplus A_1$ was found in \cite{GN6}.}
\end{remark}

\medskip

\noindent
\subsection{The reflective tower $U(3)\oplus kA_2$, $1\le k\le 3$.}
\label{U(3)kA_2}
\medskip

{\it In this subsection, we define six reflective modular forms}
using  a  modular form of singular weight for the lattice $2U\oplus 3A_2 $
proposed in  \cite[\S 4]{G10}.
This  modular form also gives  interesting  series of canonical
differential forms on Siegel modular three-folds constructed with
theta-blocks (see \cite{GPY}).

\begin{lemma}\label{J-A2}
The function
\begin{equation}\label{sigmaA2}
\sigma_{A_2}(\tau,\gz)=\sigma_{A_2}(\tau, z_1,z_2)
=\frac{\vartheta(\tau,z_1)\vartheta(\tau,z_2)\vartheta(\tau,z_1+z_2)}
{\eta(\tau)} \in J_{1,A_2}(v_\eta^8)
\end{equation}
is a holomorphic Jacobi form of singular weight
which is anti-invariant with respect to the $6$-reflections
from $\Orth(A_2)$.
\end{lemma}
\begin{proof}
By construction, $\sigma_{A_2}$ is a weak Jacobi form.
See \cite[Corollary 3.4]{CG2} for a modular proof that
this is a holomorphic Jacobi form.

One can see the same using the theory of affine Lie algebras.
$\sigma_{A_2}$ is a dual variant of the
denominator function of the affine Lie algebra $\hat{\mathfrak g}(A_2)$
(see the remark after Theorem \ref{thm-affineKM}) which is a holomorphic Jacobi form.
Let $v_1$ and $v_2$ be the simple roots of $A_2$ ($v_1^2=v_2^2=2$ and
$(v_1,v_2)=-1$) and
$\lambda_1=(2v_1+v_2)/3$ and  $\lambda_2=(v_1+2v_2)/3$
is the corresponding dual basis of $A_2^\ast$, that is
$(v_i,\lambda_j)=\delta_{i,j}$).
Then $\pm 3\lambda_1$, $\pm 3\lambda_2$ and $\pm 3(\lambda_1-\lambda_2)$
are reflective vectors of square $6$ in  $A_2$, and
$$
\sigma_{A_2}(\tau, \gz)=
-\frac{\vartheta(\tau,(\gz, \lambda_1))\vartheta(\tau, (\gz,\lambda_2))
\vartheta(\tau, (\gz, \lambda_1-\lambda_2))}{\eta(\tau)}
\qquad (\gz\in A_2\otimes \CC).
$$
We see that this product is anti-invariant with respect to
the $6$-reflections.
\end{proof}

The direct product of three copies of $\sigma_{A_2}(\tau,\gz)$
is a Jacobi form of singular weight
$$
\sigma_{3A_2}(\tau, \gz_1+\gz_2+\gz_3)=
\prod_{i=1}^3\sigma_{A_2}(\tau,\gz_i)\in  J_{3,3A_2}
$$
with trivial character.
It was proved in \cite[Theorem 4.2]{G10} that its lifting
$$
\Delta_{3,3A_2}=\Lift(\sigma_{3A_2})\in M_3(\Tilde\Orth^+(2U\oplus 3A_2 ))
$$
is {\it a strongly reflective modular form with the complete $6$-divisor}.
We note that all Fourier coefficients of $\sigma_{3A_2}$ and $\Delta_{3,3A_2}$
are integral.

The quotient group $\Orth^+(2U\oplus 3A_2 )/\Tilde\Orth^+(2U\oplus 3A_2 )$
is isomorphic to the orthogonal group $\Orth(q_{3A_2})$ of the discriminant
form of $(3A_2)^*/(3A_2)\cong (\ZZ/3\ZZ)^3$. We have
$\Orth(q_{3A_2})\cong S_3\times C_2^3$
where  $S_3$ is the group  of  permutations of three copies of  $A_2$,
and the  cyclic group $C_2$ of order $2$ is generated
by a $6$-reflection in $A_2$.
The Jacobi form $\sigma_{3A_2}$ is invariant with respect to the  permutations
of the copies of $A_2$ and anti-invariant with respect to $6$-reflections
due to Lemma \ref{J-A2}. The same properties are valid  for the  lifting
of $\sigma_{3A_2}$.
Therefore, we prove that
$$
\Delta_{3,3A_2}=\Lift(\sigma_{3A_2})\in M_3(\Orth^+(2U\oplus 3A_2 ), \chi_2)
$$
where $\chi_2$ is a binary character of the full orthogonal group
which is trivial
on the stable orthogonal group $\Tilde\Orth^+(2U\oplus 3A_2 )$.
For two   quasi pull-backs of $\Lift(\sigma_{3A_2})$
on  the homogeneous domains of $2U\oplus 2A_2 $ and $2U\oplus A_2$
we get
$$
\Delta_{6,2A_2}=\Lift(\eta^8\cdot\sigma_{2A_2})
\in S_6(\Orth^+(2U\oplus 2A_2 ), \chi_2),
$$
$$
\Delta_{9,A_2}=\Lift(\eta^{16}\cdot\sigma_{A_2})
\in S_9(\Orth^+(2U\oplus A_2 ), \chi_2).
$$
The quasi pull-backs  are also strongly reflective modular forms with
the complete $6$-divisor according to the arguments of Theorem \ref{Tmain}.

The Borcherds product of $\Delta_{3,3A_1}$ is defined by the
weak Jacobi form of weight $0$ with integral Fourier coefficients
(see (\ref{Hecke}))
$$
\varphi_{0,\,3A_2}(\tau,\gz)=
2^{-1}\frac{\sigma_{3,\,3A_2}|_3 T_{-}(2)}{\sigma_{3,\,3A_2}}.
$$
Then
$$
\varphi_{0,\,3A_2}(\tau,\gz)=
6+\sum_{i=1,3,5}
(r_i+r_i^{-1}+r_{i+1}+r_{i+1}^{-1} + r_ir_{i+1}^{-1}+r_i^{-1}r_{i+1})
+q(\dots)
$$
where $r_i=\exp(2\pi i (\gz,\lambda_i))$, $\gz\in (3A_2)\otimes \mathbb C$
and $\lambda_i$ ($i=1,\dots,6$) give the dual bases to simple roots
of the corresponding copies of $A_2$.
The Borcherds products of the reflective modular forms $\Delta_{6,2A_1}$
and $\Delta_{9,A_1}$ are defined by the corresponding pull-backs of
$\varphi_{0,3A_2}$.

\begin{lemma}\label{U(3)+3A_2} We have  that
$\Orth^+(2U(3)\oplus kA_2 )= \Orth^+(2U\oplus kA_2 )$,
and the $2$-reflections of $2U(3)\oplus kA_2$
($k=1$, $2$, $3$) correspond
to the $6$-reflections  of the lattice $2U\oplus kA_2^*(-3)\cong 2U\oplus kA_2$.
\end{lemma}
\begin{proof}
The isomorphism of the lemma follows from (\ref{M*(n)}) and the fact that
$A_2^*(3)\cong A_2$.  The $2$-reflections of $2U(3)\oplus kA_2 $
are the $6$-reflections in $(2U(3)\oplus kA_2 )^*(3)$.
We recall that
$[\Orth(A_2) : W_2(A_2)] = 2$
where $\Orth(A_2)$ is the integral  orthogonal group of the lattice $A_2$.
$\Orth(A_2)$ contains reflections with respect to the $2$- and $6$-vectors.
All these roots  form the root system $G_2$, and $\Orth (A_2) = W(G_2)$.
See \cite{GHS3} for the root systems $G_2$ and $F_4$ in the
context of automorphic forms.
\end{proof}

The results, proved above, give
\begin{theorem}\label{lift3A2}
For  $k=1$, $2$ or $3$ the modular form
$$
\Delta_{12-3k,\,kA_2}\in M_{12-3k}({\Orth}^+(2U(3)\oplus kA_2 ),\chi_2)
$$
is strongly $2$-reflective  modular form with the complete $2$-divisor
where $\chi_2$ is a binary character of the orthogonal group.
For $k=1$ and $2$ they are cusp forms.
\end{theorem}

\begin{remark} \label{Delta12-3kPhi45A2}
{\rm 1) The  Fourier expansions of reflective forms
$\Delta_{12-3k,\,kA_2}$ ($k=1$, $2$, $3$) are given by formula
(\ref{lift}).
All their Fourier coefficients are integral.
These modular forms determine three Lorentzian Kac--Moody algebras related
to the hyperbolic lattices $U(3)\oplus kA_2 $.
The lattice Weyl vector of the hyperbolic $2$-root system of $U(3)\oplus 3A_2 $
has norm $0$ since the modular form $\Delta_{3,\,3A_2}$ is of  singular weight.
{\it The corresponding hyperbolic root system  of signature $(7,1)$
is of the  parabolic type}.
\smallskip

2) The products of  reflective forms
$\Phi_{45, A_2}\cdot \Delta_{9,\,A_2}$,
$\Phi_{42, 2A_2}\cdot \Delta_{6,\,2A_2}$
and $\Phi_{39, 3A_2}\cdot \Delta_{3,\,3A_2}$
are strongly reflective modular forms with the complete reflective
divisor determined by all reflections in the corresponding lattices.
These modular forms determine three  Lorentzian Kac--Moody algebras
with the maximal  Weyl groups generated by all  $2$- and $6$-reflections
of the hyperbolic lattices (compare with  the root system $G_2$).
For these more general cases of reflections in roots with arbitrary squares,
one should follow general definitions from \cite{GN5}, \cite{GN6}, \cite{GN8}, \cite{Nik12}
of data (I)---(V) and Lorentzian Kac--Moody algebras instead of given in
Sec. \ref{sec:data1-5}.}
\end{remark}

\begin{remark}
\label{rem:K3discr}
{\rm K3 surfaces $X$
with transcendental lattices $T_X=T(-1)$ for
lattices $T$ of signature $(n,2)$
with strongly $2$-reflective modular forms $\Phi$
constructed above, have discriminants which are
divisors of $\Phi$, and these divisors are rational
quadratic divisors orthogonal to all $(-2)$-roots of $T_X=T(-1)$
and of multiplicity one.

These K3 surfaces can be considered as mirror
symmetric (by the arithmetic mirror symmetry defined
by the corresponding Lorentzian Kac--Moody algebras)
to K3 surfaces with Picard lattices $S_X$
from Remark \ref{rem:K3finautgroup+rho}.
See \cite{GN3}, \cite{GN7} and \cite{GN8} about
the arithmetic mirror symmetry.}
\end{remark}

\begin{remark}
\label{rem:finrefmodforms}
 {\rm We expect some finiteness results about the
set of reflective modular forms. The main reason is the K\"ocher principle.
See \cite{Nik9} about first observations. Recently, Sh. Ma (see \cite{Ma}) obtained some
finiteness results about 2-reflective modular forms.}
\end{remark}


Valery Gritsenko \par
\smallskip

Laboratoire Paul Painlev\'e et CEMPI\par
Universit\'e de Lille 1, France \par
Valery.Gritsenko@math.univ-lille1.fr \par
\smallskip

{\it Current Address}:\par
National Research University Higher School of Economics
\par
AG Laboratory, 7 Vavilova str., Moscow, Russia, 117312
\vskip0.5cm

\vskip5pt

Viacheslav V. Nikulin

\par Steklov Mathematical Institute,\par
ul. Gubkina 8, Moscow 117966, GSP-1,
Russia \par
Deptm. of Pure Mathem. The University of
Liverpool, \par
Liverpool\ \ L69\ 3BX,\ UK  \par
nikulin@mi.ras.ru\ \ vvnikulin@list.ru\ \ vnikulin@liv.ac.uk

\end{document}